\documentclass[a4paper]{amsart}
\usepackage{hyperref, enumitem}
\usepackage[dvipsnames]{xcolor} 
\usepackage{amsfonts,amsmath, amssymb,mathrsfs,mathtools,amsthm} 
\usepackage[doi=false, url=false, eprint=false, maxbibnames=10, firstinits=true]{biblatex}
\usepackage{todonotes}

\usepackage[margin=0pt]{caption}
\newtheorem{thm}{Theorem}[section]
\newtheorem{prop}[thm]{Proposition}
\newtheorem{define}[thm]{Definition}
\newtheorem{lem}[thm]{Lemma}
\newtheorem{cor}[thm]{Corollary}
\newtheorem{rem}[thm]{Remark}
\newtheorem{ex}[thm]{Example}
\numberwithin{equation}{section}

\newcommand{\A}{\mathcal{A}}

\newcommand{\Z}{\mathbb{Z}}
\newcommand{\R}{\mathbb{R}}
\newcommand{\C}{\mathbb{C}}
\newcommand{\B}{\mathcal{B}}
\newcommand{\F}{\mathscr{F}}

\newcommand{\Crit}{\operatorname{Crit}}

\newcommand{\supp}{\operatorname{supp}}
\newcommand{\Exp}{\operatorname{Exp}}
\newcommand{\Id}{\operatorname{Id}}
\newcommand{\reg}{\operatorname{reg}}
\usepackage{verbatim}
\newcommand{\ol}{\overline}
\newcommand{\wt}{\widetilde}

\newcommand{\om}{\omega}
\newcommand{\p}{\partial}
\newcommand{\one}
{{{\mathchoice \mathrm{ 1\mskip-4mu l} \mathrm{ 1\mskip-4mu l}
\mathrm{ 1\mskip-4.5mu l} \mathrm{ 1\mskip-5mu l}}}}
\newcommand{\LRFH}{{\rm LRFH}}

\begin{document}
\title{The two-boost problem and Lagrangian Rabinowitz Floer homology}

\author{Kai Cieliebak}
\address{Department of Mathematics, University of Augsburg}
\thanks{Kai Cieliebak is funded by the Deutsche Forschungsgemeinschaft (DFG, German
Research Foundation) – 541525489.}
\author{Urs Frauenfelder}
\address{Department of Mathematics, University of Augsburg}
\thanks{ Urs Frauenfelder is funded by the Deutsche Forschungsgemeinschaft (DFG, German
Research Foundation) – 541525489.}
\author{Eva Miranda} 
\address{Laboratory of Geometry and Dynamical Systems, Department of Mathematics, Universitat Politècnica de Catalunya-IMTech \& CRM}
\thanks{Eva Miranda is funded by the Catalan Institution for Research and Advanced Studies via an ICREA Academia Prize 2021 and by the Alexander Von Humboldt Foundation via a Friedrich Wilhelm Bessel Research Award. She is also supported by the Spanish State
Research Agency, through the Severo Ochoa and Mar\'{\i}a de Maeztu Program for Centers and Units
of Excellence in R\&D (project CEX2020-001084-M), by the AGAUR project 2021 SGR 00603 and by the project “Computational, dynamical and geometrical complexity in fluid dynamics”, Ayudas Fundación BBVA a Proyectos de Investigación Científica 2021. Eva Miranda and Jagna Wi\'sniewska are partially supported by the Spanish State Research Agency grants reference PID2019-103849GB-I00 of AEI / 10.13039/501100011033 and PID2023-146936NB-I00 funded by MICIU/AEI/
10.13039/501100011033 and, by ERDF/EU. }
\author{Jagna Wi\'sniewska}
\address{Laboratory of Geometry and Dynamical Systems, Department of Mathematics, Universitat Politècnica de Catalunya}\thanks{Jagna Wi\'sniewska's postdoctoral contract is financed under the project “Computational, dynamical and geometrical complexity in fluid dynamics”, Ayudas Fundación BBVA a Proyectos de Investigación Científica 2021.}

\maketitle

\begin{abstract}
\noindent
The two-boost problem in space mission design asks whether two points of phase space can be connected with the help of two boosts of given energy. We provide a positive answer for a class of systems related to the restricted three-body problem by defining and computing its Lagrangian Rabinowitz Floer homology. The main technical work goes into dealing with the noncompactness of the corresponding energy hypersurfaces. 
\end{abstract}
  
\section{Introduction}

The {\em two-boost problem} goes back to the classical work of W.\,Hohmann on the attainability of heavenly bodies \cite{Hohmann}. Although this work was written more than four decades before the first human being stepped on the moon, the Hohmann transfer is still one of the crucial ingredients in space mission design, see e.g.~\cite{Vallado}. The Hohmann transfer
is a transfer between two circular orbits in the Kepler problem with the help of a Kepler ellipse which is tangent to the two circles. It requires two tangential boosts, one to transfer from the first circle to the ellipse and a second one to transfer from the ellipse to the second circle.
Given two points in the plane different from the origin, there always exists a conic section through the two points with focus in the origin. This means that, for the Kepler problem, two points in phase space can always be connected with the help of two boosts. The motivating question for this paper is whether this continues to hold for more general systems.

The general setup for the two-boost problem is as follows: Consider the cotangent bundle $T^*Q$ of a manifold $Q$ with its canonical exact symplectic form $\om=d\lambda$, $\lambda=p\,dq$, and a Hamiltonian $H:T^*Q\to\R$. Given two points $q_0,q_1\in Q$ and an energy value $c$, the two-boost problem asks for the existence of a Hamiltonian orbit of energy $c$ connecting the cotangent fibres $T_{q_0}^*Q$ and $T_{q_1}^*Q$. 
Such orbits arise as critical points of the {\em Rabinowitz action functional}
\begin{align*}
\A^{H-c}_{q_0,q_1} & : \mathscr{H}_{q_0,q_1}\times \R \to \R,\\
\A^{H-c}_{q_0,q_1}(v,\eta) & := \int_0^1 \lambda(\partial_t v)dt - \eta \int_0^1 (H-c)(v(t))dt
\end{align*}
associated to the path space
$$
\mathscr{H}_{q_0,q_1}:= \left\lbrace v \in W^{1,2}([0,1], T^*Q)\ \big|\ v(i)\in T^*_{q_i}Q\quad \textrm{for}\quad i=0,1\right\rbrace.
$$
Thus the two-boost problem has a positive answer whenever the corresponding {\em Lagrangian Rabinowitz Floer homology} $\LRFH_*(\A^{H_0-h}_{q_0,q_1})$ is well-defined and nontrivial. 

As observed in~\cite{CieliebakFrauenfelder2009}, the well-definedness of (Lagrangian) Rabinowitz Floer homology requires some hypothesis on the Hamiltonian $H$. A suitable class is formed by {\em magnetic Hamiltonians} 
$$
  H(q,p) = \frac12|p-A(q)|^2 - V(q)
$$
for a magnetic potential $A\in\Omega^1(Q)$, a potential\footnote{
As is customary in celestial mechanics, our $V$ is {\em minus} the physical potential.}
$V:Q\to\R$, and a Riemannian metric on $Q$. Under the assumption that the underlying manifold $Q$ is {\em compact}, the Rabinowitz Floer homology of such Hamiltonians (with possibly noncompact magnetic field) has been studied in~\cite{CieliebakFrauenfelderPaternain2010} for the periodic case and in~\cite{Merry2014} for the Lagrangian case. An important dynamical quantity associated to such a Hamiltonian is its {\em Ma\~n\'e critical value}. By a Theorem of W.\,Merry~\cite{Merry2014}, for $c$ above the Ma\~n\'e critical value the Lagrangian Rabinowitz Floer homology $\LRFH_*(\A^{H-c}_{q_0,q_1})$ is well-defined and nontrivial, hence the two-boost problem is solvable. 

Returning to celestial mechanics, let us consider the {\em planar circular restricted 3-body problem}. Here two heavy bodies (the primaries) move under their mutual gravitational attraction on circles around their center of mass in a plane, and the third body (of negligible mass) moves in the same plane under the gravitational attraction by the primaries. In rotating coordinates, this system is described by the autonomous Hamiltonian
\begin{align*}
H(q_1,q_2,p_1,p_2) &=\frac{1}{2}(p_1^2+p_2^2)+ p_1q_2-p_2q_1 - V(q_1,q_2), 
\end{align*}
where $V$ is the sum of the (negative) Coulomb potentials of the two primaries (see e.g.~\cite{FrauenfeldervanKoert2018}). Writing it as
\begin{align*}
H(q_1,q_2,p_1,p_2) 
&= \frac{1}{2}(|p_1+q_2|^2+|p_2-q_1|^2)-V_{\rm eff}(q_1,q_2)
\end{align*}
with the effective potential $V_{\rm eff}(q_1,q_2)=V(q_1,q_2)+\frac12(q_1^2+q_2^2)$, we see that this is a magnetic Hamiltonian, with magnetic field generating the Coriolis force. 
In contrast to the previous paragraph, the configuration space $\R^2\setminus\{q_E,q_M\}$ is now noncompact due to possible collisions with the primaries at positions $q_E,q_M$ and possible escapes to infinity. 
Our goal is to define and compute Lagrangian Rabinowitz Floer homology in this situation. 
Since the noncompactness due to collisions can be removed by Moser or Levi-Civita regularization (see~\cite{FrauenfeldervanKoert2018}), we will focus our attention on the noncompactness at infinity in $\R^2$. For this, we consider the quadratic Hamiltonian
\begin{equation}\label{DefH0}
\begin{aligned}
H_0(q_1,q_2,p_1,p_2) & :=\frac{1}{2}(p_1^2+p_2^2)+ p_1q_2-p_2q_1 = \frac{p_r^2}{2}+\frac{p_\theta^2}{2r^2}+p_\theta,
\end{aligned}
\end{equation} 
where $(r,\theta)$ are polar coordinates on the plane and $(p_r,p_\theta)$ their conjugate momenta. 
We introduce the following set of smooth functions constant outside a compact set:
\begin{equation}\label{DefHset}
\mathcal{H}:=\left\lbrace h\in C^\infty(T^*\R^2)\ \Big|\ dh\in C_c^\infty(T^*\R^2),\quad h>0,\quad h-dh(p\partial_p)> 0\right\rbrace.
\end{equation}

For $q_0,q_1\in \R^2$ we define the Rabinowitz action functional $\A^{H_0-h}_{q_0,q_1} : \mathscr{H}_{q_0,q_1}\times \R \to \R$ as above, with $Q=\R^2$ and $H-c$ replaced by $H_0-h$. Our first result is

\begin{thm}\label{thm:LRFH}
Fix $q_0,q_1\in \R^2$ and let $H_0$ be as in \eqref{DefH0}. For every $h \in \mathcal{H}$ the Lagrangian Rabinowitz Floer homology of $\A^{H_0-h}_{q_0,q_1}$ is well defined.
Moreover, for any two $h_0,h_1\in \mathcal{H}$ the Lagrangian Rabinowitz Floer homology of $\A^{H_0-h_0}_{q_0,q_1}$ and $\A^{H_0-h_1}_{q_0,q_1}$ are isomorphic.
\end{thm}

We denote the Lagrangian Rabinowitz Floer homology of $\A^{H_0-h}_{q_0,q_1}$ by $\LRFH_*(\A^{H_0-h}_{q_0,q_1})$. Here we use coefficients in $\Z_2=\Z/2\Z$ and a half-integer grading by Maslov indices (see~\S\ref{ss:grading}). 
Moreover, it comes with a canonical action filtration. Our next result computes the positive action part in the case $q_0\neq q_1$. 

\begin{thm}\label{thm:posLRFH}
For the Hamiltonian $H_0$ in \eqref{DefH0}, $h \in \mathcal{H}$, and any $q_0,q_1\in \R^2$ with $q_0\neq q_1$, the positive Lagrangian Rabinowitz homology of $\A^{H_0-h}_{q_0,q_1}$ is well defined
and equal to
$$
\LRFH_*^+\left(\A^{H_0-h}_{q_0,q_1}\right)=
\begin{cases}
\mathbb{Z}_2 & \textrm{for}\quad *=1/2,\\
0 & \textrm{otherwise}.
\end{cases}
$$ 
\end{thm}

This shows that the two-boost problem at energy $0$ is solvable for $H_0-h$ with $h\in\mathcal{H}$. Note that $h(q,p)=c+V(q)$ belongs to $\mathcal{H}$ for each constant $c>0$ and compactly supported potential $V\geq 0$, so the two-boost problem is solvable for $H_0-V$ with such $V$ at each energy $c>0$.

The potential in the planar circular restricted 3-body problem (after regularisation) satisfies $V\geq 0$, but it is not compactly supported. The following result allows us to deal with certain non-compactly supported potentials. 

\begin{prop}\label{prop:CritH=CritH1}
Let $H_0$ be the Hamiltonian in \eqref{DefH0}. Fix $q_0,q_1\in \R^2$ with $q_0\neq q_1$ and a positive constant $c>0$.
Let $V: \R^2 \to \R$ be a nonnegative potential for which there exist $r_0>0$ and
$(\alpha, a) \in \{2\}\times  \left(0 ,\frac{c^2}{4}\right)\cup (2, +\infty) \times \R_+$,
such that for $r>r_0$ we have 
 $$
 V(r,\theta)\leq \frac{a}{r^\alpha}\qquad \textrm{and} \qquad \partial_r V (r, \theta) \geq -\frac{a}{r^{\alpha+1}}.
 $$
Then there exists a function $V_0\in \mathcal{H}$ such that $\Crit \A^{H_0-V-c}_{q_0,q_1} = \Crit \A^{H_0-V_0}_{q_0,q_1}$.
\end{prop}

In view of Theorem~\ref{thm:posLRFH}, the two-boost problem is therefore solvable for $H_0-V$ at energy $c>0$ for any $(V,c)$ as in Proposition~\ref{prop:CritH=CritH1}. In particular, this holds for any $c>0$ and $V\geq 0$ which decays at infinity as $r^{-\alpha}$ with $\alpha>2$. Note that this does {\em not} cover the planar circular restricted 3-body problem, in which the potential decays as $r^{-1}$. We defer the extension to this case to future work.

\begin{rem}\label{rem:contact}
Note that for $h \in \mathcal{H}$ the Liouville vector field $p\partial_p$ is transverse to $(H_0-h)^{-1}(0)$, so the zero level set of $H_0-h$ is of restricted contact type. Thus the solutions of the two-boost problem resulting from Theorem~\ref{thm:posLRFH} and Proposition~\ref{prop:CritH=CritH1} can also be interpreted as solutions of Arnold's chord conjecture (existence of a Reeb chord connecting two Legendrians) for certain noncompact contact manifolds.  
\end{rem}

%

\begin{rem}
Work by R.\,Nicholls~\cite{Nicholls2021} provides some evidence that, if defined, the Ma\~n\'e critical value for the restricted three-body problem should be zero. It would be interesting to make this rigorous by developing a notion of Ma\~n\'e critical value for magnetic Hamiltonians over noncompact base manifolds. 
\end{rem}

\section{Lagrangian Rabinowitz Floer homology}

Lagrangian Rabinowitz Floer homology was first defined by Merry in \cite{Merry2014} for virtually exact Lagrangian submanifolds and compact hypersurfaces of virtual contact type in symplectically aspherical symplectic manifolds. In his setting, the pull-back of the symplectic form to the universal cover of the symplectic manifold had to be exact and the pre-image of the chosen hypersurface in the universal cover had to be of contact type. The symplectic manifold and the Lagrangian submanifold could be non-compact, but the chosen hypersurface of virtual contact type had to be compact.

Our aim is to use Lagrangian Rabinowitz Floer homology to analyse the Planar Circular Restricted Three Body Problem. Therefore, we are interested in the setting of a cotangent bundle $T^*Q$ with its standard symplectic form and a pair of fibres $T_{q_0}^*Q, T_{q_1}^*Q$ as the Lagrangians. In this section we will explain the definition of the Lagrangian Rabinowitz Floer homology in this setting. It is a far easier setting than the one in \cite{Merry2014}, as the cotangent bundle is an exact manifold and the fibres are exact Lagrangians. Moreover, the level sets above the 4th and 5th Lagrange point of the Hamiltonian corresponding to the Planar Circular Restricted Three Body Problem are all of exact contact type (see Remark~\ref{rem:contact}).
The only challenge of this setting is that all the energy level sets in the Restricted Three Body Problem are non-compact.
Rabinowitz Floer homology (for periodic orbits) of noncompact energy level sets of ``tentacular Hamiltonians'' on $\R^{2n}$ has been defined in~\cite{Pasquotto2017}. However, the Hamiltonian for the Planar Circular Restricted Three Body Problem does not belong to this class, so new arguments are needed for this setting.

\subsection{The Lagrangian Rabinowitz action functional}

Throughout this section, $Q$ denotes a smooth oriented $n$-dimensional manifold, and $T^*Q$ its cotangent bundle equipped with the exact symplectic form $\omega=d\lambda$ for the canonical $1$-form $\lambda:=p\,dq$. 
Fix $q_0,q_1 \in Q$ and define the space of paths
$$
\mathscr{H}_{q_0,q_1}:= \left\lbrace v \in W^{1,2}([0,1], T^*Q)\ \big|\ v(i)\in T^*_{q_i}Q\quad \textrm{for}\quad i=0,1\right\rbrace.
$$
Consider a Hamiltonian $H: T^*Q \to \R$ with regular level set $H^{-1}(0)$. 
The {\em Lagrangian Rabinowitz action functional} $\A^H_{q_0,q_1}:\mathscr{H}_{q_0,q_1}\times \R\to\R$ associated to $H$ is defined as
$$
\A^H_{q_0,q_1}(v,\eta):= \int_0^1 \lambda(\partial_t v)dt - \eta \int_0^1 H(v(t))dt.
$$

\begin{rem}
We have a natural bijection $\mathscr{H}_{q_0,q_1}\ni v \mapsto \ol v\in \mathscr{H}_{q_1,q_0}$ with $\ol v(t):=v(1-t)$. Then
$$
\A^H_{q_0,q_1}(v, \eta) = -\A^H_{q_1,q_0}(\ol v, -\eta)\qquad\forall\ (v, \eta) \in \mathscr{H}_{q_0,q_1}\times\R.
$$
\end{rem}

The derivative of $\A^H_{q_0,q_1}$ in direction $(\xi, \sigma) \in T_v \mathscr{H}_{q_0,q_1}\times T_\eta \R$ equals
\begin{equation}\label{eq:dAH}
d\A^H_{q_0,q_1}(v,\eta)[\xi,\sigma] = \int_0^1 \omega (\xi,\partial_tv -\eta X_H) - \sigma\int_0^1 H(v)dt.
\end{equation}
Here $X_H$ is the Hamiltonian vector field defined by $dH=-i_{X_H}\om$. 
Consequently, $(v,\eta) \in \Crit (\A^H_{q_0,q_1})$ if and only if it satisfies
$$
\partial_t v= \eta X_H(v)\quad \textrm{and}\quad v(t)\in H^{-1}(0)\quad \forall\ t\in [0,1].
$$
Thus we can have three types of critical points: 
\begin{itemize}
\item $\eta>0$ and $\wt v(t):=v(t/\eta)$ is a Hamiltonian chord (i.e., an integral curve of $X_H$) on $H^{-1}(0)$ from $T^*_{q_0}Q$ to $T_{q_1}^*Q$;
\item $\eta<0$ and $\wt v(t):=v(t/\eta)$ is a Hamiltonian chord on $H^{-1}(0)$ from $T^*_{q_1}Q$ to $T_{q_0}^*Q$; 
\item $\eta=0$ and $v$ is a constant path in $T^*_{q_0}Q\cap T_{q_1}^*Q\cap H^{-1}(0)$ (which can only occur if $q_0=q_1$). 
\end{itemize}
In particular, if $T^*_{q_0}Q \cap H^{-1}(0)=\emptyset$ or $T_{q_1}^*Q \cap H^{-1}(0)=\emptyset$, then $ \Crit (\A^H_{q_0,q_1})=\emptyset$. Therefore, from now on we will assume that 
\begin{equation}\label{T*QcapHnonempty}
T^*_{q_0}Q \cap H^{-1}(0)\neq\emptyset,\quad T^*_{q_1}Q \cap H^{-1}(0)\neq\emptyset,\quad\textrm{and } T^*_{q_0}Q\pitchfork H^{-1}(0) \text{ if }q_0=q_1.
\end{equation}
In order to construct Lagrangian Rabinowitz Floer homology we want the critical set of the action functional to be bounded in $L^\infty$. It is possible to define Lagrangian Rabinowitz Floer homology without this assumption, but the construction is much more challenging, so we postpone it for future projects. In fact, to construct Lagrangian Rabinowitz Floer homology we will need the boundedness of the critical set of the action functional to persist under compact perturbations of the Hamiltonian. We formalize this in the following definition. 

\begin{define}
Consider a Hamiltonian $H: T^*Q \to \R$, such that $H^{-1}(0)$ is noncompact. Fix $q_0,q_1\in T^*Q$ satisfying \eqref{T*QcapHnonempty}. Let $K\subseteq T^*Q$ be a compact set and let 
$$
\mathcal{H}\subseteq \{ h\in C^\infty(T^*Q)\ |\ dh\in C_0^\infty(K)\},
$$
be an open neighbourhood of $0$ in $\{ h\in C^\infty(T^*Q)\ |\ dh\in C_0^\infty(K)\}$.
We say that the associated Lagrangian Rabinowitz Floer functional $\A^H_{q_0,q_1}:\mathscr{H}_{q_0,q_1}\times \R\to \R$ has critical set \emph{continuously compact} in $(K,\mathcal{H})$ if
$$
\forall\ h\in\mathcal{H}\quad \textrm{and}\quad\forall\ (v,\eta) \in \Crit \A^{H+h}_{q_0,q_1} \quad \textrm{we have}\quad v([0,1])\subseteq K.
$$
\end{define}

\subsection{Grading}\label{ss:grading}

In~\cite{CieliebakFrauenfelder2009}, Rabinowitz Floer homology (for periodic orbits) is equipped with a shifted integer grading under the hypothesis that the ambient symplectic manifold has vanishing first Chern class. In this subsection we indicate how this carries over to the Lagrangian setting; see~\cite{Merry2014} for details. 

Consider an exact symplectic manifold $(V,\om=d\lambda)$ of dimension $2n$ and a Hamiltonian $H:V\to\R$ with regular level set $\Sigma=H^{-1}(0)$ such that $\xi=\ker(\lambda|_\Sigma)$ is a contact structure. Consider in addition two Lagrangian submanifolds $L_0,L_1\subset V$ transverse to $\Sigma$ with $\lambda|_{L_i}=0$. We assume that $L_0$ and $L_1$ are either disjoint or equal. 
The condition of vanishing first Chern class in~\cite{CieliebakFrauenfelder2009} gets replaced by vanishing of the relative first Chern class $c_1(V,L_0\cup L_1)$. The transversal Conley-Zehnder index in~\cite{CieliebakFrauenfelder2009} gets replaced by the {\em transverse Maslov index} defined a follows.
Let $(v,\eta)$ be a critical point of the Rabinowitz action functional $\A^H$ on the space of paths from $L_0$ to $L_1$. Pick a symplectic trivialization $v^*\xi\cong [0,1]\times\R^{2n-2}$ sending $T_{v(i)}L_i\cap\xi_{v(i)}$ to $\R^{n-1}$ for $i=0,1$. In this trivialization, the image of $T_{v(i)}L_i\cap\xi_{v(i)}$ under the linearized flow of $\eta X_H$ gives a path $\Lambda:[0,1]\to\mathcal{L}_{n-1}$ in the Grassmannian $\mathcal{L}_{n-1}$ of Lagrangian subspaces of $(\R^{2n-2},\om_{\rm std})$ with $\Lambda(0)=\R^{n-1}$. Let $\mu^{\rm tr}(v,\eta)\in\frac12\Z$ be the Maslov index of the path $\Lambda$ defined by Robbin and Salamon in~\cite{Robbin1993}. By vanishing of $c_1(V,L_0\cup L_1)$ this is independent of the choices. In the case $\eta=0$ (which occurs only if $L_0=L_1$), $v$ is a critical point of an auxiliary Morse function $f$ on $L_0\cap\Sigma$. Its {\em signature index} is ${\rm ind}^\sigma_f(v,0) := -\frac12{\rm sign}\,{\rm Hess}_f(v)$, and we set it to zero if $\eta\neq 0$. According to~\cite[Appendix A]{CieliebakFrauenfelder2009}, the {\em Maslov index}
$$
  \mu(v,\eta) := \mu^{\rm tr}(v,\eta) + {\rm ind}^\sigma_f(v,\eta) \in\frac12\Z
$$
defines a half-integer grading on the Rabinowitz Floer complex with respect to which the differential has degreee $-1$.
  
This discussion applies to the setting in this section where $V=T^*Q$ and $L_i=T_{q_i}^*Q$ for an oriented $n$-dimensional manifold $Q$. 

\begin{ex}\label{ex:Maslov}
Consider the free Hamiltonian $H_{\bullet}(q,p):=\frac{1}{2}|p|^2$ on $T^*\R^n$ and fix an energy $c>0$. Assume first that $q_0\neq q_1\in\R^n$. Then $\A^{H_{\bullet}-c}_{q_0,q_1}$ has exactly two critical points $(v_\pm,\eta_\pm)$ given by 
$$
  v_\pm(t)=\Bigl((1-t)q_0+tq_1,\frac{q_1-q_0}{\eta_\pm}\Bigr),\qquad \eta_\pm=\pm\frac{\sqrt{2c}}{|q_1-q_0|}.
$$
The linearized flow of $\eta_\pm X_{H_\bullet}$ along $v_\pm$ equals $\Phi_t=\begin{pmatrix}\one & t\eta_\pm\one \\ 0 & \one\end{pmatrix}$, so the corresponding path of Lagrangian subspaces $\Lambda:[0,1]\to\mathcal{L}_{n-1}$ is given by $\Lambda(t)=\Phi_t(\R^{n-1})={\rm graph}\,A(t)$ with $A(t)=t\eta_\pm\one$. Thus the (Localization) property in~\cite{Robbin1993} yields 
$$
  \mu(v_\pm,\eta_\pm) = \mu^{\rm tr}(v_\pm,\eta_\pm) = \frac12{\rm sign}\,A(1) - \frac12{\rm sign}\,A(0) = \pm\frac{n-1}{2}.
$$
In the case $q_0=q_1$, let $f$ be a Morse function on the $(n-1)$-sphere $T_{q_0}Q\cap H_\bullet^{-1}(c)$ with exaxctly two critical points, the maximum $v_+$ and the minimum $v_-$. Then $\A^{H_{\bullet}-c}_{q_0,q_0}$ has exactly two critical points $(v_\pm,0)$ of index
$$
  \mu(v_\pm,0) = {\rm ind}^\sigma_f(v_\pm,0) = \pm\frac{n-1}{2}.
$$
For $q_0\neq q_1$ this immediately implies
$$
\LRFH_*^\pm\left(H_\bullet^{-1}(c), T^*_{q_0}\R^n,T^*_{q_1}\R^n\right)=
\begin{cases}
\mathbb{Z}_2 & \textrm{for}\quad *=\pm\frac{n-1}{2},\\
0 & \textrm{otherwise}.
\end{cases}
$$
Moreover, for $n\neq 3$ and any $q_0,q_1$ (equal or not) we obtain
$$
\LRFH_*\left(H_\bullet^{-1}(c), T^*_{q_0}\R^n,T^*_{q_1}\R^n\right)=
\begin{cases}
\mathbb{Z}_2 & \textrm{for}\quad *=\frac{n-1}{2}\textrm{ and }*=-\frac{n-1}{2},\\
0 & \textrm{otherwise}.
\end{cases}
$$
(This also holds for $n=2$, but an additional argument is needed to show that the two critical points of index $\pm1/2$ do not cancel in homology).
\end{ex}

\begin{rem}\label{rem:Maslov}
The previous example generalizes to the free Hamiltonian $H_{\bullet}(q,p)=\frac{1}{2}|p|^2$ on the cotangent bundle of an $n$-dimensional Riemannian manifold $(Q,g)$. For $q_0\neq q_1$ and $(v,\eta)\in \Crit\A^{H_{\bullet}-c}_{q_0,q_1}$, the path $v$ projects onto a geodesic $\bar v$ from $q_0$ to $q_1$ and Proposition 6.38 in~\cite{RobbinSalamon1995} yields
\begin{equation*}
  \mu(v,\eta) = {\rm ind}(\bar v) + {\rm sign}(\eta)\frac{n-1}{2},
\end{equation*}
where ${\rm ind}(\bar v)$ is the Morse index of the geodesic $\bar v$ (i.e., the number of conjugate points along $\bar v$). 
\end{rem}

\subsection{Floer trajectories}\label{sec:FloerTrajec}

In this subsection we introduce the fundamental notion for constructing any Floer-type homology - the Floer trajectories. We start by equipping the space $\mathscr{H}_{q_0,q_1}$ with a metric. An \emph{almost complex structure} $J$ on a manifold $M$ is a bundle endomorphism $J: TM \to TM$ satisfying $J^2 = -\operatorname{Id}$. An almost complex structure $J$ on a symplectic manifold $(M, \omega)$ is called \emph{compatible} if $\omega(\ \cdot\ ,J \cdot\ )$ defines a Riemannian metric on $M$. Denote by $\mathcal{J}(M,\omega)$ the space of all compatible almost complex structures on $(M,\omega)$ with the $C^\infty$-topology.
A linear algebra argument \cite[Prop. 13.1]{Silva2001} shows that $\mathcal{J}(M,\omega)$ is contractible.

A smooth $2$-parameter family $\{J_{t,\eta}\}_{(t,\eta)\in[0,1]\times \R}$ of compatible almost complex structures on $(T^*Q, \omega)$ defines an $L^2$-metric on $\mathscr{H}_{q_0,q_1}\times \R$ by
$$
\langle (\xi_1, \sigma_1), (\xi_2, \sigma_2)\rangle := \int_0^1 \omega(\xi_1(t),J_{t,\eta}(v(t))\xi_2(t) )+\sigma_1\sigma_2.
$$
for $(\xi_i,\sigma_i)\in T_{(v,\eta)}(\mathscr{H}_{q_0,q_1}\times \R)$.
The gradient of the Lagrangian Rabinowitz action functional $\A^H_{q_0,q_1}$ with respect to this metric equals
$$
\nabla \A^H_{q_0,q_1} (v,\eta)= \left( \begin{array}{c}
-J_{t,\eta}(v(t)) (\partial_tv -\eta X_H) \\ - \int H(v)dt.
\end{array}
\right)
$$
Fix an open subset $\mathcal{V}\subseteq T^*Q$ and $\mathbb{J}\in \mathcal{J}(T^*Q,\omega)$.
We denote by $\mathcal{J}(\mathcal{V}, \mathbb{J})$ the set of all smooth maps
$$
[0,1]\times \mathbb{R}\to \mathcal{J}(T^*Q,\omega),\qquad (t,\eta)\mapsto J_{t,\eta}
$$
satisfying
\begin{align*}
J_{t,\eta}(x) = \mathbb{J}(x)\text{ for } x\notin \mathcal{V}
\quad\textrm{and}\quad \sup_{(t,\eta)\in [0,1]\times \mathbb{R}}\|J_{t,\eta}\|_{C^k}<+\infty \text{ for all } k\in \mathbb{N}.
\end{align*}
Fix Hamiltonians $H_\pm$ and $J_\pm\in\mathcal{J}(\mathcal{V},\mathbb{J})$. A {\em homotopy} from $(H_-,J_-)$ to $(H_+,J_+)$ is a smooth family $\Gamma = \{(H_s, J_s)\}_{s\in \R}$ of Hamiltonians and compatible almost complex structures $J_s\in \mathcal{J}(\mathcal{V},\mathbb{J})$ which agrees with $(H_-,J_-)$ for $s\leq s_-$ and with $(H_+,J_+)$ for $s\geq s_+$, with some $s_\pm\in\R$. 
A solution $u: \R \to \mathscr{H}_{q_0,q_1} \times \R$ to the gradient flow equation $\partial_s u = \nabla \A^{H_s}_{q_0,q_1}(u)$ is called a \emph{Floer trajectory}. In other words, a Floer trajectory $u=(v,\eta) \in W^{1,2}(\R\times[0,1], T^*Q)\times W^{1,2}(\R)$ is a solution to the equations
\begin{equation}\label{FloerEq}
\begin{aligned}
\partial_s v(s,t) & = - J_{s, t,\eta}(v(s,t)) (\partial_sv(s,t)-\eta(s)X_{H_s}(v(s,t))),\\
\partial_s \eta (s) & = - \int_0^1 H_s\circ v(s,t)dt,
\end{aligned}
\end{equation}
$$
v(s,0) \in T^*_{q_0}Q \quad \textrm{and} \quad v(s,1) \in T^*_{q_1}Q \quad \forall\ s\in \R.
$$
For $(x_-, x_+) \in \Crit(\A^{H_-}_{q_0,q_1})\times \Crit \A^{H_+}_{q_0,q_1}$ we denote the space of Floer trajectories from $x_-$ to $x_+$ by
$$
\F_\Gamma(x_-, x_+):=
\left\lbrace\begin{array}{c|c}
& \partial_s u = \nabla \A^{H_s}_{q_0,q_1}(u),\\
\hspace*{-0.2cm}\smash{\raisebox{.5\normalbaselineskip}{ $u: \R \to \mathscr{H}_{q_0,q_1} \times \R$}}& \lim_{s\to \pm \infty}u(s)=x_\pm.
\end{array}\right\rbrace
$$
In case the homotopy $\Gamma$ is constant in $s$, i.e. $H_s\equiv H$ and $J_{s,t,\eta}\equiv J \in \mathcal{J}(\mathcal{V},\mathbb{J})$ for some Hamiltonian $H$ and a compatible almost complex structure $J$, we denote $\F_{H,J}(x_-,x_+):=\F_\Gamma(x_-, x_+)$.

Moreover, for every pair $(a,b)\in \R^2$ we denote 
\begin{equation}\label{DefM(a,b)}
\mathcal{M}^\Gamma(a,b):=\left\lbrace\begin{array}{c|c}
& (x_-,x_+)\in \Crit\A^{H_-}_{q_0,q_1}\times \Crit\A^{H_+}_{q_0,q_1},\\
\hspace*{-0.2cm}\smash{\raisebox{.5\normalbaselineskip}{ $u\in \F_\Gamma(x_-,x_+)$}}& \A^{H_-}_{q_0,q_1}(x_-)\geq a, \quad \A^{H_+}_{q_0,q_1}(x_+)\leq b.
\end{array}\right\rbrace
\end{equation}
Analogously, we denote $\mathcal{M}^{H,J}(a,b):=\mathcal{M}^\Gamma(a,b)$ whenever the homotopy $\Gamma$ is constant in $s$ and equal to the pair $(H,J)$.

If s homotopy $\Gamma$ is constant in $s$, then the action is increasing along Floer trajectories. However, for a nonconstant homotopy $\Gamma$ this need not be the case. To deal with this phenomenon, we introduce a condition that ensures that the action cannot decrease indefinitely along a Floer trajectory.
We say a that homotopy $\Gamma:= \{(H_s, J_s)\}_{s\in \R}$ with $H_{\pm}:=\lim_{s\to \pm \infty}H_s$ satisfies the \emph{Novikov finiteness condition} if for all $a,b\in \R$ we have
\begin{equation}\label{Novikov}
\begin{aligned}
\inf\left\lbrace\begin{array}{c|c}
& y \in  \Crit\A^{H_+}_{q_0,q_1}, \quad \exists\ x\in \Crit  \A^{H_-}_{q_0,q_1},\\
\hspace*{-0.2cm}\smash{\raisebox{.5\normalbaselineskip}{ $\A^{H_+}_{q_0,q_1}(y)$}}& \A^{H_-}_{q_0,q_1}(x)\geq a, \quad \F_\Gamma(x,y)\neq \emptyset.
\end{array}\right\rbrace & > -\infty,\\
\sup\left\lbrace\begin{array}{c|c}
& x \in  \Crit\A^{H_-}_{q_0,q_1}, \quad \exists\ y\in \Crit  \A^{H_+}_{q_0,q_1},\\
\hspace*{-0.2cm}\smash{\raisebox{.5\normalbaselineskip}{ $\A^{H_-}_{q_0,q_1}(x)$}}& \A^{H_+}_{q_0,q_1}(x)\leq b, \quad \F_\Gamma(x,y)\neq \emptyset.
\end{array}\right\rbrace &< +\infty.
\end{aligned}
\end{equation}

\subsection{Defining Lagrangian Rabinowitz Floer homology}

In this subsection we will recall the construction of the Lagrangian Rabinowitz Floer homology from \cite{Merry2014}, using standard Floer techniques introduced in Floer's seminal paper \cite{Floer1989} and the techniques typical for the setting of the Rabinowitz action functional from the first two author's paper \cite{CieliebakFrauenfelder2009}. More precisely, we will prove the following theorem.

\begin{thm}\label{thm:DefLRFH}
Consider a cotangent bundle $(T^*Q,\omega)$ with its standard symplectic form and a Hamiltonian $H:T^*Q\to \R$ with regular level set $H^{-1}(0)$.
Fix a pair $q_0,q_1\in Q$ such that both sets $H^{-1}(0)\cap T^*_{q_0}Q$ and $H^{-1}(0)\cap T^*_{q_1}Q$ are compact and nonempty. Fix a compatible almost complex structure $\mathbb{J}\in \mathcal{J}(T^*Q,\omega)$.
Assume that there exists
a compact subset $K$, an open subset $\mathcal{V}$ satisfying $K\subseteq \mathcal{V}\subseteq T^*Q$, and an open neighbourhood $\mathcal{H}\subseteq \{h\in C^\infty(T^*Q)\ |\ dh\in C_0^\infty(K)\}$ of $0$
such that:
\begin{enumerate}
\item for all $h\in \mathcal{H}$ the Hamiltonian $H+h$ satisfies \eqref{T*QcapHnonempty};
\item  the Lagrangian Rabinowitz Floer functional $\A^H_{q_0,q_1}$ has critical set continuously compact in $(K,\mathcal{H})$;
\item for all $h_0,h_1 \in \mathcal{H}$ and $J_0, J_1 \in \mathcal{J}(\mathcal{V}, \mathbb{J})$ every homotopy $\Gamma=\{(H+h_s, J_s)\}_{s\in \R}$ 
from $(H+h_0,J_0)$ to $(H+h_1,J_1)$ satisfies the Novikov finiteness condition \eqref{Novikov}
and for all $a, b \in \R$ the space of Floer trajectories $\mathcal{M}^\Gamma(a,b)$ is bounded in the $L^\infty$-norm.
\end{enumerate}
Then for every $h\in\mathcal{H}$ the Lagrangian Rabinowitz Floer homology $\LRFH_*(\A^{H+h}_{q_0,q_1})$ is well defined and isomorphic to $\LRFH_*(\A^{H}_{q_0,q_1})$.
\end{thm}

\begin{proof}
We will define Lagrangian Rabinowitz Floer homology in the case $q_0\neq q_1$, and then briefly explain the case $q_0=q_1$. Therefore, we first assume $q_0\neq q_1$.

The first step is to show that for $q_0\neq q_1$ and generic $h\in \mathcal{H}$ the Lagrangian Rabinowitz action functional $\A^{H+h}_{q_0,q_1}$ is Morse, i.e., its Hessian $\nabla^2\A^{H+h}_{q_0,q_1}(x)$ has trivial kernel for all $x \in \Crit \A^{H+h}_{q_0,q_1}$. 
Observe that for $q_0\neq q_1$ we we have $\Crit \A^{H+h}_{q_0,q_1} \cap \left( \mathscr{H}_{q_0,q_1}\times\{0\}\right)=\emptyset$. By a standard Sard-Smale argument \cite[Theorem A.51]{McDuff2012}, there exists a residual subset $\mathcal{H}^{\reg}\subseteq \mathcal{H}\subseteq C_0^\infty(K)$ such that for all $h \in \mathcal{H}^{\reg}$ we have
\begin{equation*}
D \phi^\eta_{H+h} (T^*_{q_0}Q) \pitchfork T^*_{q_1}Q\qquad\forall\ (v, \eta)\in \Crit\A^{H+h}_{q_0,q_1},
\end{equation*}
where $\phi^t_{H+h}$ denotes the Hamiltonian flow of $X_{H+h}$.
The condition is equivalent to the triviality of the kernel of the Hessian $\nabla^2 \A^{H+h}_{q_0,q_1}(x)$ for all $x \in \Crit \A^{H+h}_{q_0,q_1}$.
A straightforward consequence of the Morse property of the Lagrangian Rabinowitz action functional and continuous compactness of the critical set of $\A^{H}_{q_0,q_1}$ is that the critical set of $\A^{H+h}_{q_0,q_1}$ is finite for each $h \in \mathcal{H}^{\reg}$.

Fix $h\in \mathcal{H}^{\reg}$. Since $K \subseteq \mathcal{V}$, by continuous compactness of the critical set of $\A^{H}_{q_0,q_1}$ in $(K,\mathcal{H})$ for every $J\in \mathcal{J}(\mathcal{V}, \mathbb{J})$, we have
\begin{equation}\label{nonEmpty}
v(\R\times[0,1])\cap \mathcal{V}\neq \emptyset\quad \forall\ u=(v,\eta)\in \mathscr{F}_{H+h,J}(x,y)\quad\forall\ x,y \in \Crit\A^{H+h}_{q_0,q_1}.
\end{equation}
Therefore, we can apply the Sard-Smale argument in \cite[Thm. A.51]{McDuff2012} to conclude that there exists a residual set $\mathcal{J}^{\reg}_h\subseteq \mathcal{J}(T^*Q,\omega,\mathcal{V})$ such that for all $x,y\in \Crit\A^{H+h}_{q_0,q_1}$ and all $J \in \mathcal{J}^{\reg}_h$ the space of Floer trajectories $\mathscr{F}_{H+h,J}(x,y)$ is a transversely cut out smooth manifold. 

Let us now fix $h\in \mathcal{H}^{\reg}$ and $J\in \mathcal{J}^{\reg}_h$.
For every $x,y\in \Crit\A^{H+h}_{q_0,q_1}$ there is a natural $\R$-action on $\F_{H+h,J}(x,y)$ given by
$$
\F_{H+h,J}(x,y) \times \R \ni (u, s) \longmapsto u(s + \cdot) \in \F_{H+h,J}(x,y).
$$
Denote the quotient space $\overline{\F}_{H+h,J}(x,y):=\F_{H+h,J}(x,y)/ \R$.
By \cite[Thm. 2.23]{Merry2014}, for all $x,y\in \Crit\A^{H+h}_{q_0,q_1}$ with $x\neq y$ the quotient $\overline{\F}_{H+h,J}(x,y)$ is also a smooth manifold and its dimension is given by
$$
  \dim \overline{\F}_{H+h,J}(x,y)= \mu (y) - \mu (x) -1.
$$
By condition 3 in Theorem~\ref{thm:DefLRFH}, for all $x,y\in \Crit\A^{H+h}_{q_0,q_1}$ the space of Floer trajectories $\mathscr{F}_{H+h,J}(x,y)$ is bounded in the $L^\infty$-norm. Now standard compactness arguments (see \cite[Prop. 3b]{Floer1989} or \cite[Thm. 9.1.7]{Audin2014})
imply that $\overline{\F}_{H+h,J}(x,y)$ is compact up to breaking in the sense of Floer. 
If $\mu(y)-\mu(x)=1$ this means that $\overline{\F}_{H+h,J}(x,y)$ is a finite set of points \cite[Thm. 9.2.1]{Audin2014}. 

Now we are ready to define Lagrangian Rabinowitz Floer homology. For $h\in \mathcal{H}^{\reg}$ and $k\in\Z$ let $CF_k(\A^{H+h}_{q_0,q_1})$ be the $\mathbb{Z}_2$-vector space of formal sums of the form $\sum_{x\in S}x$, where $S\subseteq \Crit\A^{H+h}_{q_0,q_1}$ is a (possibly infinite) set satisfying $\mu(x)=k$ for all $x\in S$ and the \emph{Novikov finiteness condition}
$$
\# \{ x\in S\ |\ \A^{H+h}_{q_0,q_1}(x)>a\}<+\infty\qquad\forall\ a\in\mathbb{R}.
$$
Now we fix an almost complex structure $J\in \mathcal{J}^{\reg}_h$ and turn $CF(\A^{H+h}_{q_0,q_1})$ into a chain complex. For $x,y\in \Crit\A^{H+h}_{q_0,q_1}$ with $\mu(y)-\mu(x)=1$ we set
$$
  n(x,y):=\# \overline{\F}_{H+h,J}(x,y) \operatorname{mod}2\in \mathbb{Z}_2.
$$
The Floer boundary operator $\partial_k^J: CF_k(\A^{H+h}_{q_0,q_1})\to CF_{k-1}(\A^{H+h}_{q_0,q_1})$ is the linear map defined on generators by 
$$
  \partial_k^J y := \sum_x  n(x,y)x.
$$
Compactness up to breaking implies that $\partial_k^J\circ\partial_{k+1}^J=0$, so we can define the {\em Lagrangian Rabinowitz Floer homology}
$$
  \LRFH_k(\A^{H+h}_{q_0,q_1};J) := \ker \partial_k^J / \operatorname{im}\partial_{k+1}^J.
$$
A standard continuation argument shows that the Lagrangian Rabinowitz Floer homology does not depend on the choice of $h\in \mathcal{H}^{\reg}$ and $J\in \mathcal{J}^{\reg}_h$. For this, let $h_i\in \mathcal{H}^{\reg}$ and $J_i\in \mathcal{J}^{\reg}_{h_i}$ for $i=0,1$ be given. Pick a homotopy $\Gamma = \{(H+h_s, J_s)\}_{s\in \R}$ connecting $(H+h_0,J_0)$ to $(H+h_1,J_1)$ as in condition 3 of Theorem~\ref{thm:DefLRFH}. 
As before, it follows that for generic $\Gamma$ and any pair $(x,y)\in \Crit \A^{H+h_0}_{q_0,q_1} \times \Crit \A^{H+h_1}_{q_0,q_1}$ the associated space of Floer trajectories $\F_\Gamma(x,y)$ is a smooth manifold of dimension $\mu(y)-\mu(x)$ which is compact up to breaking. In particular, if $\mu(x)=\mu(y)$ then $\F_\Gamma(x,y)$ is a finite set of points and we set $m(x,y):= \# \F_\Gamma(x,y)\mod 2$. We define a linear map $\phi^\Gamma : (CF_*(\A^{H+h_1}_{q_0,q_1}),\p^{J_1})\to (CF(\A^{H+h_0}_{q_0,q_1}),\p^{J_0})$ on generators by
$$
  \phi^\Gamma(y) := \sum_{x\in \Crit \A^{H+h_0}_{q_0,q_1}}m(x,y)x.
$$
Here the Novikov finiteness condition \eqref{Novikov} on $\Gamma$ ensures that $\phi^\Gamma\left( CF(\A^{H+h_1}_{q_0,q_1})\right)\subseteq CF(\A^{H+h_0}_{q_0,q_1})$. Compactness up to breaking implies that $\phi^\Gamma$ is a chain map, and composition with a homotopy in the opposite direction shows that $\phi^\Gamma$ induces an isomorphism on homology (see \cite[Chapter 11]{Audin2014} for details on this standard argument).

Since $\LRFH_k(\A^{H+h}_{q_0,q_1};J)$ does not depend on $J\in \mathcal{J}^{\reg}_h$, we can drop $J$ from the notation and write is as $\LRFH_k(\A^{H+h}_{q_0,q_1})$. Since this homology is the same vector space for all $h\in \mathcal{H}^{\reg}$, we can unambiguously extend it as this vector space to all $h\in \mathcal{H}^{\reg}$. This proves the theorem in the case $q_0\neq q_1$. 

In the case $q_0=q_1$, by condition 1 for each $h\in\mathcal{H}$ we have critical points $(v,0)\in\Crit\A^{H+h}_{q_0,q_0}$ where $v$ is a constant path in the nonempty transverse intersection $S_h:=T^*_{q_0}Q\cap (H+h)^{-1}(0)$. If $S_h$ has positive dimension, then the functional $\A^{H+h}_{q_0,q_0}$ will only be Morse-Bott for generic $h\in\mathcal{H}$. This phenomenon is well known in Rabinowitz Floer theory and can be dealt with by picking a Morse function on $S_h$ and counting gradient flow lines with cascades, see~\cite{CieliebakFrauenfelder2009}. With this understood, the rest of the argument works as in the previous case. 
\end{proof}

\begin{cor}
Consider a cotangent bundle $(T^*Q,\omega)$ with its standard symplectic form and a Hamiltonian $H:T^*Q\to \R$ such that $H^{-1}(0)$ is of exact contact type. Assume that there exists an exhausting sequence of compact sets $\{K_n\}_{n\in\mathbb{N}}$, $K_n\subseteq K_{n+1}\subseteq T^*Q$, $\bigcup_{n\in\mathbb{N}}K_n=T^*Q$ and an open neighbourhood $\mathcal{H}$ of $0$ in $\{h\in C^\infty(T^*Q)\ |\ dh\in C_c^\infty(T^*Q)\}$ such that for every $n\in \mathbb{N}$ the sets $K_n$ and $\mathcal{H}_n=\{h\in\mathcal{H}\ |\  dh\in C_0^\infty(K_n)\}$ satisfy the assumptions of Theorem \ref{thm:DefLRFH}. Then for any $h\in\mathcal{H}$ the Lagrangian Rabinowitz Floer homology $\LRFH_*(\A^{H+h}_{q_0,q_1})$ is well defined and isomorphic to $\LRFH_*(\A^{H}_{q_0,q_1})$.
\end{cor}

\subsection{Positive Lagrangian Rabinowitz Floer homology}

The action functional $\A^{H}_{q_0,q_1}$ provides an $\R$-filtration on $CF_*(\A^{H}_{q_0,q_1})$ by
$$
CF_*^{\leq a}\left(\A^{H}_{q_0,q_1}\right)\coloneqq \left\lbrace {\textstyle \sum_{x\in S}x}\in CF_*\left(\A^{H}_{q_0,q_1}\right) \left|\ \sup_{x\in S}\A^{H}_{q_0,q_1}(x) \leq a\right.\right\rbrace.
$$
Since Floer trajectories are defined by the $L^2$-gradient of the action functional, the boundary operator does not increase the action, i.e.\
\begin{equation}\label{filtration}
\partial\left( CF^{\leq a}_{*+1}\left(\A^{H}_{q_0,q_1}\right)\right)\subseteq CF^{\leq a}_*\left(\A^{H}_{q_0,q_1}\right).
\end{equation}
The {\em positive Rabinowitz Floer homology} $\LRFH_*^+(\A^{H}_{q_0,q_1})$ is the homology of the quotient complex
$$
  CF_*^+\left(\A^{H}_{q_0,q_1}\right) \coloneqq CF_*\left(\A^{H}_{q_0,q_1}\right)\Big/CF_*^{\leq 0}\left(\A^{H}_{q_0,q_1}\right) 
$$
with boundary operator $\partial^{\scriptscriptstyle +}$ induced by $\partial$ on the quotient.

In order to get the required transversality in the preceding discussion, we need to replace $H$ by $H+h$ with $h\in\mathcal{H}^{reg}$ as in the previous subsection.  
For independence of positive Lagrangian Rabinowitz Floer homology of the choice of $h$ we need a further condition. The regular level set $H^{-1}(0)$ is said to be of {\em exact contact type} if there exists a Liouville vector field $Y$ on $T^*Q$ (satisfying $L_Y\om=\om$) such that $dH(Y)>0$ along $H^{-1}(0)$. 
Setting $(\xi,\sigma)=(Y,\eta)$ in equation~\eqref{eq:dAH} we get
\begin{equation}\label{dA^H(Y)}
  \A^H (v,\eta) - d\A^H(v,\eta)[Y, \eta] = \eta\int_0^1 dH(Y)(v(t))dt.
\end{equation}
At critical points the second term on the left-hand side vanishes and we conclude
\begin{equation}\label{Crit+}
\begin{aligned}
  \Crit^\pm\A^{H}_{q_0,q_1} &:= \left\lbrace x\in \Crit\A^{H}_{q_0,q_1}\ |\ \pm\A^{H}_{q_0,q_1}(x) >0\right\rbrace \cr
  &= \left\lbrace (v,\eta)\in \Crit\A^{H}_{q_0,q_1}\ |\ \pm \eta >0\right\rbrace.
\end{aligned}
\end{equation}
The following lemma shows that the exact contact type property provides a set of compactly supported perturbations for which the set of positive action values is bounded away from zero.

\begin{lem}\label{lem:posAction}
Consider a cotangent bundle $(T^*Q,\omega)$ with its standard symplectic form and a Hamiltonian $H:T^*Q\to \R$ such that $H^{-1}(0)$ is of exact contact type. Fix $q_0,q_1\in Q$ with $q_0\neq q_1$.
Assume that there exists a compact subset $K\subseteq T^*Q$ and an open neighbourhood $\mathcal{H}$ of $0$ in $\{h\in C^\infty(T^*Q)\ |\ dh\in C_0^\infty(K)\}$ such that for all $h\in \mathcal{H}$ the following holds:
\begin{equation}
\forall\ (v,\eta) \in \Crit \A^{H+h}_{q_0,q_1} \quad \textrm{we have}\quad v([0,1]) \subseteq K.\label{PO_K}
\end{equation}
Then there exists an open neighbourhood $\mathcal{O}(K)$ of $0$ in $\mathcal{H}$ such that
\begin{equation}
\inf\left\lbrace\begin{array}{c|c}
\A^{H+h}_{q_0,q_1}(x) & x\in \Crit^+ \A^{H+h}_{q_0,q_1}, \quad h\in\mathcal{O}(K)\end{array} \right\rbrace >0.\label{infA+}
\end{equation}
\end{lem}

\begin{proof}
By assumption there exists a Liouville vector field $Y$ such that $dH(Y)>0$ along $H^{-1}(0)$. Since $K$ is compact,
$$
\delta:=\inf\{dH(Y)(x)\ |\ x\in K\cap H^{-1}(0)\} > 0.
$$
For $h\in \mathcal{H}$ small enough in the $C^1$-norm we then have
\begin{equation}\label{nbhd}
(H+h)^{-1}(0)\cap K \subseteq \left\lbrace
\begin{array}{c|c}
x\in K& d(H+h)(Y)(x)\geq\delta/2
\end{array} \right\rbrace.
\end{equation}
Denote by $\mathcal{O}(K)$ the set of all $h\in \mathcal{H}$ satisfying \eqref{nbhd}. Then $\mathcal{O}(K)$ is an open neighbourhood of $0$ in $\{h\in C^\infty(T^*Q)\ |\ dh\in C_0^\infty(K)\}$.
Consider $h\in\mathcal{O}(K)$ and $(v,\eta)\in \Crit^+ \A^{H+h}_{q_0,q_1}$, so $\eta>0$ in view of~\eqref{Crit+} applied to $H+h$.  
By~\eqref{PO_K} we have $v([0,1])\subset K\cap(H+h)^{-1}(0)$, so from~\eqref{dA^H(Y)} and~\eqref{nbhd} we obtain
\begin{equation*}
  |\A^{H+h}_{q_0,q_1}(v,\eta)| \geq \delta\eta/2.
\end{equation*}
On the other hand, from $v(i)\in T_{q_i}^*Q$ for $i=0,1$ and $\p_sv=\eta X_{H+h}(v)$ we deduce
$$
   0 < \varepsilon \leq |v(1)-v(0)| \leq \eta\int_0^1|X_{H+h}(v(t))|dt \leq C\eta  
$$
with the positive constants
$$
  \varepsilon := \textrm{dist}(K\cap T_{q_0}^*Q,K\cap T_{q_1}^*Q),\qquad C := \max\{|X_{H+h}(x)|\;\bigl|\;x\in K\}.
$$
The two estimates combine to $|\A^{H+h}_{q_0,q_1}(v,\eta)| \geq \delta\varepsilon/2C>0$. 
\end{proof}

The next result provides conditions under which positive Lagrangian Rabinowitz Floer homology is not only well defined, but also independent of the auxiliary choices and invariant under compact perturbations.

\begin{cor}\label{cor:posLRFH}
Consider the setting as in Theorem \ref{thm:DefLRFH} with sets $K\subseteq\mathcal{V}\subseteq T^*Q$ and $\mathcal{H}\subseteq \{C^\infty(T^*Q)\ |\ dh\in C_0^\infty(K)\}$. Assume $q_0\neq q_1$. Let $\mathcal{O}\subseteq\mathcal{H}$ be an open neighbourhood of $0$ such that for every pair $h_0,h_1 \in \mathcal{O}$ there exists a homotopy $\Gamma:= \{(H+h_s, J_s)\}_{s\in \R}$ satisfying Condition 3 of Theorem \ref{thm:DefLRFH}, and such that for every $x\in \Crit^+\A^{H+h_0}_{q_0,q_1}$ and every $y\in \Crit\A^{H+h_1}_{q_0,q_1}$ for which $\F_\Gamma(x,y)\neq \emptyset$ we have $\A^{H+h_1}_{q_0,q_1}(y)>0$.

Then for every $h\in\mathcal{O}$ its positive Lagrangian Rabinowitz Floer homology is well defined, and for every pair $h_0,h_1 \in \mathcal{O}$, $\LRFH_*^+(\A^{H+h_0}_{q_0,q_1})$ is isomorphic to $\LRFH_*^+(\A^{H+h_1}_{q_0,q_1})$.
\end{cor}

\begin{proof}
By assumption the critical set of $\A^H_{q_0,q_1}$ is continuously compact in $(K,\mathcal{O})$. Denote by $\mathcal{O}^{\reg}$ the subset of $\mathcal{O}$ consisting of all Hamiltonians $h\in \mathcal{O}$, such that $\A^{H+h}_{q_0,q_1}$ is Morse. By the standard Sard-Smale argument $\mathcal{O}^{\reg}$ is dense in $\mathcal{O}$.
Fix $h_0,h_1 \in \mathcal{O}^{\reg}$. By assumption, there exist $J_i \in \mathcal{J}_{h_i}^{\reg}\subset \mathcal{J}(\mathcal{V}, \mathbb{J})$ for $i=0,1$ and a homotopy $\Gamma= \{(H+h_s, J_s)\}_{s\in \R}$ with $h_s\in \mathcal{O}$ and $J_s\in \mathcal{J}(\mathcal{V}, \mathbb{J})$ from $(H+h_0,J_0)$ to $(H+h_1,J_1)$ with the following properties:
\begin{enumerate}[label=\alph*)]
\item the homotopy $\Gamma$ satisfies the Novikov finiteness condition \eqref{Novikov};
\item for any pair $a,b\in \R$ the space of Floer trajectories $\mathcal{M}^\Gamma(a,b)$ is bounded in the $L^\infty$-norm;
\item for every $x\in \Crit^+\A^{H+h_0}_{q_0,q_1}$ and every $y\in \Crit\A^{H+h_1}_{q_0,q_1}$ such that $\F_\Gamma(x,y)\neq \emptyset$ we have $\A^{H+h_1}_{q_0,q_1}(y)>0$.
\end{enumerate}
From properties a) and b) we get a chain map $\phi^\Gamma : CF_*(\A^{H+h_1}_{q_0,q_1})\to CF_*(\A^{H+h_0}_{q_0,q_1})$.
It satisfies $\phi^\Gamma \left(CF_*^{\leq 0}(\A^{H+h_1}_{q_0,q_1})\right) \subseteq CF_*^{\leq 0}(\A^{H+h_0}_{q_0,q_1})$, 
since otherwise there would exist $x\in \Crit^+\A^{H+h_0}_{q_0,q_1}$ and $y\in \Crit\A^{H+h_1}_{q_0,q_1}$ such that $\F_\Gamma(x,y)\neq \emptyset$ and $\A^{H+h_1}_{q_0,q_1}(y)\leq 0$, contradicting property c). 
Thus $\phi^\Gamma$ descends to a chain map $\phi^\Gamma_+ : CF_*^+(\A^{H+h_1}_{q_0,q_1})\to CF_*^+(\A^{H+h_0}_{q_0,q_1})$, and the usual argument using a homotopy in the opposite direction (see~\cite[Proposition 11.2.9]{Audin2014}) shows that the induced map on homology\linebreak $\Phi^\Gamma_+ \colon  \LRFH_*^+(\A^{H+h_1}_{q_0,q_1})\to \LRFH_*^+(\A^{H+h_0}_{q_0,q_1})$ is an isomorphism.
\end{proof}

\section{Bounds on the Floer trajectories}

The aim of this section is to show that the Copernican Hamiltonian $H_0$ defined in \eqref{DefH0} together with the set of compactly supported perturbations defined in \eqref{DefHset} satisfy the assumptions of Theorem \ref{thm:DefLRFH}, in order to apply it and prove Theorem \ref{thm:LRFH}.

We will start by showing that the necessary condition for the existence of Reeb chords on $(H_0-h)^{-1}(0)$ between any two cotangent fibres $T_{q_0}\R^2$ and $T_{q_1}\R^2$, assumption 1 of Theorem \ref{thm:DefLRFH}, holds true in our setting.
Next, we will prove that the second assumption of Theorem \ref{thm:DefLRFH} holds true in our setting, i.e., the critical set of the Lagrangian Rabinowitz action functional is bounded in $L^\infty$. It is possible to define the Lagrangian Rabinowitz Floer homology without this assumption as a limit of homologies defined in an action window, but this is more involved and we postpone it to a future paper.

The most challenging part, which will occupy most of this section, is to prove that the Floer trajectories are uniformly bounded in the $L^\infty$-norm. This is essential in first defining the Lagrangian Rabinowitz homology and then constructing the isomorphism from Theorem \ref{thm:LRFH}.

\begin{lem}\label{lem:nonEmpty}
Let $H_0$ be the Copernican Hamiltonian defined in \eqref{DefH0} and let $\mathcal{H}$ be the set of perturbations defined in \eqref{DefHset}. Then the set $(H_0-h)^{-1}(0)\cap T_q\R^2$ is compact and nonempty for every $h\in \mathcal{H}$ and every $q\in \R^2$.
\end{lem}

\begin{proof}
Fix $q\in \R^2$. Since $h$ is constant outside a compact set,
we have
$$
\lim_{|p|\to\infty} \left(H_0(q,p)-h(q,p)\right)=+\infty.
$$
On the other hand, we have $h>0$ by assumption and $H_0(q,0)=0$, hence
$$
H_0(q,0)-h(q,0)=-h(q,0)<0.
$$
Consequently, by the intermediate value theorem there exists $p\in T_q\R^2$, such that $H_0(q,p)-h(q,p)=0$.
\end{proof}

The following lemma proves that the Copernican Hamiltonian $H_0$ together with the set of compactly supported perturbations defined in \eqref{DefHset} satisfies the second assumption of Theorem \ref{thm:DefLRFH}.

\begin{lem}\label{lem:H0Chord}
Let $H_0:T^*\R^2 \to \R$ be the Copernican Hamiltonian defined in \eqref{DefH0}. Fix two points $q_0,q_1\in \R^2$. For $n\in\mathbb{N}$ and $m>0$ denote $r_n := \max\{|q_0|,|q_1|\}+n$ and define 
the compact set $K_{n,m} \subseteq T^*\R^2$ by
$$
K_{n,m}:=
\left\lbrace\begin{array}{c|c}
& r\leq r_n, \quad p_r \leq \sqrt{r_n^2+2m},\\
\hspace*{-0.2cm}\smash{\raisebox{.5\normalbaselineskip}{ $(r, p_r,\theta,p_\theta)\in T^*\R_+\times T^*S^1$}}& p_\theta \leq r_n\sqrt{ 2m}.
\end{array}\right\rbrace\,.
$$
Then for every $h \in \mathcal{H}$ such that $dh \in C_0^\infty( K_{n,m})$ and $\|h\|_{L^\infty}<m$ and any 
$(v,\eta)\in \Crit \A^{H_0-h}_{q_0,q_1}$ we have $v([0,1])\subseteq K_{n,m}$.
\end{lem}

\begin{proof}
We will use the Poisson bracket defined by $\{f,g\}=\om(X_f,X_g)=dg(X_f)$.
Using \eqref{DefH0} we calculate
$$
\{H_0,r\} = p_r, \quad \textrm{and}\quad \{H_0,\{H_0, r\}\}=\frac{p_\theta^2}{r^3}.
$$
Thus at a point where $\{\{H_0, r\}=0\}$ and $\{\{H_0,\{H_0, r\}\}\leq 0$ we have $p_r=p_\theta=0$. Since $H_0(r,\theta, 0, 0) = 0$, this implies that for each constant $c>0$ we have
\begin{equation}\label{supp_h}
H_0^{-1}(c)\cap \{\{H_0, r\}=0\}\cap \{\{H_0,\{H_0, r\}\}\leq 0\} = \emptyset.
\end{equation}
Let now $(v, \eta) \in \Crit \A^{H_0-h}_{q_0,q_1}$. We will first show that
$$
  \max r \circ v \leq r_n.
$$
Arguing by contradiction, suppose that $r\circ v(t_0)=\max r \circ v > r_n$ for some $t_0\in [0,1]$.  
Since $r\circ v(i)=|q_i|<r_n$ for $i=0,1$, we have $t_0\in (0,1)$ and the condition that $r\circ v$ attains its maximum at $t_0$ gives  
\begin{gather*}
\frac{d}{dt}r\circ v (t_{0}) = \eta\{H_0-h, r\}\circ  v (t_{0})= 0, \\
 \textrm{and} \quad \frac{d^{2}}{dt^{2}}r\circ v (t_{0})=\eta^2\{H_0-h, \{H_0-h, r\}\} \circ v (t_{0})\leq 0.
\end{gather*}
The definition of $K_{n,m}$ and $dh \in C_0^\infty( K_{n,m})$ imply that $h$ is equal to a constant $c>0$ near $v(t_0)$ and can therefore be ignored in the Poisson brackets, and we obtain a contradiction to~\eqref{supp_h}. 

For the bounds on $p_r\circ v$ and $p_\theta\circ v$, recall that $v(t)\in (H_0-h)^{-1}(0)$ for all $t\in [0, 1]$. Consequently, we have
\begin{align*}
\frac{p_r^2}{2}+\frac{p_\theta^2}{2r^2} +p_\theta & =h(r,\theta, p_r, p_\theta),\\
p_r^2+\left(\frac{p_\theta}{r}+r\right)^2 & =r^2+2h  \leq r_n^2+2m,\\
|p_r \circ v|  & \leq \sqrt{r_n^2+2m},\\
|p_\theta \circ v| & \leq r_n\left(\sqrt{r_n^2+2m}-r_n\right)\leq r_n\sqrt{2m}.
\end{align*}
\end{proof}

Before formulating the main theorem of this section we will first introduce the following notation. 
First, to avoid unnecessary clutter we will abbreviate $\A^{H_0-h}:=\A^{H_0-h}_{q_0,q_1}$ as in this section dependence on $q_0,q_1\in \R^2$ is not very relevant.
Next, recall the set $\mathcal{H}$ from \eqref{DefHset}. For a compact subset of $K\subseteq T^*\R^2$ we define
$$\mathcal{H}(K):= \{h\in \mathcal{H}\ |\ dh\in C_0^\infty(K)\}.$$
Moreover, for an open precompact subset $\mathcal{V}\subseteq T^*\R^2$ we denote by $\mathcal{J}(\mathcal{V},\mathbb{J})$ 
the set of all the $2$-parameter families of $\omega_0$-compatible almost complex structures on $(T^*\R^2, \omega_0)$ which are equal to the standard complex structure $\mathbb{J}$ outside of $\mathcal{V}$ (see the definition in Section \ref{sec:FloerTrajec}).

\begin{thm}\label{thm:FloerBounds}
Let $H_0$ be the Hamiltonian defined in \eqref{DefH0} and let $h_0,h_1 \in \mathcal{H}$. Fix two points $q_0,q_1 \in \R^2$. Let $K$ be a compact subset of $T^*\R^2$, such that for $i=0,1$ it satisfies
$$
\supp h_i \subseteq K\qquad \textrm{and}\qquad v([0,1])\subseteq K\qquad \textrm{for all}\quad (v,\eta)\in \Crit \A^{H_0-h_i}.
$$
Let $\mathcal{V}\subseteq T^*\R^2$ be an open, but precompact subset, such that $K\subseteq \mathcal{V}$. Let $\Gamma:=\{(h_s, J_s)\}_{s\in \R}$ be a smooth homotopy of Hamiltonians $h_s\in \mathcal{H}(K)$ and $2$-parameter families of almost complex structures $J_s \in C^\infty([0,1]\times \R,\mathcal{J}(\mathcal{V},\mathbb{J}))$ constant in $s$ outside $[0,1]$, such that
\begin{gather}
\|\partial_{s}h_{s}\|_{L^{\infty}}\left(\frac{1}{c}+\|J\|_{L^{\infty}}^2\right) \leq \frac{1}{3},\label{inqGamma}\\
\textrm{where}\qquad c:=\inf_{s\in[0,1]}\{h_s-dh_s(p\partial_p)\}.\label{Defc}
\end{gather}
Then for every pair $(a,b)\in \R^2$ the corresponding space $\mathcal{M}^\Gamma(a,b)$ of Floer trajectories defined in \eqref{DefM(a,b)} is bounded in the $L^\infty$-norm.
\end{thm}

\subsection{The bounds on the action} 

If the family of Hamiltonians $\{H_s\}_{s\in \R}$ is constant in $s$, i.e. $H_s = H$ for all $s\in \R$ then $\mathcal{M}^\Gamma(a,b)\neq \emptyset$ implies $b>a$. However, since in the parametric case the action along the Floer trajectories may not be monotonically increasing, we need to ensure that we have a bound on how much the action can decrease along a Floer trajectory.
Proving the bounds on the action conditions will be the main aim of this section. 

We will start by proving the following lemma:

\begin{lem}\label{lem:inqEtaAct}
For every $h\in\mathcal{H}$
\begin{gather*}
\textrm{if}\qquad \|\nabla \A^{H_0-h}(v,\eta)\|_{L^2\times\R}  <1,\\
\textrm{then}\qquad |\eta | \leq \frac{1}{c}|\A^{H_0-h}(v,\eta)|+\frac{1}{\sqrt{2c}},
\end{gather*}
where $c:=\inf\{h-dh(p\partial_p\}>0$.
\end{lem}
\begin{proof}
Using \eqref{DefH0} we can compute the Hamiltonian $H_0$ satisfies the following equality:
\begin{equation}\label{H0pdp}
dH_0(p\partial_p) =  p_1^2+p_2^2+q_2p_1-q_1p_2 =\frac{1}{2}|p|^2+H_0(p,q).
\end{equation}
Combining that with equality \eqref{dA^H(Y)} for all $(q,p,\eta)=(v,\eta)\in \mathscr{H}_{q_0,q_1}$ we obtain
\begin{align}
\A^{H_0-h}(v,\eta) -d\A^{H_0-h}[p\partial_p,0] & =\eta\int_0^1\left(d(H_0-h)(p\partial_p)-H_0(v)+h(v)\right)\nonumber \\
& = \eta \left(\frac{1}{2}\|p\|_{L^2} +\int_0^1 \left(h(v)-dh(p\partial p)\right)\right)\\
& \geq \eta \left(\frac{1}{2}\|p\|_{L^2}+c\right), \label{IntdH(Y)}
\end{align}
where the last inequality comes from the fact that
$c:=\inf\{h-dh(p\partial_p\}>0$.

Consequently, for all $(q,p,\eta)=(v,\eta)\in \mathscr{H}_{q_0,q_1}$, such that $\|\nabla \A^{H_0-h}(v,\eta)\|_{L^2\times\R}  <1$, we obtain
$$
|\eta| < \frac{|\A^{H_0-h}(v,\eta)|+\|p\|_{L^2}}{\frac{1}{2}\|p\|_{L^2}^2+c}.
$$
Now let us analyze the right-hand side as functions of $\|p\|_{L^2}$:
\begin{align*}
\max_{x\geq 0}\left(\frac{1}{\frac{1}{2} x^2+c}\right) & =\frac{1}{c},\\
\frac{d}{dx}\left( \frac{x}{\frac{1}{2}x^2+c}\right) & = \frac{c-\frac{1}{2}x^2}{\left(\frac{1}{2}x^2+c\right)^2},\\
\max_{x\geq 0} \left( \frac{x}{\frac{1}{2}x^2+c}\right) & = \frac{1}{\sqrt{2c}}.
\end{align*}
These estimates give the desired inequality.
\end{proof}
\begin{lem}\label{lem:ActBound}
Consider the setting as in Theorem \ref{thm:FloerBounds}. Fix $a, b\in \mathbb{R}$. Then $\|\eta\|_{L^\infty}$, $\|\A^{H_0-h_s}\circ u(s)\|_{L^\infty}$ and $\|\nabla^{J_s}\A^{H_0-h_s}\circ u(s)\|_{L^2([0,1])\times\R}$ are uniformly bounded for all $u \in \mathcal{M}^\Gamma(a,b)$.
\end{lem}
The following proof follows closely the arguments presented in \cite[Prop. 3.3]{Pasquotto2017}. We will not present all of them here, but we encourage a curious reader to see the details in \cite{Pasquotto2017}.

\begin{proof}
Using the fact that $u\in \mathcal{M}^\Gamma(a,b)$ is a Floer trajectory one can calculate the derivative of the action functional over $s$ and obtain the following inequalities (see  \cite[Prop. 3.3]{Pasquotto2017}):
\begin{align}
\|\mathcal{A}^{H_0-h_{s}}( u)\|_{L^{\infty}} & \leq \max\{|a|,|b|\} + \|\eta\|_{L^{\infty}}\|\partial_{s}h_{s}\|_{L^{\infty}},\label{inqAct}\\
\|\nabla^{J_{s}} \mathcal{A}^{H_0-h_{s}}(u)\|_{L^{2}(\mathbb{R}\times [0,1])}^{2}
& \leq \|J\|_{L^{\infty}} (b-a + \|\eta\|_{L^{\infty}}\|\partial_{s}h_{s}\|_{L^{\infty}}). \label{inqEnergy}
\end{align}

In particular the convergence of the integral
$$
\|\nabla \mathcal{A}^{H_0-h_{s}}(u)\|_{L^{2}(\mathbb{R}\times [0,1])} \leq \|J\|_{L^{\infty}}  \|\nabla^{J_{s}} \mathcal{A}^{H_0-h_{s}}(u)\|_{L^{2}(\mathbb{R}\times [0,1])},
$$
implies that for small enough $s$ we have $\|\nabla \mathcal{A}^{H_0-h_{s}}(u(s))\|_{L^{2}\times\mathbb{R}}<1$.
This ensures that for all $s \in \mathbb{R}$ the following value $\tau_{0}(s)$ is well defined and finite
$$
\tau_{0}(s): = \sup\{ \tau \leq s\ |\ \|\nabla \mathcal{A}^{H_0-h_{\tau}}(u(\tau))\|_{L^{2}\times\mathbb{R}}<1\}.
$$
For $\tau \in [\tau_{0}(s),s]$ we have $\|\nabla \mathcal{A}^{H_0-h_{\tau}}(u(\tau))\|_{L^{2}\times\mathbb{R}}\geq 1$, which allows us to estimate (see  \cite[Prop. 3.3]{Pasquotto2017}):
\begin{align}
|s -\tau_0(s)| & \leq \|J\|_{L^{\infty}}^3 (b-a + \|\eta\|_{L^{\infty}}\|\partial_{s}h_{s}\|_{L^{\infty}}),\\
|\eta(s) -\eta(\tau_0(s))| & \leq \|J\|_{L^{\infty}}^2 (b-a + \|\eta\|_{L^{\infty}}\|\partial_{s}h_{s}\|_{L^{\infty}}). \label{inqeta}
\end{align}

Now using 
the constant $c>0$ as in \eqref{Defc} together with the
the estimates from Lemma \ref{lem:inqEtaAct} and \eqref{inqAct} we can calculate:
\begin{align*}
|\eta(s)|& \leq |\eta(\tau_0(s))|+|\eta(s) -\eta(\tau_0(s))|\\
& \leq \frac{1}{c}|\A^{H_0-h_s}(u(s))|+\frac{1}{\sqrt{2c}}+\|J\|_{L^{\infty}}^2 (b-a + \|\eta\|_{L^{\infty}}\|\partial_{s}h_{s}\|_{L^{\infty}})\\
& \leq \|\eta\|_{L^{\infty}}\|\partial_{s}h_{s}\|_{L^{\infty}}\left(\frac{1}{c}+\|J\|_{L^{\infty}}^2\right)+\frac{1}{c}\max\{|a|,|b|\}+\frac{1}{\sqrt{2c}}+\|J\|_{L^{\infty}}^2 (b-a).
\end{align*}
Since this inequality has to hold for all $s\in \R$ it also has to hold for $\|\eta\|_{L^\infty}$. Using \eqref{inqGamma} we obtain the following estimate:
\begin{align}
\|\eta\|_{L^{\infty}} & \leq \frac{\frac{1}{c}\max\{|a|,|b|\}+\frac{1}{\sqrt{2c}}+\|J\|_{L^{\infty}}^2 (b-a)}{1-\|\partial_{s}h_{s}\|_{L^{\infty}}\left(\frac{1}{c}+\|J\|_{L^{\infty}}^2\right)}\nonumber\\
&\leq \frac{3}{2} \left( \frac{1}{c}\max\{|a|,|b|\}+\frac{1}{\sqrt{2c}}+\|J\|_{L^{\infty}}^2 (b-a)\right)=:\mathfrak{y},\label{eqEta}
\end{align}
Now using \eqref{inqGamma}, \eqref{inqAct}, \eqref{inqEnergy} and \eqref{eqEta} we obtain the desired uniform bounds:
\begin{align}
\|\mathcal{A}^{H_0-h_{s}-c}( u)\|_{L^{\infty}} &\leq \max\{|a|,|b|\} + \frac{c}{3}\mathfrak{y}=:\mathfrak{a},\label{DefA}\\ 
\|\nabla^{J_{s}} \mathcal{A}^{H_0-h_{s}-c}(u)\|_{L^{2}(\mathbb{R}\times [0,1])}^{2} &  \leq \|J\|_{L^{\infty}} \left(b-a + \frac{c}{3}\mathfrak{y}\right)=:\mathfrak{e}.\label{DefE}
\end{align}
\end{proof}
Having obtained the bounds on the action we are ready to prove the Novikov finiteness condition:
\begin{lem}\label{lem:Novikov}
Consider the setting as in Theorem \ref{thm:FloerBounds}. Then $\mathcal{M}_\Gamma(a,b)\neq \emptyset$ implies
$$
a \leq \max\left\lbrace 3 b, 3\sqrt{\frac{c}{2}}\right\rbrace \qquad \textrm{and} \qquad b \geq \min\left\lbrace 3a, -3\sqrt{\frac{c}{2}}\right\rbrace.
$$
\end{lem}
\begin{proof}
This proof follows arguments presented in \cite[Cor. 3.8]{CieliebakFrauenfelder2009}. We will prove the first inequality since the second is completely analogous.
\begin{align}
\textrm{First assume}\ \qquad 3\sqrt{\frac{c}{2}} & \leq a \quad \textrm{and} \quad |b|\leq a.\label{aBound}\\
\textrm{Then} \qquad\max\{|a|,|b|\} & = a \quad \textrm{and} \quad b-a\leq 0.\nonumber
\end{align}
By \eqref{eqEta} and \eqref{aBound} we get:
$$
\|\eta\|_{L^\infty}\leq \frac{3}{2}\left(\frac{a}{c}+\frac{1}{\sqrt{2c}}\right)=2\frac{a}{c}.
$$
On the other hand, by \eqref{inqGamma} we have $\|\partial_sh_s\|_{L^\infty}\leq \frac{c}{3}$. This, together with \eqref{inqEnergy}, implies
$$
b \geq a- \|\eta\|_{L^\infty}\|\partial_sh_s\|_{L^\infty} \geq a -2 \frac{a}{c}\cdot \frac{c}{3}=\frac{1}{3}a.
$$

Now assume that  $ 3\sqrt{\frac{c}{2}} \leq a < |b|$. To finish the proof it suffices to exclude the case $a< - b$.
$$
\textrm{Then} \qquad\max\{|a|,|b|\} = -b \quad \textrm{and} \quad b-a\leq 0.
$$
By \eqref{eqEta} and \eqref{aBound} we get:
$$
\|\eta\|_{L^\infty}\leq \frac{3}{2}\left(\frac{-b}{c}+\frac{1}{\sqrt{2c}}\right)\leq \frac{3}{2}\left(\frac{-b}{c}+\frac{a}{3c}\right)\leq -2 \frac{b}{c}.
$$
This, together with \eqref{inqGamma} and \eqref{inqEnergy}, gives us a contradiction:
\begin{gather*}
b \geq a- \|\eta\|_{L^\infty}\|\partial_sh_s\|_{L^\infty}\geq a + 2 \frac{b}{c} \cdot \frac{c}{3}=a+\frac{2}{3}b,\\
0 \geq \frac{1}{3} b \geq a \geq 3\sqrt{\frac{c}{2}}>0.
\end{gather*}
\end{proof}

\subsection{The set of infinitesimally small action derivation}

Let $H_0$ be the Hamiltonian defined in \eqref{DefH0} and let $\mathcal{H}$ be the set of Hamiltonians as in \eqref{DefHset}. For a Hamiltonian $h \in \mathcal{H}$ and fixed constants $\mathfrak{a}, \mathfrak{y}, \varepsilon>0$ we define the following set:
\begin{equation}
\mathcal{B}_h(\mathfrak{a}, \mathfrak{y}, \varepsilon):= \left\lbrace\begin{array}{c|c}
 & |\nabla\A^{H_0-h}(v,\eta)|_{L^2\times\R}<\varepsilon,\\
{\smash{\raisebox{.5\normalbaselineskip}{ $(v, \eta) \in\mathscr{H}_{q_0,q_1}\times\R$}}}& |\A^{H_0-h}(v,\eta)|\leq \mathfrak{a},\ |\eta|\leq \mathfrak{y}.
\end{array}\right\rbrace
\end{equation}
We will call this set the set of infinitesimally small action derivation.

The main aim of this subsection will be to prove the following lemma:

\begin{prop}\label{prop:smallDerivSetBound}
For fixed constants $\mathfrak{a}, \mathfrak{y}>0$ and $\varepsilon>0$ small enough, the corresponding set $\mathcal{B}_h(\mathfrak{a}, \mathfrak{y}, \varepsilon)$ is bounded both in the $L^{\infty}\times\R$-norm and in the $L^2\times \R$-norm.
\end{prop}

We will prove Proposition \ref{prop:smallDerivSetBound} in a series of smaller lemmas:

\begin{lem}\label{lem:partial_p}
Let $h\in \mathcal{H}$ and let $c:=\inf\{h-dh(p\partial_p)\}>0$.
If we fix $\mathfrak{a}, \mathfrak{y}, \varepsilon>0$ then for every $(q,p,\eta) \in \mathcal{B}_h (\mathfrak{a}, \mathfrak{y}, \varepsilon)$ we have:
$$
|\eta|\|p\|_{L^2}\leq  2\varepsilon +\frac{\mathfrak{a}}{\sqrt{2c}}\qquad\textrm{and}\qquad \|\partial_t p\|_{L^2}\leq  3\varepsilon +\frac{\mathfrak{a}}{\sqrt{2c}} + \mathfrak{y}\|\nabla h\|_{L^\infty}
$$
\end{lem}

\begin{proof}
Using \eqref{IntdH(Y)} for $(q,p,\eta) \in \mathcal{B}_h (\mathfrak{a}, \mathfrak{y}, \varepsilon)$ we obtain
\begin{gather*}
\eta \left(\frac{1}{2}\|p\|_{L^2}^2+c\right) \leq \A^{H_0-h}(v,\eta)-d\A^{H_0-h}[p\partial_p,0] \leq \mathfrak{a}+\varepsilon\|p\|_{L^2},\\
|\eta| \|p\|_{L^2}\leq \frac{\mathfrak{a}\|p\|_{L^2}+\varepsilon\|p\|_{L^2}^2}{\frac{1}{2}\|p\|_{L^2}^2+c}.
\end{gather*}
Now we would like to estimate the maximum of the function on the right-hand side of the inequality:
\begin{gather*}
\frac{d}{dx}\left(\frac{\mathfrak{a}x+\varepsilon x^2}{\frac{1}{2}x^2+c}\right) = -\frac{\frac{1}{2}\mathfrak{a}x^2-2\varepsilon c x -\mathfrak{a} c}{\left(\frac{1}{2}x^2+c\right)^2},\\
\frac{d}{dx}\left(\frac{\mathfrak{a}x+\varepsilon x^2}{\frac{1}{2}x^2+c}\right) = 0 \quad \iff \quad x=\frac{1}{\mathfrak{a}}\left(2\varepsilon c \pm \sqrt{4 \varepsilon^2 c^2 + 2 \mathfrak{a}^2 c}\right),\\
\max_{x\geq 0} \left(\frac{\mathfrak{a}x+\varepsilon x^2}{\frac{1}{2}x^2+c}\right)= \varepsilon+ \sqrt{\frac{\mathfrak{a}^2}{2c}+\varepsilon^2} \leq 2\varepsilon +\frac{\mathfrak{a}}{\sqrt{2c}}.
\end{gather*}
This gives us the bound on $|\eta| \|p\|_{L^2}$, i.e. the first inequality.

Let us denote by 
$$
(X_{H_0-h})_p:= \left(p_2 +\frac{\partial h}{\partial_{q_1}}\right)\partial_{p_1}-\left(p_1-\frac{\partial h}{\partial_{q_2}}\right)\partial_{p_2}.
$$
Then we can estimate
\begin{align*}
\|\partial_tp\|_{L^2} & \leq \| \partial_t p-\eta (X_{H_0-h})_p\|_{L^2}+|\eta| \|(X_{H_0-h})_p\|_{L^2}\nonumber\\
& \leq \|\nabla\A^H(q,p,\eta)\|_{L^2\times \R}+|\eta|(\|p\|_{L^2}+\|\nabla h \|_{L^\infty}),\\
& \leq  3\varepsilon +\frac{\mathfrak{a}}{\sqrt{2c}} + \mathfrak{y}\|\nabla h \|_{L^\infty}.
\end{align*}
\end{proof}

\begin{lem}\label{lem:qBound}
If we fix $\mathfrak{a}, \mathfrak{y}, \varepsilon>0$ then for every $(q,p,\eta) \in \mathcal{B}_h (\mathfrak{a}, \mathfrak{y}, \varepsilon)$ we have:
$$
\|q\|_{L^2}, \| q\|_{L^\infty} \leq \min\{|q_0|,|q_1|\}+2 \left(3\varepsilon + \frac{\mathfrak{a}}{\sqrt{2c}}+\mathfrak{y}\|\nabla h\|_{L^\infty}\right).
$$
\end{lem}

\begin{proof}
Begin first observe that
$$
qdq(X_{H_0})=q_1(p_1+q_2)+q_2(p_2-q_1)=q_1 p_1 + q_2 p_2.
$$

Therefore, for all $t\in [0,1]$ we have
\begin{align}
|q(t)|^2 - |q_0|^2& 
 =2\int_0^t qdq\left(\partial_tv - \eta X_{H_0-h}\right)+2\eta \int_0^t \left(qdq(X_{H_0})-qdq(X_h)\right)\nonumber\\
& \leq 2\|q\|_{L^2}\|\nabla \A^H(q,p,\eta)\|_{L^2\times \R}+2|\eta|\|q\|_{L^2}\left(\|p\|_{L^2} + \|\nabla h\|_{L^\infty}\right),\nonumber\\
& \leq 2 \|q\|_{L^2}\left(3\varepsilon + \frac{\mathfrak{a}}{\sqrt{2c}}+\mathfrak{y}\|\nabla h\|_{L^\infty}\right), \label{ineq1}
\end{align}
where the last inequality we obtained using the result form Lemma \ref{lem:partial_p}. Integrating both sides, we obtain
$$
\|q\|_{L^2}^2\leq |q_0|^2+2 \|q\|_{L^2}\left(3\varepsilon + \frac{\mathfrak{a}}{\sqrt{2c}}+\mathfrak{y}\|\nabla h\|_{L^\infty}\right).
$$
By solving this quadratic inequality we obtain the following bound:
\begin{align*}
\|q\|_{L^2} & \leq 3\varepsilon + \frac{\mathfrak{a}}{\sqrt{2c}}+\mathfrak{y}\|\nabla h\|_{L^\infty} + \sqrt{\left(3\varepsilon + \frac{\mathfrak{a}}{\sqrt{2c}}+\mathfrak{y}\|\nabla h\|_{L^\infty}\right)^2+|q_0|^2}\\
& \leq |q_0|+2 \left(3\varepsilon + \frac{\mathfrak{a}}{\sqrt{2c}}+\mathfrak{y}\|\nabla h\|_{L^\infty}\right)
\end{align*}
By repeating this procedure with the equation $|q(t)|^2=|q_1|^2-2\int_t^1 q\partial_tq$, we obtain the bound on $\|q\|_{L^2}$ we were looking for.

To obtain the bound for $\|q\|_{L^\infty}$ we will use \eqref{ineq1} again:
\begin{align*}
\|q\|_{L^\infty} & \leq \sqrt{|q_0|^2+ 2 \|q\|_{L^2}\left(3\varepsilon + \frac{\mathfrak{a}}{\sqrt{2c}}+\mathfrak{y}\|\nabla h\|_{L^\infty}\right)}\\
& \leq \sqrt{|q_0|^2+ 2 \left( |q_0|+2 \left(3\varepsilon + \frac{\mathfrak{a}}{\sqrt{2c}}+\mathfrak{y}\|\nabla h\|_{L^\infty}\right)\right)\left(3\varepsilon + \frac{\mathfrak{a}}{\sqrt{2c}}+\mathfrak{y}\|\nabla h\|_{L^\infty}\right)}\\
& \leq |q_0|+ 2\left(3\varepsilon + \frac{\mathfrak{a}}{\sqrt{2c}}+\mathfrak{y}\|\nabla h\|_{L^\infty}\right)
\end{align*}
Analogously as before, we repeat this procedure with the equation\linebreak $|q(t)|^2=|q_1|^2-2\int_t^1 q\partial_tq$ to obtain the bound on $\|q\|_{L^\infty}$ we were looking for.
\end{proof}
\begin{lem}\label{lem:pBound}
If we fix $\mathfrak{a}, \mathfrak{y}, \varepsilon>0$ then for every $(q,p,\eta) \in \mathcal{B}_h (\mathfrak{a}, \mathfrak{y}, \varepsilon)$ we have:
\begin{align*}
\|p\|_{L^2} & \leq 2\mathfrak{q}+\sqrt{2 \left( \|h\|_{L^\infty}+c+\varepsilon\right)}=:\mathfrak{p},\\
\|p\|_{L^\infty} & \leq \mathfrak{p}+\left( \varepsilon+\|\nabla h\|_{L^\infty}\right),\\
\textrm{where}\qquad \mathfrak{q} & :=\min\{|q_0|,|q_1|\}+2 \left(3\varepsilon + \frac{\mathfrak{a}}{\sqrt{2c}}+\mathfrak{y}\|\nabla h\|_{L^\infty}\right).
\end{align*}
\end{lem}
\begin{proof}
By \eqref{H0pdp} we have
\begin{align*}
\frac{1}{2}\|p\|_{L^2}^2 & = \int_0^1 (H_0-h)(q,p)dt -\int_0^1\left(p_1q_2-p_2q_1\right)dt +\int_0^1 h(q,p)dt\\
& \leq \|\nabla \A^{H_0-h-c}\|_{L^2\times \R}+\|p\|_{L^2}\|q\|_{L^2}+\|h\|_{L^\infty}+c\\
& \leq \varepsilon + \mathfrak{q}\|p\|_{L^2} +\|h\|_{L^\infty}+c
\end{align*}
where $\mathfrak{q}:=\min\{|q_0|,|q_1|\}+2 \left(3\varepsilon + \frac{\mathfrak{a}}{\sqrt{2c}}+\mathfrak{y}\|\nabla h\|_{L^\infty}\right)$ as in Lemma \ref{lem:qBound}. By solving this quadratic inequality we obtain the following bound:
\begin{align*}
\|p\|_{L^2} & \leq \mathfrak{q} + \sqrt{\mathfrak{q}^2 + 2 \left( \|h\|_{L^\infty}+c+\varepsilon\right)} \leq 2 \mathfrak{q}+\sqrt{2 \left( \|h\|_{L^\infty}+c+\varepsilon\right)}\\
& = 2\min\{|q_0|,|q_1|\}+4 \left(3\varepsilon + \frac{\mathfrak{a}}{\sqrt{2c}}+\mathfrak{y}\|\nabla h\|_{L^\infty}\right)+\sqrt{2 \left( \|h\|_{L^\infty}+c+\varepsilon\right)}.
\end{align*}

To prove the second bound first observe that there exists $t_0\in [0,1]$, such that $|p(t_0)|\leq  2 \mathfrak{q}+\sqrt{2 \left( \|h\|_{L^\infty}+c+\varepsilon\right)}$. Thus can estimate:
\begin{align*}
|p(t)| & \leq |p(t_0)| + \int_{t_0}^t |\partial_tv|= |p(t_0)| + \int_{t_0}^t |\partial_tv- X_{H_0-h}|+\int_{t_0}^t |X_h|\\
& \leq |p(t_0)| +\|\nabla\A^{H_0-h-c}\|_{L^2\times \R} + \|\nabla h\|_{L^\infty} \leq |p(t_0)| +\varepsilon+\|\nabla h\|_{L^\infty}\\
& \leq 2\min\{|q_0|,|q_1|\}+13\varepsilon + 4\frac{\mathfrak{a}}{\sqrt{2c}}+(4\mathfrak{y}+1)\|\nabla h\|_{L^\infty} +\sqrt{2 \left( \|h\|_{L^\infty}+c+\varepsilon\right)}.
\end{align*}
\end{proof}

\begin{rem}
Note that in the proof above we did not use the assumption that $dh \in C_c^\infty(T^*\R^2)$. Thus we can conclude that for any $\mathfrak{a}$, $\mathfrak{y}$, $\varepsilon>0$ and any Hamiltonian $h\in C^\infty(T^*\R^2)\cap W^{1,\infty}(T^*\R^2)$, such that $h>0$ and $h-dh(p\partial_p)>0$ the corresponding set $\mathcal{B}_h(\mathfrak{a},\mathfrak{y},\varepsilon)$ is bounded in the $L^{\infty}\times \R$-norm and the $L^2\times \R$-norm.
\end{rem}

\begin{rem}\label{rem:smallDeriv}
Note that for $h \in \mathcal{H}$ the bounds on $\mathcal{B}_h(\mathfrak{a},\mathfrak{y}, \varepsilon)$ depend smoothly on the constants $\mathfrak{a}$, $\mathfrak{y}$, $\varepsilon$, $c:=\inf\{h-dh(p\partial_p)\}$ and $\|h\|_{W^{1,\infty}}$. Thus if we take a subset $\mathcal{H}'\subseteq \mathcal{H}$ bounded in the $W^{1,\infty}$-norm such that $\inf_{h\in\mathcal{H}'}\inf\{h-dh)p\partial_p)\}>0$, then the corresponding set
$$
\bigcup_{h\in\mathcal{H}'}\mathcal{B}_h(\mathfrak{a},\mathfrak{y},\varepsilon)\subseteq \mathscr{H}_{q_0,q_1}\times \R,
$$
will be also bounded in the $L^{\infty}\times \R$-norm and the $L^2\times \R$-norm.
\end{rem}

\subsection{The $L^2$ bounds}
The aim of this subsection will be to prove the following Proposition:
\begin{prop}\label{prop:L2bounds}
Consider a setting as in Theorem \ref{thm:FloerBounds}. Then for every pair $a,b\in \R$ the corresponding space $\mathcal{M}^\Gamma(a,b)$ is bounded in the $L^2\times\R$-norm.
\end{prop}
\begin{proof}
By Lemma \ref{lem:ActBound} there exist $\mathfrak{a,e, y}>0$, such that for every $u:=(v,\eta)\in \mathcal{M}^\Gamma(a,b)$ we have
\begin{align*}
\|\eta\|_{L^\infty} \leq \mathfrak{y}, \qquad\|\mathcal{A}^{H_0-h_s}( u)\|_{L^{\infty}} &\leq\mathfrak{a},\\ \textrm{and}\qquad \|\nabla^{J_s} \mathcal{A}^{H_0-h_s}(u)\|_{L^{2}(\mathbb{R}\times [0,1])}^{2}& \leq \mathfrak{e}.
\end{align*}
Moreover, by \eqref{DefE} the convergence of the integral
$$
\|\nabla \mathcal{A}^{H_0-h_s}(u)\|_{L^{2}(\mathbb{R}\times [0,1])}^{2} \leq \mathfrak{e}\|J\|_{\infty}^2,
$$
implies that for every $\varepsilon>0$ there exists $s\in \R$ small enough such that\linebreak $\|\nabla \mathcal{A}^{H_0-h_s}(u(s))\|_{L^{2}\times\mathbb{R}}<\varepsilon$. 
This ensures that for all $s \in \mathbb{R}$ the following value $\tau_\varepsilon(s)$ is well defined and finite
$$
\tau_\varepsilon(s): = \sup\left\lbrace \tau \leq s\ \Big|\ u(\tau) \in \mathcal{B}_{h_\tau}(\mathfrak{a},\mathfrak{y}, \varepsilon)\right\rbrace.
$$
Note that for all $\tau \in [\tau_\varepsilon, s]$ we have $\|\nabla \mathcal{A}^{H_0-h_\tau-c}(u(\tau))\|_{L^{2}\times\mathbb{R}}\geq \varepsilon$, thus we can estimate
\begin{gather} 
\varepsilon^2 |s-\tau_\varepsilon(s)| \leq\|\nabla \mathcal{A}^{H_0-h_s}(u)\|_{L^{2}(\mathbb{R}\times [0,1])}^{2}\leq\mathfrak{e}\|J\|_{\infty}^2,\nonumber\\
|s-\tau_\varepsilon(s)|\leq\frac{ \mathfrak{e}\|J\|_{\infty}^2}{\varepsilon^2}.\label{sBoundEpsi}
\end{gather}

On the other hand, we know that
\begin{gather*}
c:=\inf_{s\in[0,1]}\{h_s-dh_s(p\partial_p)\}>0,\\
\sup_{s\in[0,1]}\|h_s\|_{L^\infty} < +\infty , \qquad\textrm{and} \qquad \sup_{s\in[0,1]}\|\nabla h_s\|_{L^2} < +\infty,
\end{gather*}
therefore by Remark \ref{rem:smallDeriv} we know that the set
\begin{equation}\label{DefBGamma}
\mathcal{B}^\Gamma(\mathfrak{a},\mathfrak{y},\varepsilon):=\bigcup_{s\in [0,1]}\mathcal{B}_{h_s}(\mathfrak{a},\mathfrak{y},\varepsilon)\subseteq \mathscr{H}_{q_0,q_1}\times \R,
\end{equation}
is bounded in $L^\infty\times \R$- and $L^2\times \R$-norm. In fact, we can use the same bounds as in Lemmas \ref{lem:qBound} and \ref{lem:pBound} denoting 
\begin{align*}
\mathfrak{h} & :=\sup_{s\in[0,1]}\|h_s\|_{W^{1,\infty}},\\
\mathfrak{q} & := \min\{|q_0|,|q_1|\}+2 \left(3\varepsilon + \frac{\mathfrak{a}}{\sqrt{2c}}+\mathfrak{y h}\right),\\
\mathfrak{p} & := 2\mathfrak{q}+\sqrt{2 \left( \mathfrak{h}+c+\varepsilon\right)}.
\end{align*}

Now using \eqref{sBoundEpsi} together with Lemma \ref{lem:qBound} we obtain
\begin{align}
|q(s,t)| & \leq |q(\tau_\varepsilon(s),t)|+\int_{\tau_\varepsilon(s)}^s |\partial_s q(\tau,t)|d\tau,\nonumber\\
& \leq |q(\tau_\varepsilon(s),t)|+\int_{\tau_\varepsilon(s)}^s |\partial_s v(\tau,t)|d\tau,\nonumber\\
\|q(s)\|_{L^2}& \leq \|q \circ \tau_0(s)\|_{L^2}+\left( \int_0^1\left| \int_{\tau_0(s)}^s\partial_sv d\tau\right|^2dt\right)^{\frac{1}{2}}\nonumber\\
& \leq \|q \circ \tau_0(s)\|_{L^2}+\sqrt{|s-\tau_0(s)|}\left( \int_0^1\int_{\tau_0(s)}^s\left|\partial_sv\right|^2d\tau dt\right)^{\frac{1}{2}}\nonumber\\
& \leq \|v \circ \tau_0(s)\|_{L^2}+\frac{\sqrt{\mathfrak{e}}\|J\|_\infty}{\varepsilon}\ \|\nabla^{J_s}\A^{H_0-h_s}\|_{L^2(\R\times[0,1])}\nonumber\\
& \leq \mathfrak{q}+\frac{\mathfrak{e}\|J\|_\infty}{\varepsilon}. \label{qL2Bound}
\end{align}
Analogously, using Lemma \ref{lem:pBound} we obtain
\begin{align}
\|p(s)\|_{L^2}& \leq \mathfrak{p} +\frac{\mathfrak{e}\|J\|_\infty}{\varepsilon},\label{pL2Bound}\\
\|v(s)\|_{L^2}& \leq \mathfrak{q}+\mathfrak{p} +\frac{\mathfrak{e}\|J\|_\infty}{\varepsilon}\label{vL2Bound}.
\end{align}
\end{proof}

\subsection{The maximum principle}
In this section we will explain how to use Aleksandrov’s maximum principle to find $L^\infty$-bounds on the Floer trajectories outside $\B(a, \varepsilon)$, following the argument of Abbondandolo and Schwartz in \cite{Abbon2009}. The contents of this section are word to word from the last author's paper with Pasquotto \cite{Pasquotto2017}, but we include it here for the completeness of the argument.

\begin{thm}(\textbf{Aleksandrov's maximum principle})\\
Let $\Omega$ be a domain in $\mathbb{R}^{2}$ and let $\rho:\Omega \to \mathbb{R}$  be a function in $C^{2}(\Omega)\cap C^(\overline{\Omega})$ function satisfying the elliptic differential inequality
$$
\triangle \rho + \langle h, \nabla \rho \rangle \geq f,
$$
where $g\in L^2(\Omega,\R^2)$ and $f\in L_{\operatorname{loc}}^1(\Omega)$. Then there exists $C>0$ which depends only on $\operatorname{diam}\Omega$ and $\|g\|_{L^2(\Omega)}$, such that
$$
\sup_{\Omega} \rho \leq \sup_{\partial \Omega}\rho +C\left(\|g\|_{L^{2}(\Omega)}\right)\|f^{-}\|_{L^{2}(\Omega)},
$$
provided $g$ and the negative part of $f$ are in $L^{2}(\Omega)$.
\end{thm}

In order to apply Aleksandrov's maximum principle and find $L^{\infty}$ bounds on the Floer trajectories, one first has to construct a function $F$ with compact level sets, whose composition with a Floer trajectory $u= (v,\eta)\in \mathscr{H}_{q_0,q_1}\times \R$ satisfies the elliptic differential inequality
$$
\triangle (F\circ v) + \langle g, \nabla (F\circ v) \rangle \geq f 
$$
outside of the set of infinitesimal action derivation $\mathcal{B}^\Gamma(\mathfrak{a},\mathfrak{y},\varepsilon)$, i.e. on every connected component $\Omega$ of
$$
\Omega\subseteq \left(\R \setminus u^{-1}\left(\mathcal{B}^\Gamma(\mathfrak{a},\mathfrak{y},\varepsilon)\right)\right)\times[0,1].
$$
Having such an inequality, one can apply the Aleksandrov maximum principle, which gives us
$$
\sup_{\Omega} (F\circ v) \leq \sup_{\partial \Omega}(F\circ v)+ C(\|g\|_{L^{2}(\Omega)})\|f^{-}\|_{L^{2}(\Omega)},
$$
provided $g$ and the negative part of $f$ are in $L^{2}(\Omega)$.

The core of this method is to find a function satisfying all the required properties. The classical approach is to use plurisubharmonic functions.
\begin{define}
Let $(M,\omega)$ be a symplectic manifold and $J$ a compatible almost complex structure. Then a $C^{2}$ function $F:M\to\mathbb{R}$ is called plurisubharmonic if
$$
-dd^{\mathbb{C}}F=\omega,
$$
where $d^{\mathbb{C}}F=dF\circ J$.
\end{define}

\begin{rem}\label{rem:pluri}
Let $(M, \omega)$ be a symplectic manifold and let $J$ be a compatible almost complex structure on $M$. Then a function $F:M\to\mathbb{R}$ is plurisubharmonic if and only if its gradient with respect to $g$ is a Liouville vector field.
\end{rem}

The reason one uses plurisubharmonic functions is because their composition with a $2$-dimensional curve $v:\Omega \to M,\Omega \subseteq \R^2$ satisfies
$$
-dd^\C (F \circ v)=\triangle (F \circ v) ds\wedge dt.
$$
In particular, if $v$ is a $J$-holomorphic curve then the elliptic inequality is trivially satisfied. Unfortunately, in the case of Floer trajectories proving the elliptic inequality is a little more complicated. One has to investigate how the plurisubharmonic function interacts with the Hamiltonian vector field, in particular one needs to understand the functions
$dF(X_H)$ and $d^{\mathbb{C}}F(X_H)$, which appear if we calculate $d^{\mathbb{C}}(F\circ u)$.

The proof of the following Lemma you can find in \cite{Pasquotto2017} as a part of the proof of Proposition 7.1:
\begin{lem}
Let $H$ be a Hamiltonian function on an exact symplectic manifold $(M, \omega=d\lambda)$ and let $J$ be a compatible almost complex structure on $M$. Let $u: \R \times [0,1]\to M \times \R$ be a Floer trajectory satisfying $\partial_s u =\nabla \A^H(u)$ (with constant $H$ and $J$). Then for a plurisubharmonic function $F: M \to \R$ we have
\begin{equation}\label{Laplace}
\begin{aligned}
\triangle (F \circ v) & = \|\partial_{s}v\|^{2}+\eta \left(dH +d(d^{\mathbb{C}}F(X_H))+d^\mathbb{C}(dF(X_H))\right)(\partial_{s}v)\\
& +\eta^2 d(dF(X_H))(X_H)+\partial_{s}\eta\ d^{\mathbb{C}}F(X_H).
\end{aligned}
\end{equation}

\end{lem}

\subsection{The $L^\infty$ bounds}
In this subsection we will apply the Maximum Principle explained in the previous subsection to the setting of the Floer trajectories from $\mathcal{M}^\Gamma(a,b)$ to establish the $L^\infty$-bounds.

\vspace*{.25cm}
\noindent \textit{Proof of Theorem \ref{thm:FloerBounds}:} 

Let $\Gamma:={(h_s,J_s)}_{s\in \R}$ be a smooth homotopy of Hamiltonians $h_s\in\mathcal{H}$ and almost complex structures $J_s\in C^\infty([0,1]\times \R, \mathcal{J}(\mathcal{V},\mathbb{J}))$ satisfying \eqref{inqGamma}. Here $\mathcal{V}\subseteq T^*\R^2$ is the open, but pre-compact subset defined in Theorem \ref{thm:FloerBounds} with the property that for all $s\in [0,1]$ $\supp dh_s\subseteq \mathcal{V}$ and $J_s\big|_{T^*\R^2\setminus \mathcal{J}}\equiv \mathbb{J}$.

Let $c>0$ be as in \eqref{Defc}.
If we fix $a,b \in \R$ then by Lemma \ref{lem:ActBound} we know that there exists $\mathfrak{y,a,e}>0$, which depend only on $a,b,c$ and $\|J\|_\infty$ such that
\begin{equation}\label{EAEtaBounds}
\begin{aligned}
\sup\left\lbrace|\eta(s)|\ |\ (v,\eta)\in \mathcal{M}^\Gamma(a,b), \ s\in \R\right\rbrace & \leq \mathfrak{y},\\
\sup\left\lbrace|\A^{H_0-h_s-c}(u)|\ |\ u\in \mathcal{M}^\Gamma(a,b)\right\rbrace & \leq \mathfrak{a},\\
\sup\left\lbrace\|\nabla^{J_s}\A^{H_0-h_s-c}(u)\|_{L^2(\R\times[0,1])\times\R}\ |\ u\in \mathcal{M}^\Gamma(a,b)\right\rbrace & \leq \mathfrak{e}.
\end{aligned}
\end{equation}

On the other hand, if we fix $\varepsilon>0$ then by Proposition \ref{prop:smallDerivSetBound} and Remark \ref{rem:smallDeriv} we know that the set $\mathcal{B}^\Gamma(\mathfrak{a},\mathfrak{y},\varepsilon)$, defined in \eqref{DefBGamma}, is bounded in the $L^\infty\times\R$ norm and the bounds depend only on $a,b,\varepsilon$ and $\Gamma$. Consequently there exists a compact subset $K_\varepsilon\subseteq T^*\R^2$, such that
\begin{equation}\label{DefKepsi}
 v([0,1])\subseteq K_\varepsilon\qquad\textrm{for all}\qquad (v,\eta)\in \B^\Gamma(\mathfrak{a},\mathfrak{y},\varepsilon).
\end{equation}
Without loss of generality, we assume that $\overline{\mathcal{V}} \subseteq K_\varepsilon$.

Let us fix a Floer trajectory $u=(v,\eta)\in \mathcal{M}^\Gamma(a,b)$ and denote a connected component 
\begin{equation}\label{Omega}
\Omega \subseteq \left(\mathbb{R} \times [0,1]\right)\setminus v^{-1}(K_\varepsilon).
\end{equation}
By assumption $\supp dh_s \subseteq \overline{V} \subseteq K_\varepsilon$, hence $dh_s \circ v |_{\Omega}\equiv 0$ for all $s\in [0,1]$. This means that instead of $H-h_s$ we can use the Hamiltonian $H_0-\mathbf{c}_s$ for some positive function $\mathbf{c}_s>0$ on $v(\Omega)$, which will simplify our computations. More precisely, for $(s,t)\in \Omega$ the Floer trajectory $(v,\eta)$ satisfies
\begin{equation}\label{FloerH0}
\partial_sv = \mathbb{J}(\partial_t v - \eta X_{H_0}).
\end{equation}
Naturally, $\mathfrak{c}_s\leq \sup_{s\in[0,1]}\|h_s\|_{L^\infty}$.

By the convergence of the integral
$$\|\nabla\A^{H_0-h_s}(u)\|_{L^2(\R\times[0,1])\times\R} 
\leq \|J\|_{L^\infty}\sqrt{\mathfrak{e}},$$
we know that $\lim_{s\to\pm\infty}u(s)\in \B^\Gamma(\mathfrak{a},\mathfrak{y},\varepsilon)$. Therefore, 
$$v(\partial\Omega)\subseteq K_\varepsilon\cup T^*_{q_0}\R^2\cup T^*_{q_1}\R^2$$
Consequently, we have $\sup_{\partial \Omega}|q(s,t)|\leq \sup_{K_\varepsilon}|q|$. Unfortunately, we do not have a uniform bound on $\sup_{\partial\Omega}|p|$. Therefore, we will need to treat the cases of the position and momenta coordinates separately. We introduce the following functions on $T^*\R^2$:
\begin{equation}\label{DefF}
F_1:= \frac{1}{2}|q|^2\qquad \textrm{and}\qquad F_2:=\frac{1}{2}|p|^2,
\end{equation}
Both of the functions $F_1$ and $F_2$ are plurisubharmonic, since their corresponding gradients $q\partial_q$ and $p\partial_p$ are Liouville vector fields (see Remark \ref{rem:pluri}).

In Lemma \ref{lem:LaplaceF} we will show that on $\Omega$ the functions $F_1\circ v$ and $F_2\circ v$ satisfy
\begin{align*}
\triangle (F_1 \circ v) &\geq -\frac{1}{2}\eta^2 F_1\circ v+ \partial_{s}\eta \left(F_2\circ v -H_0\circ v\right)=:f_1(s,t),\\
\triangle (F_2 \circ v)&\geq -\frac{1}{2}\eta^2F_1\circ v-\partial_s\eta\left(F_2\circ v+H_0\circ v\right)=:f_2(s,t).
\end{align*}
In Lemma \ref{lem:f1W11Bound} and Lemma \ref{lem:f2W11Bound} we will prove that $f_1$ and $f_2$, respectively, are bounded in the $W^{1,1}$-norm on $\Omega$
and the bounds depend only on $a,b,c,\varepsilon$ and $\Gamma$, but not on our choice of the Floer trajectory $u\in \mathcal{M}^\Gamma(a,b)$ or our choice of the connected component $\Omega\subseteq \left(\R\times[0,1] \right)\setminus v^{-1}(K_\varepsilon)$. By the continuity of the Sobolev embedding $W^{1,1}\hookrightarrow L^2$ we obtain $L^2$-bounds on $f_1$ and $f_2$.

Since we know that $\sup_{\partial \Omega}|q(s,t)|\leq \sup_{K_\varepsilon}|q|$ we can now apply the Alexandrov maximum principle to obtain
$$
\sup_{(s,t)\in\R\times[0,1]}F_1\circ v  \leq \sup_{K_\varepsilon}f_1\circ v+C\|f_1\|_{L^2(\Omega)},
$$
and the right-hand side of the inequality does not depend on our choice of the Floer trajectory $u\in \mathcal{M}^\Gamma(a,b)$ or our choice of the connected component $\Omega\subseteq \left(\R\times[0,1] \right)\setminus v^{-1}(K_\varepsilon)$. In other words, we obtain that the set
$$
\{\pi \circ v \ |\ u=(v,\eta)\in\mathcal{M}^\Gamma(a,b)\},
$$
is bounded in $\R^2$.

Now we would like to establish the $L^\infty$-bounds on the $p$-variable of the Floer trajectories. To establish uniform bounds on $F_2 \circ v$ for the Floer trajectory $u=(v,\eta)$ we extend the domain of the Floer trajectory $v:\R\times[0,1]\to T^*\R^2$ to a cylinder $\R\times \left([-1,1]/-1\sim 1\right)$ in the following way
\begin{equation}\label{barv}
\bar{v}(s,t):=\begin{cases}
v(s,-t) & \textrm{for}\quad t\leq 0,\\
v(s,t) & \textrm{for}\quad t\geq 0.
\end{cases}
\end{equation}
If we extend the domain $\Omega$ to the cylinder $\R\times \left([-1,1]/-1\sim 1\right)$ in the following way
\begin{align}
&\Theta:=\Omega\cup\{(s,-t)\ |\ (s,t)\in \Omega\},\label{Theta}\\
then \qquad &\bar{v}(\partial\Theta)\subseteq K_\varepsilon\qquad\textrm{and}\qquad  v(\Omega) =\bar{v}(\Theta). \nonumber
\end{align}
In particular, we have that $\sup_{\Omega}f\circ v=\sup_\Theta f\circ \bar{v}$ and $\|f\circ \bar{v}\|_{L^2(\Theta)}=2\|f\circ v\|_{L^2(\Omega)}$ for any smooth function $f\in T^*\R^2$. 

After proving in Lemma \ref{lem:F2smooth} that $F_2\circ \bar{v}$ is in $C^2(\Theta)$ we can apply the Aleksandrov maximum principle and obtain
\begin{align*}
\sup_{(s,t)\in\R\times[0,1]}F_2\circ v  & \leq \sup_{\partial\Theta}F_2\circ\bar{v}+C\|f_2\circ \bar{v}\|_{L^2(\Theta)}\\
& \leq \sup_{K_\varepsilon}f_2\circ v+2C\|f_2\circ v\|_{L^2(\Omega)}.
\end{align*}
Since the established bounds do not depend on the choice of $u\in\mathcal{M}^\Gamma(a,b)$ or the choice of the connected component $\Omega\subseteq \left(\R\times[0,1] \right)\setminus v^{-1}(K_\varepsilon)$, this concludes the proof that $\mathcal{M}^\Gamma(a,b)$ is bounded in the $L^\infty\times\R$-norm.
\hfill $\square$

\vspace*{.25cm}
To complete the proof of Theorem \ref{thm:FloerBounds} we need to prove the following technical lemmas:
\begin{lem}\label{lem:LaplaceF}
Consider the standard symplectic space $(T^*\R^2, \omega_0)$ with the Hamiltonian $H_0$ defined as in \eqref{DefH0} and functions $F_1$ and $F_2$ defined in \eqref{DefF}.
 Let $u:\R \to \mathscr{H}_{q_0,q_1}\times \R$, $u=(v,\eta)$ be a Floer trajectory corresponding to the action functional $\A^{H_0-c}$, i.e. satisfying the relation
\eqref{FloerH0}.
Then
\begin{align*}
\triangle (F_1 \circ v) & \geq -\frac{1}{2}\eta^2 F_1\circ v+ \partial_{s}\eta \left(F_2\circ v -H_0\circ v \right),\\
\triangle (F_2 \circ v) & \geq -\frac{1}{2}\eta^2 F_1\circ v-\partial_s\eta\left(F_2\circ v+H_0\circ v\right).
\end{align*}
\end{lem}
\begin{proof}
Using \eqref{DefH0} we can calculate that
\begin{align}
X_{H_0} & =  (p_1+q_2)\partial_{q_1}+(p_2-q_1)\partial_{q_2}+p_2\partial_{p_1}-p_1\partial_{p_2},\nonumber\\
dF_1 (X_{H_0}) & = q_1p_1+q_2p_2,\nonumber\\
d^\mathbb{C}(dF_1(X_{H_0})) & = qd^\mathbb{C}p+pd^\mathbb{C}q=pdp-qdq=dF_2-dF_1,\nonumber\\
dF_2 (X_{H_0}) & = 0,\label{dF2XH}\\
d^\C F_1 (X_{H_0}) & = qdp(X_{H_0})=q_1p_2-q_2p_1 = F_2-H_0 ,\label{dCF1XH}\\
d^\C F_2 (X_{H_0}) & =-\left(|p|^2 +q_1p_2-q_2p_1\right) = -F_2-H_0,\label{dCF2XH}\\
d(dF_1 (X_{H_0}))(X_{H_0}) & = |p|^2=2F_2,\nonumber\\
\big(dH_0+d(d^\C F_1 (X_{H_0}))&+d^\C \left( dF_1(X_{H_0})\right)\big)  = 2dF_2-dF_1,\nonumber\\
\big(dH_0+d(d^\C F_2 (X_{H_0}))&+d^\C \left( dF_2(X_{H_0})\right)\big)  = -dF_1.\nonumber
\end{align}
Plugging these qualities into \eqref{Laplace} for a Floer trajectory $u=(v,\eta)$ satisfying the relation \eqref{FloerH0} we can calculate
\begin{align*}
\triangle (F_1 \circ v) & =  \|\partial_{s}v\|^{2}+\eta \left(dH_0 +d(d^{\mathbb{C}}F_1(X_{H_0}))+d^\mathbb{C}(dF_1(X_{H_0}))\right)(\partial_{s}v)\\
& +\eta^2 d(dF_1(X_{H_0}))(X_{H_0})+\partial_{s}\eta\ d^{\mathbb{C}}F_1(X_{H_0})\\
& =  \|\partial_{s}v\|^{2}+\eta \left(2pdp-qdq\right)(\partial_{s}v)+\eta^2|p|^2+ \partial_{s}\eta \left(\frac{1}{2}|p|^2-H_0\circ v \right)\\
& \geq  -\eta^2 \frac{1}{4}|q|^2+ \partial_{s}\eta  \left(\frac{1}{2}|p|^2-H_0 \right)\\
& = -\frac{1}{2}F_1\circ v+  \partial_{s}\eta \left(F_2\circ v-H_0\circ v \right),\\
\triangle (F_2 \circ v) & = \|\partial_{s}v\|^{2}+\eta \left(dH_0 +d(d^{\mathbb{C}}F_2(X_{H_0}))+d^\mathbb{C}(dF_2(X_{H_0}))\right)(\partial_{s}v)\\
& +\eta^2 d(dF_2(X_{H_0}))(X_{H_0})+\partial_{s}\eta\ d^{\mathbb{C}}F_2(X_{H_0})\\
& =  \|\partial_{s}v\|^{2}-\eta qdq(\partial_sv)-\partial_s\eta\left( \frac{1}{2}|p|^2+H_0\right)\\
& \geq -\frac{1}{4}\eta^2|q|^2-\partial_s\eta\left(\frac{1}{2}|p|^2+H_0\right)\\
& = -\frac{1}{2}\eta^2F_1\circ v-\partial_s\eta\left(F_2\circ v+H_0\circ v\right)
\end{align*}
\end{proof}

\begin{lem}\label{lem:F2smooth}
Consider the standard symplectic space $(T^*\R^2, \omega_0)$ with the Hamiltonian $H_0$ defined as in \eqref{DefH0}. Let $u:\R \to \mathscr{H}_{q_0,q_1}\times \R$, $u=(v,\eta)$ be a Floer trajectory associated to $\A^{H_0-\mathbf{c}_s}$, i.e. satisfying the relation \eqref{FloerH0} and let $\bar{v}:\R\times([-1,1]/-1\sim 1) \to T^*\R^2$ be the extension of $v$ defined in \eqref{barv}. Then the composition $F_2\circ\bar{v}$ with the function $F_2:T^*\R^2\to\R$, $F(q,p)=\frac{1}{2}|p|^2$ is $C^2$.
\end{lem}
\begin{proof}
First, observe that $\bar{v}$ is continuous. Moreover, $\bar{v}$ is smooth on\linebreak $\R\times\left((-1,0)\cup(0,1)\right)$. Therefore, $F_2\circ \bar{v}$ is also everywhere continuous and smooth on $\R\times\left((-1,0)\cup(0,1)\right)$. What is left to check is that $F_2\circ \bar{v}$ is $C^2$ on $\R\times \{-1,0,1\}$.
Note that
$$
\frac{d}{dt} (F_2\circ \bar{v})(s,t)=\begin{cases}
dF_2( \partial_tv)(s,t) & \textrm{for}\quad t\geq 0,\\
-dF_2(\partial_tv)(s,-t) & \textrm{for}\quad t\leq 0.
\end{cases}
$$
Hence, the function $F_2\circ\bar{v}$ is $C^1$ if and only if 
\begin{equation}
dF_2( \partial_tv)(s,0)=dF_2( \partial_tv)(s,1)=0\qquad \forall\ s\in \R.
\label{pdpPartialtv}
\end{equation}
By assumption $v: \R \to \mathscr{H}_{q_0,q_1}$, so
$v(s,0)\equiv q_0$ and $v(s,1)\equiv q_1$ for all $s\in\R$. Consequently,
\begin{equation}
dq_i(\partial_sv)(s,0)=dq_i(\partial_sv)(s,1)=0\qquad \textrm{for}\quad i=0,1 \quad \textrm{and}\quad \forall\ s\in \R.
\label{dqPartialsv}
\end{equation}
Now using the Floer equation \eqref{FloerH0} for the Hamiltonian $H_0$ and the relation \eqref{dF2XH} we can calculate for $t\geq 0$
\begin{align*}
dF_2( \partial_tv) & = d^\mathbb{C}F_2(\partial_sv)+\eta dF_2(X_{H_0})=-pdq(\partial_sv).
\end{align*}
Plugging \eqref{dqPartialsv} into the equation above we obtain \eqref{pdpPartialtv}, which proves that $F_2\circ \bar{v}$ is $C^1$ everywhere.

To prove that $F_2\circ \bar{v}$ is $C^2$ we first calculate:
\begin{align*}
\frac{d^2}{dt^2}F_2\circ v & = D^2F_2(\partial_t v,\partial_t v)+dF_2(\partial_{tt}v)=  \eta^2 D^2F_2(X_{H_0},X_{H_0})+dF_2(\partial_{tt}v)\\ &=\eta^2|p|^2+dF_2(\partial_{tt}v).
\end{align*}
Obviously $|p|^2$ does not depend on the sign of $\partial_tv$. On the other hand, the function $\partial_{tt}\bar{v}= \partial_{tt}v$ is continuous, since the two negative signs cancel each other. Therefore, the sum of these two functions is also continuous on the whole cylinder $\R\times([-1,1]/-1\sim 1)$.

Now $\frac{d^2}{dsdt}F_2\circ \bar{v}$ is well defined and continuous on the whole $\R\times((-1,0)\cup(0,1))$. On the other hand by \eqref{pdpPartialtv} we obtain that $\frac{d^2}{dsdt}F_2\circ v (s,0)=\linebreak\frac{d^2}{dsdt}F_2\circ v (s,1)=0$, thus $\frac{d^2}{dsdt}F_2\circ \bar{v}$ is a continuous function on the whole cylinder $\R\times([-1,1]/-1\sim 1)$.

Finally, we observe that $\frac{d^2}{ds^2}(F_2\circ v (s,t))=\frac{d^2}{ds^2}(F_2\circ v (s,-t))$ as it does not depend on the sign of the $t$-variable.
\end{proof}

\begin{lem}
Let 
$\Omega \subseteq \left(\mathbb{R} \times [0,1]\right)\setminus v^{-1}(K_\varepsilon)$
be a connected component as defined in \eqref{Omega}. Then the following values are well-defined and finite:
\begin{equation}\label{Defs0s1}
s_0 :=\inf\{s\ |\ (s,t)\in \Omega\}\qquad \textrm{and}\qquad
s_1 :=\sup\{s\ |\ (s,t)\in \Omega\}.
\end{equation}
Moreover, $s_0$ and $s_1$ satisfy
\begin{equation}\label{s0s1Bound}
|s_1-s_0|\leq \frac{\mathfrak{e}\|J\|^2_\infty}{\varepsilon^2},
\end{equation}

\end{lem}
\begin{proof}
By Lemma \ref{lem:ActBound} we obtain uniform bounds on the action, energy and the $\eta$-parameter for all $u\in\mathcal{M}^\Gamma(a,b)$. By the convergence of the integral
$$\|\nabla\A^{H_0-h_s}(u)\|_{L^2(\R\times[0,1])\times\R} 
\leq \|J\|_{L^\infty}\sqrt{\mathfrak{e}},$$
we know that $\lim_{s\to\pm\infty}u(s)\in \B^\Gamma(\mathfrak{a},\mathfrak{y},\varepsilon)$. For the compact subset $K_\varepsilon\subseteq T^*\R^2$ defined in \eqref{DefKepsi} we have $v([0,1])\subseteq K_\varepsilon$ for all $(v,\eta)\in\mathcal{M}^\Gamma(a,b)$. Since\linebreak $\Omega \subseteq \left(\mathbb{R} \times [0,1]\right)\setminus v^{-1}(K_\varepsilon)$, thus the values $s_0$ and $s_1$ as in \eqref{Defs0s1} are well defined and finite. Moreover, for all $s\in (s_0,s_1)$, $u(s)\notin \mathcal{B}^\Gamma(\mathfrak{a},\mathfrak{y}, \varepsilon)$, therefore by \eqref{sBoundEpsi} we can estimate
$$
|s_1-s_0|\leq \frac{\mathfrak{e}\|J\|^2_\infty}{\varepsilon^2},
$$
where $\mathfrak{e}=\|J\|_\infty(b-a+\frac{c}{3}\mathfrak{y})$ as in Lemma \ref{lem:ActBound}. 
\end{proof}

\begin{lem}\label{lem:f1L1Bound}
Consider the setting as in Theorem \ref{thm:FloerBounds}. Fix $a,b\in \R$ and let  $u\in \mathcal{M}^\Gamma(a,b)$ be a Floer trajectory. Let 
$\Omega \subseteq \left(\mathbb{R} \times [0,1]\right)\setminus v^{-1}(K_\varepsilon)$
be a connected component as defined in \eqref{Omega}. Then the function
$$
f_1(s,t):= -\frac{1}{2}\eta^2(s) F_1 \circ v(s,t) + \partial_{s}\eta(s)  \left(F_2-H_0\right)\circ v(s,t).
$$
is bounded in the $L^1(\Omega)$ norm and the bounds do not depend on $u$ or $\Omega$.
\end{lem}
\begin{proof}
Let $s_0$ and $s_1$ be as in \eqref{Defs0s1}.
Denote $\overline{\Omega}:= [s_0,s_1]\times[0,1]$. Naturally,\linebreak $\Omega\subseteq \overline{\Omega}$. Before we estimate $\|f_1\|_{L^1(\Omega)}$ we first use the Cauchy-Schwartz inequality together with \eqref{DefE}, \eqref{qL2Bound}, \eqref{pL2Bound} and \eqref{s0s1Bound} to do the following estimates:
\begin{align}
\|\partial_s\eta\|_{L^1(\Omega)}& \leq \|\partial_s\eta\|_{L^1(\overline{\Omega})}\leq \sqrt{|s_1-s_0|}\|\nabla^{J_s} \A^{H_0-h_s}(u(s))\|_{L^2\times\R}\nonumber\\
&\leq \frac{\mathfrak{e}\|J\|_\infty}{\varepsilon}\label{partialEtaBound}\\
\|\partial_s\eta(H_0-\mathbf{c}_s)\|_{L^1(\Omega)}&\leq
\|\partial_s\eta (H_0-h_s)\circ v\|_{L^1(\overline{\Omega})}=\|\partial_s\eta\|_{L^2(\overline{\Omega})}^2\nonumber\\
& \leq \|\nabla^{J_s} \A^{H_0-h_s}(u(s))\|_{L^2\times\R}^2\leq \mathfrak{e}\label{H0Bound}\\
\|q\|_{L^2(\Omega)} & \leq \|q \|_{L^2(\overline{\Omega})}\leq \sqrt{|s_1-s_0|} \sup_{s\in\R}\|q(s)\|_{L^2([0,1])}\nonumber\\
& \leq \frac{\sqrt{\mathfrak{e}}\|J\|_\infty}{\varepsilon}\left(\mathfrak{q}+\frac{\mathfrak{e}\|J\|_\infty}{\varepsilon}\right)\label{qOmegaBound}\\
\|p\|_{L^2(\Omega)} & \leq \frac{\sqrt{\mathfrak{e}}\|J\|_\infty}{\varepsilon}\left(\mathfrak{p}+\frac{\mathfrak{e}\|J\|_\infty}{\varepsilon}\right)\label{pOmegaBound}\\
\|\partial_s\eta\ p^2\|_{L^1(\Omega)} & \leq  \|\partial_s\eta\|_{L^1(\overline{\Omega})}\sup_{s\in\R}\|p(s)\|_{L^2([0,1])}^2 \leq\frac{\mathfrak{e}\|J\|_\infty}{\varepsilon}\left(\mathfrak{p}+\frac{\mathfrak{e}\|J\|_\infty}{\varepsilon}\right)^2 \label{partialEtaPBound}
\end{align}
Using the bounds obtained above and the inequality \eqref{eqEta} we can calculate:
\begin{align*}
\|f_1\|_{L^1(\Omega)}& \leq \frac{1}{4}\mathfrak{y}^2 \|q\|_{L^2(\Omega)}^2+\frac{1}{2}\|\partial_s\eta\ p^2\|_{L^1(\Omega)}+\|\partial_s\eta(H_0-\mathbf{c}_s)\|_{L^1(\Omega)}+\|\mathbf{c}_s\partial_s\eta\|_{L^1(\Omega)}\\
&\leq \frac{\mathfrak{e}\mathfrak{y}^2\|J\|_\infty^2}{4\varepsilon^2}\left(\mathfrak{q}+\frac{\mathfrak{e}\|J\|_\infty}{\varepsilon}\right)^2+\frac{\mathfrak{e}\|J\|_\infty^2}{2\varepsilon^2}\left( \mathfrak{p}+\frac{\mathfrak{e}\|J\|_\infty}{\varepsilon}\right)^2+\mathfrak{e}\left(\frac{\mathfrak{h}\|J\|_\infty}{\varepsilon}+1\right).
\end{align*}
\end{proof}

\begin{lem}\label{lem:f1W11Bound}
Consider the setting as in Lemma \ref{lem:f1L1Bound}. Then the function $f_1$ is bounded in the $W^{1,1}$-norm and the bounds do not depend on $u$ or $\Omega$.
\end{lem}

\begin{proof}
To prove that $f_1\in W^{1,1}(\Omega)$ we first need to calculate its derivatives:
\begin{align*}
\partial_sf_1 & = - \eta\partial_s\eta\ F_1+\partial_{ss}\eta\ d^\mathbb{C}F_1(X_{H_0})+\left(\partial_s\eta\ d(F_2-H_0)-\frac{1}{2}\eta^2 dF_1\right)(\partial_sv)\\
\partial_tf_1 & = \left(\partial_s\eta\ d(F_2-H_0)-\frac{1}{2}\eta^2 dF_1\right)(\mathbb{J}\partial_sv+\eta X_{H_0})\\
& = \left(\partial_s\eta\ d^\mathbb{C}(F_2-H_0)-\frac{1}{2}\eta^2 d^\mathbb{C} F_1\right)(\partial_sv)-\frac{1}{2}\eta^3dF_1(X_{H_0}),
\end{align*}
where the last equality comes from the fact that $dF_2(X_{H_0})(q,p)=0$.
Using \eqref{qL2Bound}, \eqref{vL2Bound} and \eqref{s0s1Bound} we can estimate:
\begin{align}
\|dF_1(X_{H_0})\|_{L^1(\Omega)}& = \|q_1p_1+q_2p_2\|_{L^1(\Omega)}\leq\frac{1}{2}\| v\|_{L^2(\Omega)}^2\leq \frac{1}{2}|s_1-s_0|\sup_{s\in\R}\|v(s)\|_{L^2([0,1])}^2\nonumber\\
& \leq \frac{\mathfrak{e}\|J\|_\infty^2}{2\varepsilon^2}\left( \mathfrak{q}+\mathfrak{p}+\frac{\mathfrak{e}\|J\|_\infty}{\varepsilon}\right)^2\label{eq0}\\
\|dF_1(\partial_sv)\|_{L^1(\Omega)}& \leq \|\nabla F_1\|_{L^2(\Omega)}\|\partial_sv\|_{L^1(\Omega)}\leq \|q\|_{L^2(\overline{\Omega})}\|\nabla^{J_s}\A^{H_0-h_s}(u)\|_{L^2\times\R}\nonumber\\
& \leq \frac{\mathfrak{e}\|J\|_\infty}{\varepsilon}\left(\mathfrak{q}+\frac{\mathfrak{e}\|J\|_\infty}{\varepsilon}\right)\label{eq1}
\end{align}
\begin{align}
\|&\partial_s\eta\ d(F_2 -H_0)(\partial_sv)\|_{L^1(\Omega)}  \leq  \int_{s_0}^{s_1}|\partial_s\eta|\int_0^1|\nabla (F_2-H_0)||\partial_sv|dtds\nonumber\\
 &= \int_{s_0}^{s_1}|\partial_s\eta|\int_0^1|v||\partial_sv|dtds \leq \sup_{s\in\R}\|v(s)\|_{L^2([0,1])}\int_{s_0}^{s_1}|\partial_s\eta|\|\partial_sv(s)\|_{L^2([0,1])}ds\nonumber\\
 & \leq\sup_{s\in\R}\|v(s)\|_{L^2([0,1])}\|\nabla^{J_s}\A^{H_0-h_s}(u)\|_{L^2\times\R}^2 \leq \mathfrak{e}\left( \mathfrak{q}+\mathfrak{p}+\frac{\mathfrak{e}\|J\|_\infty}{\varepsilon}\right)\label{eq2}
\end{align}
Similarly, we have
\begin{align}
\|d^\C F_1(\partial_sv)\|_{L^1(\Omega)} & \leq  \frac{\mathfrak{e}\|J\|_\infty}{\varepsilon}\left(\mathfrak{q}+\frac{\mathfrak{e}\|J\|_\infty}{\varepsilon}\right),\label{eq3}\\
\|\partial_s\eta\ d^\C(F_2-H_0)(\partial_sv)\|_{L^1(\Omega)} &\leq \mathfrak{e}\left( \mathfrak{q}+\mathfrak{p}+\frac{\mathfrak{e}\|J\|_\infty}{\varepsilon}\right).
\end{align}
Using the estimates above and \eqref{eqEta} we can combine them to obtain a uniform $L^1$-bound on $\partial_tf_1$:
\begin{align*}
\|\partial_tf_1\|_{L^1(\Omega)} &\leq \|\partial_s\eta d^\C( F_2-H_0)(\partial_sv)\|_{L^1(\Omega)}+\frac{1}{2}\|\eta^2 d^\C F_1(\partial_sv)\|_{L^1(\Omega)}+\frac{1}{2}\|\eta^3 dF_1(X_{H_0})\|_{L^1(\Omega)}\\
& \leq  \mathfrak{e}\mathfrak{p}+ \mathfrak{e}\left(1+\frac{\mathfrak{ey}^2\|J\|_\infty}{2\varepsilon}\right)\left(\mathfrak{q}+\frac{\mathfrak{e}\|J\|_\infty}{\varepsilon}\right)+\frac{\mathfrak{ey}^3\|J\|_\infty^2}{4\varepsilon^2}\left( \mathfrak{q}+\mathfrak{p}+\frac{\mathfrak{e}\|J\|_\infty}{\varepsilon}\right)^2
\end{align*}
This gives us a bound on $\|\partial_tf_1\|_{L^1(\Omega)}$, which does not depend on the choice of the Floer trajectory $u$ or the connected component $\Omega$.

To estimate $\|\partial_sf_1\|_{L^1(\Omega)}$ we first use \eqref{qL2Bound} and \eqref{partialEtaBound} to estimate
\begin{align}
\|\partial_s\eta\ F_1\circ v\|_{L^1(\Omega)} & \leq \|\partial_s\eta\ F_1\circ v\|_{L^1(\overline{\Omega})} = \frac{1}{2}\int_{s_0}^{s_1}|\partial_s\eta|\|q(s)\|_{L^2([0,1])}^2ds\nonumber\\
& \leq \frac{1}{2}\|\partial_s\eta\|_{L^1(\overline{\Omega})}\sup_{s\in\R}\|q(s)\|_{L^2([0,1])}^2\nonumber\\
&\leq \frac{\mathfrak{e}\|J\|_\infty}{2\varepsilon}\left(\mathfrak{q}+\frac{\mathfrak{e}\|J\|_\infty}{\varepsilon}\right)^2\label{partialsEtaF1}
\end{align}
Secondly, we analyze $\partial_{ss}\eta$:
$$
\partial_{ss}\eta = -\int_0^1d(H_0-h_s)(\partial_sv)dt-\int_0^1 \partial_sh_s (v)dt.
$$
Using \eqref{inqGamma}, \eqref{pL2Bound} and \eqref{vL2Bound} we can estimate:
\begin{align}
\|\partial_{ss}\eta\|_{L^1(\Omega)} & = \int_{\Omega}\bigg|\int_0^1\left(d(H_0-h_s)(\partial_sv)+\partial_sh_s(v)\right)dt\bigg|dt ds\nonumber\\
& \leq \|d(H_0-h_s)(\partial_sv)\|_{L^1(\overline{\Omega})}+\|\partial_sh_s\|_{L^\infty}\nonumber\\
&\leq \left(\|\nabla H_0\|_{L^2(\overline{\Omega})}+\|\nabla h_s\|_{L^2(\overline{\Omega})}\right)\|\partial_sv\|_{L^2(\overline{\Omega})}\|\partial_sh_s\|_{L^\infty}\nonumber\\
& \leq \left(\|v+p\|_{L^2(\overline{\Omega})}+\|\nabla h_s\|_{L^2(\overline{\Omega})}\right)\|\nabla^{J_s}\A^{H_0-h_s}(u)\|_{L^2\times\R}+\|\partial_sh_s\|_{L^\infty}\nonumber\\
&\leq \sqrt{\mathfrak{e}|s_1-s_0|}\left(\sup_{s\in\R}\|v(s)\|_{L^2([0,1])}+\sup_{s\in\R}\|p(s)\|_{L^2([0,1])}+\sup_{s\in\R}\|\nabla h_s\|_{L^\infty}\right)+\|\partial_sh_s\|_{L^\infty}\nonumber\\
&\leq \frac{\mathfrak{e}\|J\|_\infty}{\varepsilon}\left(\mathfrak{h}+ \mathfrak{q}+2\mathfrak{p}+\frac{2\mathfrak{e}\|J\|_\infty}{\varepsilon}\right)+\frac{c}{3}\label{partialSS}
\end{align}
Now we will use the Floer equations, the Cauchy-Schwartz inequality together with \eqref{DefE}, \eqref{sBoundEpsi}, \eqref{qL2Bound}, \eqref{pL2Bound} and \eqref{dCF1XH} to calculate the following:
\begin{align}
\|\partial_{ss}\eta\ d^\mathbb{C}F_1(X_{H_0})\|_{L^1(\Omega)} & \leq \|\partial_{ss}\eta\ d^\mathbb{C}F_1(X_{H_0})\|_{L^1(\overline{\Omega})} =\int_{s_0}^{s_1}|\partial_{ss}\eta|\int_0^1|q_1p_2-q_2p_1|dtds\nonumber\\
& \leq \frac{1}{2}\int_{s_0}^{s_1}|\partial_{ss}\eta|(\|q(s)\|_{L^2([0,1])}^2+\|p(s)\|_{L^2([0,1])}^2)ds\nonumber\\
&\leq \frac{1}{2}\|\partial_{ss}\eta\|_{L^1(\overline{\Omega})}\sup_{s\in\R}\|v(s)\|_{L^2([0,1])}^2\nonumber\\
& \leq \left(\frac{\mathfrak{e}\|J\|_\infty}{\varepsilon}\left( \frac{1}{2}(\mathfrak{h}+\mathfrak{q})+\mathfrak{p}+\frac{\mathfrak{e}\|J\|_\infty}{\varepsilon}\right)+\frac{c}{3}\right)\left(\mathfrak{q}+\mathfrak{p}+\frac{\mathfrak{e}\|J\|_\infty}{\varepsilon}\right)^2\label{EtaSSdCF1XH0}
\end{align}
Using the estimates above together with \eqref{eqEta}, \eqref{eq1}, \eqref{eq2} and \eqref{partialsEtaF1} we can estimate:
\begin{align*}
\|\partial_sf_1\|_{L^1(\Omega)} & \leq \mathfrak{y}\|\partial_s\eta\ F_1\circ v\|_{L^1(\Omega)}+\|\partial_{ss}\eta d^\mathbb{C}F_1(X_{H_0})\|_{L^1(\Omega)}+\frac{1}{2}\mathfrak{y}^2\|d F_1(\partial_sv)\|_{L^1(\Omega)}\\
& +\|\partial_s\eta\ d(F_2-H)(\partial_sv)\|_{L^1(\Omega)} \\
& \leq \frac{\mathfrak{e y}\|J\|_\infty}{2\varepsilon}\left(\mathfrak{q}+\frac{\mathfrak{e}\|J\|_\infty}{\varepsilon}\right)^2 + \frac{\mathfrak{e y}^2\|J\|_\infty}{2\varepsilon}\left(\mathfrak{q}+\frac{\mathfrak{e}\|J\|_\infty}{\varepsilon}\right)\\
& + \left(\frac{\mathfrak{e}\|J\|_\infty}{\varepsilon}\left( \frac{1}{2}(\mathfrak{h}+\mathfrak{q})+\mathfrak{p}+\frac{\mathfrak{e}\|J\|_\infty}{\varepsilon}\right)+\frac{c}{3}\right)\left(\mathfrak{q}+\mathfrak{p}+\frac{\mathfrak{e}\|J\|_\infty}{\varepsilon}\right)^2\\
& +\mathfrak{e}\left( \mathfrak{q}+\mathfrak{p}+\frac{\mathfrak{e}\|J\|_\infty}{\varepsilon}\right)
\end{align*}
This gives us a bound on $\|\partial_sf_1\|_{L^1(\Omega)}$, which does not depend on the choice of the Floer trajectory $u$ or the connected component $\Omega$.
\end{proof}

\begin{lem}\label{lem:f2L1Bound}
Consider the setting as in Theorem \ref{thm:FloerBounds}. Fix $a,b\in \R$ and let  $u\in \mathcal{M}^\Gamma(a,b)$ be a Floer trajectory. Let 
$\Omega \subseteq \left(\mathbb{R} \times [0,1]\right)\setminus v^{-1}(K_\varepsilon)$
be a connected component as defined in \eqref{Omega}. Then the function
$$
f_2(s,t):= -\frac{1}{2}\eta^2(s)F_1\circ v(s,t)-\partial_s\eta(s)\left(F_2+H_0\right)\circ v(s,t).
$$
is bounded in the $L^1(\Omega)$ norm and the bounds do not depend on $u$ or $\Omega$.
\end{lem}
\begin{proof}
To estimate $\|f_2\|_{L^1(\Omega)}$ we use the bounds obtained in \eqref{eqEta}, \eqref{partialEtaBound}, \eqref{H0Bound}, \eqref{qOmegaBound} and \eqref{partialEtaPBound} 
to calculate:
\begin{align*}
\|f_2\|_{L^1(\Omega)} & \leq \frac{\mathfrak{y}^2}{4}\|q\|^2_{L^2(\Omega)} +
\frac{1}{2}\|\partial_s\eta\ p^2\|_{L^2(\Omega)}+\|\partial_s\eta (H_0-\mathfrak{c}_s)\|_{L^2(\Omega)}+\|\mathfrak{c}_s\partial_s\eta\|_{L^1(\Omega)}\\
&\leq \frac{\mathfrak{ey}^2\|J\|_\infty^2}{4\varepsilon^2}\left(\mathfrak{q}+\frac{\mathfrak{e}\|J\|_\infty}{\varepsilon}\right)^2+ \frac{\mathfrak{e}}{2\varepsilon}\left(\mathfrak{p}+\frac{\mathfrak{e}\|J\|_\infty}{\varepsilon}\right)^2+\mathfrak{e}\left(\frac{c}{\varepsilon}+\|J\|_\infty^2\right).
\end{align*}
\end{proof}

\begin{lem}\label{lem:f2W11Bound}
Consider the setting as in Lemma \ref{lem:f2L1Bound}. Then the function $f_2$ is bounded in the $W^{1,1}$-norm and the bounds do not depend on $u$ or $\Omega$.
\end{lem}

\begin{proof}
To prove that $f_2\in W^{1,1}(\Omega)$ we first need to calculate its derivatives:
\begin{align*}
\partial_sf_2 & = - \eta\partial_s\eta F_1 + \partial_{ss}\eta\ d^\mathbb{C}F_2(X_{H_0})-\left(\partial_s\eta\ d(F_2+H_0)+\frac{1}{2}\eta^2 dF_1\right)(\partial_sv)\\
\partial_tf_2 & = -\left(\partial_s\eta\ d(F_2+H_0)+\frac{1}{2}\eta^2 dF_1\right)(\mathbb{J}\partial_sv+\eta X_{H_0})\\
& = -\left(\partial_s\eta\ d^\mathbb{C}(F_2+H_0)+\frac{1}{2}\eta^2 d^\mathbb{C} F_1\right)(\partial_sv)-\frac{1}{2}\eta^3 dF_1(X_{H_0}),
\end{align*}
where the last equality comes from \eqref{dF2XH}.

Let $s_0$ and $s_1$ be as in \eqref{Defs0s1}.
Denote $\overline{\Omega}:= [s_0,s_1]\times[0,1]$. Naturally, $\Omega\subseteq \overline{\Omega}$.
Before we estimate $\|\partial_t f_2\|_{L^1(\Omega)}$ we will first
recall \eqref{dCF1XH} and \eqref{dCF2XH} which give us the following relation
\begin{equation}
d^\mathbb{C}F_1(X_{H_0})=F_2-H_0=2F_2+d^\mathbb{C}F_2(X_{H_0}).
\label{F2-H0}
\end{equation}
Further on, we use \eqref{DefE}, \eqref{pL2Bound}, \eqref{pOmegaBound} and \eqref{eq2} to calculate the following bounds:
\begin{align*}
\|\partial_s\eta\ dF_2(\partial_sv)\|_{L^1(\Omega)}& \leq \|\partial_s\eta\ dF_2(\partial_sv)\|_{L^1(\overline{\Omega})}\\
& \leq  \int_{s_0}^{s_1}|\partial_s\eta|\int_0^1|p(s,t)||\partial_sv(p,s)|dtds\\
&\leq \int_{s_0}^{s_1}|\partial_s\eta|\|p(s)\|_{L^2([0,1])}\|\partial_s v(s)\||_{L^2([0,1])}ds\\
&\leq \sup_{s\in \R}\|p(s)\|_{L^2([0,1])}\|\partial_s\eta\|_{L^2(\R)}\|\partial_sv\|_{L^2(\R\times[0,1])}\\
&\leq \mathfrak{e}\left(\mathfrak{p}+\frac{\mathfrak{e}\|J\|_\infty}{\varepsilon}\right).
\end{align*}
This, together with \eqref{F2-H0} gives us:
\begin{align}
\|\partial_s\eta\ d(F_2+H_0)(\partial_sv)\|_{L^1(\Omega)} & \leq 2\|\partial_s\eta\ dF_2(\partial_sv)\|_{L^1(\Omega)}+\|\partial_s\eta\ d(F_2-H_0)(\partial_sv)\|_{L^1(\Omega)}\nonumber\\
&\leq \mathfrak{e}\left( \mathfrak{q}+3\mathfrak{p}+2\frac{\mathfrak{e}\|J\|_\infty}{\varepsilon}\right)\label{eq4}
\end{align}
Similarly, we have
$$
\|\partial_s\eta\ d^\mathbb{C}(F_2+H_0)(\partial_sv)\|_{L^1(\Omega)}\leq \mathfrak{e}\left( \mathfrak{q}+3\mathfrak{p}+2\frac{\mathfrak{e}\|J\|_\infty}{\varepsilon}\right).
$$
Using the bounds obtained above together with \eqref{eqEta}, \eqref{eq0} and \eqref{eq3} we can estimate:
\begin{align*}
\|\partial_tf_2\|_{L^1(\Omega)} & \leq \|\partial_s\eta\ d^\mathbb{C}(F_2+H_0)(\partial_s v)\|_{L^1(\Omega)}+\frac{1}{2}\mathfrak{y}^2 \left(\|d^\mathbb{C}  F_1(\partial_sv)\|_{L^1(\Omega)}+\mathfrak{y} \|dF_1(X_{H_0})\|_{L^1(\Omega)}\right)\\
&\leq  \mathfrak{e}\left( \mathfrak{q}+3\mathfrak{p}+2\frac{\mathfrak{e}\|J\|_\infty}{\varepsilon}\right)+\frac{\mathfrak{ey}^2\|J\|_\infty}{2\varepsilon}\left(\mathfrak{q}+\frac{\mathfrak{e}\|J\|_\infty}{\varepsilon}\right)\\
&+\frac{\mathfrak{ey}^3\|J\|_\infty^2}{4\varepsilon^2}\left( \mathfrak{q}+\mathfrak{p}+\frac{\mathfrak{e}\|J\|_\infty}{\varepsilon}\right)^2.
\end{align*}

Before we estimate $\|\partial_sf_2\|_{L^2}$ we first use \eqref{pL2Bound}, \eqref{partialSS} and \eqref{EtaSSdCF1XH0}
to calculate the following bounds:
\begin{align*}
\|\partial_{ss}\eta\ F_2\circ v \|_{L^1(\Omega)} & \leq \|\partial_{ss}\eta\ F_2\circ v \|_{L^1(\overline{\Omega)}}=\frac{1}{2}\int_{s_0}^{s_1}|\partial_{ss}\eta|\int_0^1|p(s,t)|^2dtds\\
&\leq \frac{1}{2}\|\partial_{ss}\eta\|_{L^1(\overline{\Omega})}\sup_{s\in\R}\|p(s)\|_{L^2([0,1])}^2\\
& \leq  \left(\frac{\mathfrak{e}\|J\|_\infty}{\varepsilon}\left(\frac{1}{2}(\mathfrak{h}+ \mathfrak{q})+\mathfrak{p}+\frac{\mathfrak{e}\|J\|_\infty}{\varepsilon}\right)+\frac{c}{3}\right)\left(\mathfrak{p}+\frac{\mathfrak{e}\|J\|_\infty}{\varepsilon}\right)^2\\
\|\partial_{ss}\eta\ d^\mathbb{C}F_2(X_{H_0})\|_{L^1(\Omega)} & = \|\partial_{ss}\eta\left(d^\mathbb{C}F_1(X_{H_0})-2F_2\right)\|_{L^1(\Omega)}\\
& \leq \|\partial_{ss}\eta\ d^\mathbb{C}F_1(X_{H_0})\|_{L^1(\Omega)}+2\|\partial_{ss}\eta\ F_2\circ v\|_{L^1(\Omega)}\\
& \leq \left(\frac{\mathfrak{e}\|J\|_\infty}{\varepsilon}\left( \frac{1}{2}(\mathfrak{h}+\mathfrak{q})+\mathfrak{p}+\frac{\mathfrak{e}\|J\|_\infty}{\varepsilon}\right)+\frac{c}{3}\right)\\
&\cdot \left(\left(\mathfrak{q}+\mathfrak{p}+\frac{\mathfrak{e}\|J\|_\infty}{\varepsilon}\right)^2+2\left(\mathfrak{p}+\frac{\mathfrak{e}\|J\|_\infty}{\varepsilon}\right)^2\right)
\end{align*}
Finally, using the estimates above together with \eqref{eqEta}, \eqref{partialEtaPBound}, \eqref{eq2}, \eqref{partialsEtaF1}  and \eqref{eq4} we can calculate.
\begin{align*}
\|\partial_sf_2\|_{L^1(\Omega)} & \leq \mathfrak{y}\|\partial_s\eta\ F_1\|_{L^1(\Omega)}+\|\partial_s\eta\ d(F_2+H_0)(\partial_sv)\|_{L^1(\Omega)}\\
& +\frac{\mathfrak{y}^2}{2} \|dF_1(\partial_sv)\|_{L^1(\Omega)}+\|\partial_{ss}\eta\left(F_2+H_0\right)\|_{L^1(\Omega)}\\
 & \leq \frac{\mathfrak{ey}\|J\|_\infty}{2\varepsilon}\left(\mathfrak{q}+\frac{\mathfrak{e}\|J\|_\infty}{\varepsilon}\right)^2+\mathfrak{e}\left( \mathfrak{q}+3\mathfrak{p}+2\frac{\mathfrak{e}\|J\|_\infty}{\varepsilon}\right)\\
 &+\frac{\mathfrak{ey}^2}{2}\left( \mathfrak{q}+\mathfrak{p}+\frac{\mathfrak{e}\|J\|_\infty}{\varepsilon}\right)\\
& +\left(\frac{\mathfrak{e}\|J\|_\infty}{\varepsilon}\left( \frac{1}{2}(\mathfrak{h}+\mathfrak{q})+\mathfrak{p}+\frac{\mathfrak{e}\|J\|_\infty}{\varepsilon}\right)+\frac{c}{3}\right)\\
&\cdot \left(\left(\mathfrak{q}+\mathfrak{p}+\frac{\mathfrak{e}\|J\|_\infty}{\varepsilon}\right)^2+2\left(\mathfrak{p}+\frac{\mathfrak{e}\|J\|_\infty}{\varepsilon}\right)^2\right)
\end{align*}
This concludes the proof that $f_2$ is bounded in the $W^{1,1}$-norm and the bounds do not depend on the choice of the Floer trajectory $u\in \mathcal{M}^\Gamma(a,b)$ or the connected component $\Omega\subseteq \R\times[0,1]\setminus v^{-1}(K_\varepsilon)$, but only on the constants $a,b, \varepsilon$ and the smooth homotopy $\Gamma$.
\end{proof}

\subsection{The Lagrangian Rabinowitz Floer homology of the Copernican Hamiltonian}

Having proven the $L^\infty$-bounds on the Floer trajectories we are finally ready to prove the main theorem of this paper:

\vspace*{.25cm}
\textit{Proof of Theorem \ref{thm:LRFH}:}
Let $H_0:T^*\R^2\to\R$ be the Copernican Hamiltonian defined in \eqref{DefH0} and let $\mathcal{H}$ be the set of perturbations as defined in \eqref{DefHset}.
Fix a pair $q_0,q_1\in \R^2$. 
Let $\{K_{n,m}\}_{n,m\in\mathbb{N}}$ be the sequence of compact sets
as in Lemma \ref{lem:H0Chord}, such that 
\begin{equation}\label{Knm}
\begin{aligned}
K_{n,m}\subseteq K_{n+1,m}\quad &\textrm{and} \quad \bigcup_{n\in\mathbb{N}}K_{n,m}=T^*\R^2\quad &&\forall\ m\in \mathbb{N},\\
K_{n,m}\subseteq K_{n,m+1} \quad &\textrm{and} \quad T_{q_0}^*\R^2, T_{q_1}^*\R^2\subseteq \bigcup_{m\in\mathbb{N}}K_{n,m}\quad  &&\forall\ n\in\mathbb{N}.
\end{aligned}
\end{equation}

Fix a Hamiltonian $h_0\in \mathcal{H}$ and denote $c(h_0):=\inf (h-dh(p\partial_p))$. Naturally, $c(h_0)>0$.
Let $K_{n,m}$ be any of the compact sets in the sequence, such that $\supp dh_0\subseteq K_{n,m}$ and $\|h_0\|_{L^\infty}<m$. Denote
$$
\mathcal{H}_{h_0}(K_{n,m}):=\left\lbrace\begin{array}{c|c}
& dh \in C_0^\infty(K_{n,m}), \quad \|h\|_{L^\infty}<\frac{c(h_0)}{50},\\
h \in C^\infty(T^*\R^2) &  h_0+h \in \mathcal{H}, \quad \|h_0+h\|_{L^\infty}<m,\\
& \inf((h+h_0)-d(h+h_0)(p\partial_p))>\frac{1}{2} c(h_0).
\end{array}\right\rbrace.
$$
This way $\mathcal{H}_{h_0}(K_{n,m})$ is an open neighbourhood of $0$ in $\{h\in C^\infty(T^*\R^2)\ |\ dh\in C_0^\infty(K_{n,m})\}$. By Lemma \ref{lem:H0Chord} the critical set of the action functional $\A^{H-h_0}_{q_0,q_1}$ is continuously compact in $(K_{n,m}, \mathcal{H}_{h_0}(K_{n,m}))$. On the other hand, by construction if $h\in \mathcal{H}_{h_0}(K_{n,m})$, then $h_0+h\in \mathcal{H}$, hence by Lemma \ref{lem:nonEmpty} every Hamiltonian $H_0-h_0-h$ satisfies the condition \eqref{nonEmpty}. This proves that the compact set $K_{n,m}$ and the set of perturbations $\mathcal{H}_{h_0}(K_{n,m})$ satisfy assumptions 1. and 2. of Theorem \ref{thm:DefLRFH} for the Hamiltonian $H_0-h_0$.

Now, if we take any open, pre-compact set $\mathcal{V}\subseteq T^*\R^2$, such that $K_{n,m}\subseteq \mathcal{V}$, any two almost complex structures $J_0,J_1\in \mathcal{J}(\mathcal{V},\mathbb{J})$ and any $h\in\mathcal{H}_{h_0}(K_{n,m})$, then every smooth homotopy $\Gamma:=\{(h_0+h,J_s)\}_{s\in\R}$, with $J_s\in C^\infty([0,1]\times\R,\mathcal{J}(\mathcal{V},\mathbb{J}))$ such that
\begin{equation}\label{AC}
J_s=\begin{cases}
 J_0 & \textrm{for}\quad s\leq 0,\\
 J_1 & \textrm{for}\quad s\geq 1.
\end{cases}
\end{equation}
will automatically satisfy condition \eqref{inqGamma}. Consequently, by Theorem \ref{thm:FloerBounds} for every pair $(a,b)\in \R^2$ the corresponding space $\mathcal{M}^\Gamma(a,b)$ of Floer trajectories is bounded in $L^\infty$-norm. Therefore, we can apply Theorem \ref{thm:DefLRFH} and conclude that there exists a residual set {$\mathcal{H}^{\reg}_{h_0}(K_{n,m})\subseteq \mathcal{H}_{h_0}(K_{n,m})$}, such that\linebreak for every $h\in \mathcal{H}^{\reg}_{h_0}(K_{n,m})$ the Lagrangian Rabinowitz Floer homology\linebreak $\LRFH_*(\A^{H_0-h_0-h}_{q_0,q_1})$ is well defined and independent of the choice of the almost complex structure.

Now if we fix $h_1,h_2\in \mathcal{H}^{\reg}_{h_0}(K_{n,m})$ then we can construct a homotopy $\Gamma$ satisfying \eqref{inqGamma}. Indeed, let $J_0\in \mathcal{J}^{\reg}_{h_0}$ and $J_1\in \mathcal{J}^{\reg}_{h_2}$ will be two almost complex structures close enough, so that there exists a smooth homotopy $\{J_s\}_{s\in\R}$ with $J_s\in C^\infty([0,1]\times\R,\mathcal{J}(\mathcal{V},\mathbb{J}))$ satisfying \eqref{AC} and $\sup_{s\in\R}\|J_s\|_{L^\infty}<\frac{1}{\sqrt{6c(h_0)}}$. Moreover, let $\chi\in C^\infty(\R)$ be a smooth function satisfying $\|\chi'\|_{L^\infty} \leq 2$ and
\begin{equation}\label{chi}
\chi(s)=\begin{cases}
 0 & \textrm{for}\quad s\leq 0,\\
 1 & \textrm{for}\quad s\geq 1.
\end{cases}
\end{equation}
Define $h_s:=h_0+h_1(1-\chi(s))+\chi(s)h_2$. Then the smooth homotopy $\Gamma=\{(h_s,J_s)\}_{s\in\R}$ satisfies
\begin{align*}
\|\partial_sh_s\|\left(\frac{4}{c(h_0)}+\|J\|_{L^\infty}^2\right) &\leq \|\chi'\|_{L^\infty}\|h_1-h_2\|_{L^\infty}\left(\frac{4}{c(h_0)}+\frac{1}{6c(h_0)}\right)\\
&< 2\cdot \frac{c(h_0)}{25}\cdot \frac{25}{6c(h_0)}=\frac{1}{3}.
\end{align*}
In other words, $\Gamma$ satisfies \eqref{inqGamma} and thus by Theorem \eqref{thm:FloerBounds} for every pair $(a,b)\in \R^2$ the corresponding space $\mathcal{M}^\Gamma(a,b)$ of Floer trajectories is bounded in $L^\infty$-norm. Therefore, we can apply Theorem \ref{thm:DefLRFH} and conclude that for every $h\in \mathcal{H}_{h_0}(K_{n,m})$ the Lagrangian Rabinowitz Floer homology is well-defined and isomorphic to $\LRFH_*(\A^{H_0-h_0}_{q_0,q_1})$.

Taking all the sets of the form $\mathcal{H}_{h_0}(K_{n,m})$
with $h_0\in \mathcal{H}$ and $n,m\in\mathbb{N}$ gives us an open cover of $\mathcal{H}$. Since $\mathcal{H}$ is path connected and the Lagrangian Rabinowitz Floer homology is constant on every open set $\mathcal{H}_{h_0}(K_{n,m})$ a basic topological argument gives us that for any pair $h_0,h_1\in \mathcal{H}$ the respective homologies $\LRFH_*(\A^{H_0-h_0}_{q_0,q_1})$ and $\LRFH_*(\A^{H_0-h_1}_{q_0,q_1})$ are isomorphic.
\hfill $\square$

\section{Positive Lagrangian Rabinowitz Floer homology}

The aim of this section is to prove Theorem \ref{thm:posLRFH}. We will start by showing that the positive Lagrangian Rabinowitz homology of the Copernican Hamiltonian is well-defined and invariant of perturbations. Further on we will prove that for fixed endpoints of the chords and high enough energy, the Copernican Hamiltonian has exactly one Reeb chord with positive action. Consequently, the corresponding positive Lagrangian Rabinowitz Floer homology has only one generator. Finally, we will show that the generator's Maslov index is in fact $0$.

\subsection{Definition and invariance of perturbations}
First, we will prove that if the perturbation is small enough then there exists a homotopy, such that any associated Floer trajectory starting at a critical point with a positive action has to end at a critical point with positive action.

\begin{lem}\label{lem:posGamma}
Let $H_0$ be the Hamiltonian defined in \eqref{DefH0} and let $\mathcal{H}$ be the set of perturbations defined in \eqref{DefHset}. Fix two points $q_0,q_1 \in \R^2$ and a perturbation\linebreak $h_0\in \mathcal{H}$. Denote $c(h_0):=\inf(h_0-dh_0(p\partial_p))$. Let $K\subseteq T^*\R^2$ be a compact subset, such that $\supp dh_0 \subseteq K$ and for which there exists $\mathcal{O}_{h_0}(K)$ an open neighbourhood of $0$ in $\{h\in C^\infty(T^*\R^2)\ |\ dh \in C_0^\infty(K)\}$, satisfying:
\begin{enumerate}
\item $\{h_0+h\ |\ h\in \mathcal{O}_{h_0}(K)\}\subseteq \mathcal{H}$;
\item The critical set of $ \A^{H_0-h_0}$ is continuously compact in $(K, \mathcal{O}_{h_0}(K))$;
\item $\inf\left\lbrace\begin{array}{c|c}
\inf((h_0+h)-d(h_0+h)(p\partial_p)) & h\in\mathcal{O}_{h_0}(K)\end{array} \right\rbrace >\frac{1}{2}c(h_0)$;
\item $\delta(h_0):=\inf\left\lbrace\begin{array}{c|c}
\A^{H_0-h_0-h}_{q_0,q_1}(x) & x\in \Crit^+ \A^{H_0-h_0-h}_{q_0,q_1}, \quad h\in\mathcal{O}_{h_0}(K)\end{array} \right\rbrace >0$.
\end{enumerate}
 
Fix $h \in \mathcal{O}_{h_0}(K)$ and denote $h_1:=h_0+h$. Let $\mathcal{V}\subseteq T^*\R^2$ be an open, but precompact subset, such that $K\subseteq \mathcal{V}$.
 Let $\Gamma:=\{(h_s, J_s)\}_{s\in \R}$ be a smooth homotopy of Hamiltonians $h_s\in h_0+\mathcal{O}_{h_0}(K)$ and $2$-parameter families of almost complex structures $J_s \in C^\infty([0,1]\times \R,\mathcal{J}(\mathcal{V},\mathcal{J}))$ constant in $s$ outside $[0,1]$ satisfying \eqref{AC} and such that
\begin{equation}\label{inq2Gamma}
\|\partial_{s}h_{s}\|_{L^{\infty}}\leq \frac{1}{3}\min\left\lbrace\left(\frac{4}{c(h_0)}+\|J\|_{L^{\infty}}^2\right)^{-1},\ \frac{c(h_0)}{2}\left(1+\frac{\sqrt{c(h_0)}}{2\delta(h_0)}\right)^{-1}\right\rbrace.
\end{equation}
Then for every $x\in \Crit^+\A^{H_0-h_0}_{q_0,q_1}$ and every $y\in \Crit\A^{H_0-h_1}_{q_0,q_1}$, such that $\F_\Gamma(x,y)\neq \emptyset$ we have $\A^{H-h_1}_{q_0,q_1}(y)>0$.

\end{lem}

\begin{proof}
The following proof is an adjustment of the proof of \cite[Cor. 3.8]{CieliebakFrauenfelder2009} to our setting and it
is similar to the proof of Lemma \ref{lem:Novikov} where we have proven the Novikov finiteness condition.

Fix $x\in \Crit^+\A^{H_0-h_0}_{q_0,q_1}$ and $y\in \Crit\A^{H_0-h_1}_{q_0,q_1}$ and abbreviate:
$$
a:= \A^{H_0-h_0}_{q_0,q_1}(x) \qquad\textrm{and}\qquad \A^{H_0-h_1}_{q_0,q_1}(y).
$$
By assumption $a\geq \delta(h_0)$. Let $u=(v,\eta)\in \F_\Gamma(x,y)$. Since $h_0,h_1\in \mathcal{H}_c$ and $\Gamma$ satisfies \eqref{inqGamma}, thus our setting satisfies the assumptions of Lemma \ref{lem:ActBound}. Suppose $|b|\leq a$. Then $b-a\leq 0$ and \eqref{eqEta} gives
\begin{align}
\|\eta\|_{L^\infty} &\leq \frac{3}{2} \left(\frac{2}{c(h_0)}\max\{|a|,|b|\}+\frac{1}{\sqrt{c(h_0)}}+\|J\|_{L^{\infty}}^2 (b-a)\right)\nonumber\\
&\leq\frac{3}{c(h_0)}\left(a+\frac{\sqrt{c(h_0)}}{2}\right)\leq \frac{3a}{c(h_0)}\left(1+\frac{\sqrt{c(h_0)}}{2\delta(h_0)}\right).\label{etaBDdelta}
\end{align}
On the other hand, by equation \eqref{inqEnergy} we have
$$
\|\nabla^{J_{s}} \mathcal{A}^{H_0-h_{s}}(u)\|_{L^{2}(\mathbb{R}\times [0,1])}^{2} \leq \|J\|_{L^{\infty}} (b-a + \|\eta\|_{L^{\infty}}\|\partial_{s}h_{s}\|_{L^{\infty}}).
$$
Combined with \eqref{inq2Gamma} and \eqref{etaBDdelta}, we obtain the following estimate:
$$
b \geq a-  \|\eta\|_{L^{\infty}}\|\partial_{s}h_{s}\|_{L^{\infty}} \geq a \left( 1-\frac{3}{c(h_0)}\left(1+\frac{\sqrt{c(h_0)}}{2\delta(h_0)}\right)\|\partial_{s}h_{s}\|_{L^{\infty}}\right)\geq \frac{1}{2}a.
$$
In particular, $\A^{H_0-h_1}_{q_0,q_1}(y)=b\geq \frac{1}{2}\delta(h_0)$. By assumption $h_1 \in \mathcal{O}_{h_0}(K)$, hence by Lemma \ref{lem:posAction} we infer that $\A^{H_0-h_1}_{q_0,q_1}(y)\geq\delta(h_0)$. This proves the claim under the assumption $|b|\leq a$.

Now suppose $b< -a\leq -\delta(h_0)$. Then $b-a\leq 0$ and by \eqref{eqEta} we have
\begin{align*}
\|\eta\|_{L^\infty} &\leq \frac{3}{2} \left( \frac{2}{c(h_0)}\max\{|a|,|b|\}+\frac{1}{\sqrt{c(h_0)}}+\|J\|_{L^{\infty}}^2 (b-a)\right)\nonumber\\
&\leq \frac{3}{c(h_0)}\left(\frac{\sqrt{c(h_0)}}{2}-b\right)<\frac{-3b}{c(h_0)}\left(1+\frac{\sqrt{c(h_0)}}{2\delta(h_0)}\right)
\end{align*}
Combining it with \eqref{inqEnergy} and \eqref{inq2Gamma}, we obtain the following inequality:
$$
a \leq b + \|\eta\|_{L^{\infty}}\|\partial_{s}h_{s}\|_{L^{\infty}} \leq b\left(1-\frac{3b}{c(h_0)}\left(1+\frac{\sqrt{c(h_0)}}{2\delta(h_0)}\right)\|\partial_{s}h_{s}\|_{L^{\infty}}\right)\leq \frac{1}{2}b<-\frac{1}{2}a,
$$
which contradicts the assumption $a\geq \delta(h_0)>0$. That excludes the case $b<-a$ and proves the lemma.
\end{proof}

We are now ready to prove the first part of Theorem \ref{thm:posLRFH}, i.e.
to prove that the positive Lagrangian Rabinowitz homology of the Copernican Hamiltonian is well-defined and invariant under compact perturbations.
 We will do that by showing that all conditions of Corollary \ref{cor:posLRFH} are satisfied in our setting.

\begin{prop}\label{prop:posLRFH}
Let $H_0$ be the Copernican Hamiltonian defined in \eqref{DefH0} and let $\mathcal{H}$ be the corresponding set of compact perturbations as in \eqref{DefHset}. Then for all $q_0,q_1\in \R^2$ and all $h\in \mathcal{H}$ the positive Lagrangian Rabinowitz Floer homology $\LRFH_*^+(\A^{H_0-h}_{q_0,q_1})$ is well defined and isomorphic to\linebreak $\LRFH_*^+(\A^{H_0-c}_{q_0,q_1})$ for any $c>0$.
\end{prop}

\begin{proof}
Fix two points $q_0,q_1 \in \R^2$ 
and let $\{K_{n,m}\}_{n,m\in\mathbb{N}}$ be the sequence\linebreak of compact sets
as in Lemma \ref{lem:H0Chord} satisfying \eqref{Knm}.
Fix $h_0\in \mathcal{H}^{\reg}$ and denote $c(h_0):=\inf(h_0-dh_0(p\partial_p))>0$.
Let $K_{n,m}\subseteq T^*\R^2$ be any compact set from the sequence, such that $\supp dh_0 \subseteq K_{n,m}$ and $\|h_0\|_{L^\infty}<m$. By Lemma \ref{lem:posAction} and \ref{lem:H0Chord} there exists an open neighbourhood $\mathcal{O}_{h_0}(K_{n,m})$ of $0$ in $\{h\in C^\infty(T^*\R^2)\ |\ dh\in C_0^\infty(K_{n,m})\}$, such that
\begin{enumerate}
\item $\{h_0+h\ |\ h\in \mathcal{O}_{h_0}(K_{n,m})\}\subseteq \mathcal{H}$;
\item The critical set of $ \A^{H_0-h_0}$ is continuously compact in $(K_{n,m}, \mathcal{O}_{h_0}(K_{n,m}))$;
\item For all $h\in\mathcal{O}_{h_0}(K_{n,m})$ we have
$\inf((h_0+h)-d(h_0+h)(p\partial_p)) >\frac{1}{2}c(h_0)$;
\item There exists $\delta(h_0)>0$, such that for all $h\in\mathcal{O}_{h_0}(K_{n,m})$ and all\linebreak $x\in \Crit^+ \A^{H_0-h_0-h}_{q_0,q_1}$ we have
$\A^{H_0-h_0-h}_{q_0,q_1}(x)\geq \delta(h_0)$;
\item For all $h\in \mathcal{O}_{h_0}(K_{n,m})$ we have
$$
\|h_0-h\|_{L^\infty}<\frac{1}{2}c(h_0)\inf \left\lbrace\frac{1}{13}, \frac{1}{6}\left(1+\frac{\sqrt{c(h_0)}}{2\delta(h_0)}\right)^{-1}\right\rbrace.
$$
\end{enumerate}

We will show that for every $h\in \mathcal{O}_{h_0}(K_{n,m})$ the positive Lagrangian Rabinowitz Floer homology $\LRFH_*^+(\A^{H_0-h_0-h}_{q_0,q_1})$ is well defined and isomorphic to $\LRFH_*^+(\A^{H_0-h_0}_{q_0,q_1})$.

Fix $h_1 \in\{h_0+h\ |\ h\in\mathcal{O}_{h_0}(K_{n,m})\}\cap \mathcal{H}^{\reg}$. Let $J_0\in \mathcal{J}^{\reg}_{h_0}$ and $J_1\in \mathcal{J}^{\reg}_{h_1}$ will be two almost complex structures close enough, so that there exists a smooth homotopy $\{J_s\}_{s\in\R}$ with $J_s\in C^\infty([0,1]\times\R,\mathcal{J,\mathbb{J}}))$ satisfying \eqref{AC} and $\sup_{s\in\R}\|J_s\|_{L^\infty}<\frac{1}{\sqrt{3c(h_0)}}$. Moreover, let $\chi\in C^\infty(\R)$ be a smooth function satisfying $\|\chi'\|_{L^\infty} \leq 2$ and \eqref{chi}. Define $h_s:=h_0(1-\chi(s))+\chi(s)h_1$. Then the smooth homotopy $\Gamma=\{(h_s,J_s)\}_{s\in\R}$ satisfies \eqref{inq2Gamma}.

In particular, the homotopy $\Gamma$ satisfies \eqref{inqGamma}, thus by Lemma \ref{lem:Novikov} it satisfies the Novikov finiteness condition. Moreover, by Theorem \ref{thm:FloerBounds} for any pair\linebreak $a,b\in \R$ the space of Floer trajectories $\mathcal{M}^\Gamma(a,b)$ is bounded in $L^\infty$-norm. Finally, by Lemma \ref{lem:posGamma} we know that for every $x\in \Crit^+\A^{H_0-h_0}_{q_0,q_1}$ and every $y\in \Crit\A^{H_0-h_1}_{q_0,q_1}$, such that $\F_\Gamma(x,y)\neq \emptyset$ we have $\A^{H-h_1}_{q_0,q_1}(y)>0$. This means that for every $h_1 \in\{h_0+h\ |\ h\in\mathcal{O}_{h_0}(K_{n,m})\}\cap \mathcal{H}^{\reg}$ there exists a homotopy $\Gamma$ satisfying all three conditions of Corollary \ref{cor:posLRFH}. Hence for every $h \in\mathcal{O}_{h_0}(K_{n,m})$ the positive Lagrangian Rabinowitz Floer homology $\LRFH_*^+(\A^{H_0-h_0-h}_{q_0,q_1})$ is well defined and isomorphic to $\LRFH_*^+(\A^{H_0-h_0}_{q_0,q_1})$.

Taking all the sets of the form $\mathcal{O}_{h_r}(K_{n,m})$
with $h_r\in \mathcal{H}^{\reg}$ and $n,m\in\mathbb{N}$ gives us an open cover of $\mathcal{H}$.
Since $\mathcal{H}$ is path connected and the positive Lagrangian Rabinowitz Floer homology is constant on every open set $\mathcal{O}_{h_r}(K_{n,m})$ a basic topological argument gives us that for any pair $h_0,h_1\in \mathcal{H}$ the respective homologies $\LRFH_*^+(\A^{H_0-h_0}_{q_0,q_1})$ and $\LRFH_*^+(\A^{H_0-h_1}_{q_0,q_1})$ are isomorphic.
\end{proof}

\subsection{The critical set of the Copernican Hamiltonian}
In this subsection we will analyse the Reeb chords corresponding to the Copernican Hamiltonian $H_0$ defined in \eqref{DefH0}.
More precisely, we will show that the positive Lagrangian Rabinowitz action functional $\A^{H_0-c}_{q_0,q_1}$
has an odd number of Reeb chords with positive $\eta>0$ under specific conditions relating the energy level set of the Hamiltonian $H_0^{-1}(c)>0$ and the endpoints of Reeb chords $q_0,q_1\in \R^2$, $q_0\neq q_1$.

\begin{prop}\label{prop:CritPM}
Let $H_0$ be the Hamiltonian defined in \eqref{DefH0} and let $\A^{H_0-c}_{q_0,q_1}:\mathscr{H}_{q_0,q_1}\times\R \to \R$ be the Lagranian Rabinowitz action functional corresponding to the energy $c>0$ and a pair $(q_0,q_1)\in \R^4, q_0\neq q_1$.
Then $\Crit\A^{H_0}_{q_0,q_1}$ has the following properties:
\begin{enumerate}
\item If $q_0\neq q_1$ and $|q_0||q_1|\leq 2c$ then $\#\Crit^+\A^{H_0-c}_{q_0,q_1}=\#\Crit^-\A^{H_0-c}_{q_0,q_1}=1$.
\item For a fixed $c>0$ there exists a residual set $\mathcal{Q}^c\subseteq \R^4\setminus \Delta$, where\linebreak $\Delta:=\left\lbrace (q,q)\ |\ q\in \R^2\right\rbrace$, such that for all pairs $(q_0,q_1)\in \mathcal{Q}^c$ both\linebreak $\#\Crit^+\A^{H_0-c}_{q_0,q_1}$ and $\#\Crit^-\A^{H_0-c}_{q_0,q_1}$ are odd numbers.
\item For a fixed pair $(q_0,q_1)\in \R^4,\ q_0\neq q_1$ there exists a residual set\linebreak $\mathcal{I}_{q_0,q_1}\subseteq \R_+$, such that for all $c\in \mathcal{I}_{q_0,q_1}$ both $\#\Crit^+\A^{H_0-c}_{q_0,q_1}$ and\linebreak $\#\Crit^-\A^{H_0-c}_{q_0,q_1}$ are odd numbers.
\end{enumerate}
\end{prop}
\begin{proof}
For clarity of the argument, we will divide the proof of this proposition in a sequence of lemmas. First recall that by equation~\eqref{Crit+} we have
$$
\Crit^\pm\A^{H}_{q_0,q_1}=\left\lbrace (v,\eta)\in \Crit\A^{H}_{q_0,q_1}\ |\ \pm \eta >0\right\rbrace.
$$
In Lemma \ref{lem:bijection} we will show that for each $c>0$ and each pair $(q_0,q_1)\in \R^4$ there exists a smooth function $f\in C^\infty(\R)$, such that $(v,\eta)\in \Crit\A^{H_0-c}_{q_0,q_1}$ if and only if $f(\eta)=0$. We also show that for $q_0\neq q_1$ the critical set of $\A^{H_0-c}_{q_0,q_1}$ is discrete. Denote by $Z(f)$ the set of roots of the function $f$, $Z(f):=\{\eta\in\R |\ f(\eta)=0\}$.
In Lemma \ref{lem:2c>q0q1} we will consider pairs $(q_0,q_1)\in \R^4$, such that $|q_0||q_1|\leq 2c$ and show that for $q_0\neq q_1$ the corresponding function $f$ satisfies $\#\{\eta>0\ |\ f(\eta)=0\}=\#\{\eta<0\ |\ f(\eta)=0\}=1$, whereas for $q_0=q_1$ we have $Z(f)=\{0\}$.

Denote $\Delta:=\{(q,q)\ |\ q\in \R^2\}\subseteq \R^4$. In Lemma \ref{lem:Res_q0q1} we prove that for a fixed $c>0$ there exists a residual set $\mathcal{Q}^c\subseteq \R^4\setminus \Delta$, such that for all pairs $(q_0,q_1)\in \mathcal{Q}^c$ the corresponding function $f$ satisfies:
 \begin{equation}\label{f'neq0}
 f'(\eta)\neq 0\qquad \forall\ \eta \in Z(f), \qquad \textrm{where}\quad Z(f):=\{\eta \in \R\ |\ f(\eta)=0\}. 
 \end{equation}
 
 In Lemma \ref{lem:Res_c} we show that for a fixed pair $(q_0,q_1)\in \R^4, q_0\neq q_1$ there exists a residual set $\mathcal{I}_{q_0,q_1}\subseteq \R_+$, such that for all $c\in \mathcal{I}_{q_0,q_1}$ the corresponding function $f$ satisfies property \eqref{f'neq0}.
 
 In Lemma \ref{lem:Odd} we prove that if $c>0$ and $(q_0, q_1)\in \R^4, q_0\neq q_1$ are such that the corresponding function $f$ satisfies \eqref{f'neq0}, then $f$ has an odd number of roots on each positive and negative half line. This together with the bijection established in Lemma \ref{lem:bijection} concludes the proof.
\end{proof}
Before we prove the first lemma we will start by recalling the properties of the group of matrices corresponding to rotations: if we define
\begin{equation}
R(t) :=\left(\begin{array}{c c}
\cos(t) & \sin (t) \\
-\sin(t)& \cos(t) \\
\end{array}\right).\label{DefR}
\end{equation}
then $\{R(t)\}_{t\in\R}$ form a group with the following properties:
\begin{align}
 R(0)=\Id\quad \textrm{and}\quad R\left(\frac{\pi}{2}\right) = & \left(\begin{array}{c c}
0 & 1 \\
-1 & 0 \\
\end{array}\right),\nonumber\\
R(t)R(\tau)=R(t+\tau)\qquad & \forall\ t,\tau\in\R,\label{Radd}\\
R^{-1}(t)=R^T(t)=R(-t)\qquad & \forall\ t\in\R.\label{Rinv}
\end{align}
Moreover, we have
\begin{align}
R'(t) &= \left(\begin{array}{c c}
-\sin(t) & \cos (t) \\
-\cos(t)& -\sin(t) \\
\end{array}\right)=\left(\begin{array}{c c}
0 & 1 \\
-1 & 0 \\
\end{array}\right)
\left(\begin{array}{c c}
\cos(t) & \sin (t) \\
-\sin(t)& \cos(t) \\
\end{array}\right)\nonumber\\ 
&= R\left(\frac{\pi}{2}\right)R(t)= R\left( t +\frac{\pi}{2}\right).\label{DerR}
\end{align}

In the following lemma we will show that the critical set of $\A^{H_0-c}_{q_0,q_1}$ consists of isolated points and, in case $q_0=q_1$ it is a circle. 

\begin{lem}\label{lem:bijection}
Let $H_0$ be the Hamiltonian defined in \eqref{DefH0}. For every $c>0$ and every pair $q_0,q_1\in \R^2$ we define
\begin{equation}\label{Deff}
f(\eta):=-c\eta^2+\eta q_1^TR\left(\eta+\frac{\pi}{2}\right)q_0+\frac{1}{2}\left(|q_0|^2+|q_1|^2\right)-q_1^TR(\eta)q_0.
\end{equation}
If we denote 
\begin{equation}\label{DefZ(f)}
Z(f):=\{\eta\in \R \ |\ f(\eta)=0\},
\end{equation}
then the map $Z(f)\setminus\{0\}\ni \eta \longmapsto ( v_\eta, \eta),$
$$
\textrm{with}\qquad v_\eta(t)  := \left(\begin{array}{c c}
(1-t) R(t\eta) & t R(\eta(t-1))\\
 -\frac{1}{\eta}R(t\eta) & \frac{1}{\eta} R(\eta(t-1))
\end{array}\right)
\left(\begin{array}{c}
q_0 \\ q_1
\end{array}\right).
$$
defines a bijection between the set $Z(f)\setminus\{0\}$ and $\Crit\A^{H_0-c}_{q_0,q_1}\setminus (\{0\}\times \mathscr{H}_{q_0,q_1})$.
\end{lem}
\begin{proof}
By definition a pair $(v,\eta)\in \Crit\A^{H_0-c}_{q_0,q_1}$ if and only if it satisfies the following three conditions:
\begin{enumerate}
\item The curve $v$ starts in the Lagrangian $T^*_{q_0}\R^2$ and ends in $T^*_{q_1}\R^2$;
\item The curve $v$ is a Reeb chord and $\eta$ is its period;
\item The curve $v$ lies on the level set $H_0^{-1}(c)$.
\end{enumerate}
Note that if $q_0\neq q_1$ we have $\Crit\A^{H_0-c}_{q_0,q_1}\cap\left(\mathscr{H}_{q_0,q_1}\times\{0\}\right)=\emptyset$. In case  $q_0= q_1$ we always have a submanifold of constant solutions 
$$
\Crit\A^{H_0-c}_{q_0,q_1}\cap\left(\mathscr{H}_{q_0,q_0}\times\{0\}\right)=\{q_0\}\times\{p_0\in T^*_{q_0}\R^2\ |\ H_0(q_0,p_0)=c\}.
$$
One can easily calculate that for a fixed $q_0\in \R^2$ the set $\{ p \in T^*\R^2\ |\  H_0(q_0,p)=c\}$ is a circle with origin at $-\mathbb{J}q_0$ and radius $\sqrt{|q_0|^2+2c}$:
$$
H_0(q,p)=c \quad \iff \quad |p+ \mathbb{J}q_0|^2=2c+|q_0|^2.
$$

From now on we will assume $\eta\neq 0$.

Using \eqref{DefH0} and the notation from \eqref{DefR} we can express $H_0$ and $X_{H_0}$ in the following way:
\begin{align}
H_0(q,p) & = \frac{1}{2}(q,p) A (q,p)^T, \nonumber\\
X_{H_0}(q,p) &= \mathbb{J}A (q,p)^T\nonumber\\
\textrm{where}\qquad A & =\left(\begin{array}{c c}
0 &  -R\left(\frac{\pi}{2}\right)\\
R\left(\frac{\pi}{2}\right) & \Id
\end{array}\right).\label{DefMatrixA}
\end{align}
In particular, the flow $\varphi^t$ of the Hamiltonian vector field $X_{H_0}$ can be easily calculated to be
\begin{equation}
\varphi^t(q,p)=\Exp(t\mathbb{J}A)\left(\begin{array}{c}
q \\ p
\end{array}\right) =
\left(\begin{array}{c c}
 R(t) & t R(t)\\
 0 & R(t)
\end{array}\right)
\left(\begin{array}{c}
q \\ p
\end{array}\right)\label{flow}
\end{equation}

By definition, a critical point $(v,\eta)\in \Crit\A^{H_0-c}_{q_0,q_1}$ satisfies $\varphi^\eta \circ v (0)=v(1)$. Moreover, we have $\pi \circ v(0)=q_0$ and $\pi\circ v(1)=q_1$. 
If we denote $v(0)=(q_0,p_0)$ and $v(1)=(q_1,p_1)$ and
combine it with the matrix formula for the Hamiltonian flow, we obtain
$$
\left(\begin{array}{c}
q_1 \\ p_1
\end{array}\right)=
\left(\begin{array}{c c}
R(\eta) & \eta R(\eta)\\
0 & R(\eta)
\end{array}\right)
\left(\begin{array}{c}
q_0 \\ p_0
\end{array}\right).
$$
Assuming, that $\eta\neq 0$ we can use \eqref{Rinv} to solve for $p_0$ and $p_1$:
$$
\left(\begin{array}{c}
p_0 \\ p_1
\end{array}\right)
 =\frac{1}{\eta}
\left(\begin{array}{c c}
- \Id & R(-\eta)\\
-R(\eta) & \Id
\end{array}\right)
\left(\begin{array}{c}
q_0 \\ q_1
\end{array}\right).
$$
Using the equation above we obtain the following formula for $v(0)$:
\begin{equation}\label{q0p0}
\left(\begin{array}{c}
q_0 \\ p_0
\end{array}\right) =
\left(\begin{array}{c c}
 \Id &0\\
-\frac{1}{\eta} & \frac{1}{\eta}R(-\eta)
\end{array}\right)
\left(\begin{array}{c}
q_0 \\ q_1
\end{array}\right)
\end{equation}
Now, by combining \eqref{flow} and \eqref{q0p0} we can express a Reeb chord $v$ in terms of $q_0$ and $q_1$ in the following way:
\begin{align*}
v_\eta(t) & := \left(\begin{array}{c c}
 R(t\eta) & t \eta R(t\eta)\\
 0 & R(t\eta)
\end{array}\right)
\left(\begin{array}{c c}
\Id & 0\\
-\frac{1}{\eta} & \frac{1}{\eta}R(-\eta)
\end{array}\right)
\left(\begin{array}{c}
q_0 \\ q_1
\end{array}\right)\\
& = \left(\begin{array}{c c}
(1-t) R(t\eta) & t R(\eta(t-1))\\
 -\frac{1}{\eta}R(t\eta) & \frac{1}{\eta} R(\eta(t-1))
\end{array}\right)
\left(\begin{array}{c}
q_0 \\ q_1
\end{array}\right).
\end{align*}
Such defined $v_\eta$ is a solution of the Hamiltonian flow and it belongs to $\mathscr{H}_{q_0,q_1}$.
To assure that $v_\eta(t)\in H_0^{-1}(c)$ for all $t\in [0,1]$ we use \eqref{Radd}, \eqref{Rinv}, \eqref{DefMatrixA} and \eqref{q0p0} to express the condition $(q_0,p_0)\in H_0^{-1}(c)$ in the following way:
\begin{align*}
c & =\frac{1}{2}
\left(q_0,\ p_0\right)
\left(\begin{array}{c c}
 \Id & -\frac{1}{\eta}\\
0 & \frac{1}{\eta}R(\eta)
\end{array}\right)
\left(\begin{array}{c c}
0 &  -R\left(\frac{\pi}{2}\right)\\
R\left(\frac{\pi}{2}\right) & \Id
\end{array}\right)
\left(\begin{array}{c c}
 \Id &0\\
-\frac{1}{\eta} & \frac{1}{\eta}R(-\eta)
\end{array}\right)
\left(\begin{array}{c}
q_0 \\ q_1
\end{array}\right)\\
& =\frac{1}{2} \left(q_0,\ p_0\right)
\left(
\frac{1}{\eta^2}
\left(\begin{array}{c c}
 \Id &-R(-\eta)\\
-R(\eta) & \Id
\end{array}\right)
+\frac{1}{\eta}
\left(\begin{array}{c c}
 0 & -R\left(\frac{\pi}{2}-\eta\right)\\
R\left(\frac{\pi}{2}+\eta\right) & 0
\end{array}\right)
\right)
\left(\begin{array}{c}
q_0 \\ q_1
\end{array}\right)\\
& = \frac{1}{\eta^2}\left(\frac{1}{2}\left( |q_0|^2+|q_1|^2\right)-q_1R(\eta)q_0^T\right)
+\frac{1}{\eta}q_1R\left(\eta+\frac{\pi}{2}\right)q_0^T.
\end{align*}
Now if we define $f\in C^\infty(R)$ as in \eqref{Deff}
then $(v_\eta,\eta)\in\Crit\A^{H_0-c}_{q_0,q_1}$ if and only if $f(\eta)=0$. 

Observe that
$f(0)=\frac{1}{2}|q_1-q_0|^2$, thus $0\in Z(f)$ if and only if $q_0=q_1$. In that case, $\left(\A^{H_0-c}_{q_0,q_1}\right)^{-1}(0)$ is a circle.
\end{proof}
\begin{lem}\label{lem:2c>q0q1}
In case $|q_0||q_1|\leq 2c$ the critical set of $\A^{H_0-c}_{q_0,q_1}$ satisfies one of the following properties:
\begin{itemize}[leftmargin=2cm]
\item[either] $q_0 \neq q_1$ and then $\#\Crit^+\A^{H_0-c}_{q_0,q_1}=\#\Crit^-\A^{H_0-c}_{q_0,q_1}=1$,
\item[or] $q_0=q_1$ and then $\#\Crit^+\A^{H_0-c}_{q_0,q_1}=\#\Crit^-\A^{H_0-c}_{q_0,q_1}=0$.
\end{itemize}
\end{lem}
\begin{proof}
By Lemma \ref{lem:bijection} we know that we can restrict ourselves to the analysis of function $f$ defined in \eqref{Deff}.
Note that $f(0)=\frac{1}{2}|q_1-q_0|^2$, thus 
\begin{align*}
\textrm{for}\quad q_0 & \neq q_1\quad \textrm{we have}\quad f(0)>0,\\
\textrm{and for} \quad q_0 & =q_1 \quad \textrm{we have}\quad f(0)=0.
\end{align*}
In both cases, we know that $\lim_{\eta \to \pm \infty}f(\eta)=-\infty$ due to the dominance of the summand $-c\eta^2$ for $|\eta|$ large enough. In case $q_0\neq q_1$ we can use the intermediate value theorem to prove that $f$ has a root both on the positive and the negative half line. What is left to show is that $f$ has only one root on each half-line.

Using \eqref{DerR} we can now calculate the derivative of $f$:
\begin{align*}
f'(\eta) & = -2c\eta +q_1^T\left(R\left(\eta+\frac{\pi}{2}\right)-R'\left(\eta\right)\right)q_0+\eta q_1^TR'\left(\eta+\frac{\pi}{2}\right)q_0\\
& = -\eta (2 c  +q_1^TR(\eta)q_0).
\end{align*}
Thus in case $|q_0||q_1|\leq 2c$ we have
$$
2 c  +q_1^TR(\eta)q_0) \geq 2c -|q_0||q_1|\geq 0,
$$
and consequently $f'(\eta)\leq 0$ for $\eta>0$ and $f'(\eta)\geq 0$ for $\eta<0$. Moreover, in case $|q_0||q_1|=2c$ we have $f'(\eta)=0$ if and only if $\eta=0$ or $q_1=-\frac{2c}{|q_1|^2}R(\eta)q_0$. In particular, in case $|q_0||q_1|=2c$ the critical points of $f$ are isolated. Consequently, $f$ obtains its maximum at $0$ and is non-decreasing on the negative half line and non-increasing on the positive half line. This proves that in case $q_0\neq q_1$ there exists only one root on each positive and negative half line. On the other hand, since $f$ obtains its maximum at $0$, thus in case $q_0=q_1$ we know that $f(\eta)<0$ for all $\eta\neq 0$, which concludes the proof.
\end{proof}

\begin{lem}\label{lem:Res_q0q1}
 For a fixed $c>0$ there exists a residual set $\mathcal{Q}^c\subseteq \R^4\setminus \Delta$, such that for all pairs $(q_0,q_1)\in \mathcal{Q}^c$ the corresponding function $f$ as defined in \eqref{Deff} satisfies property \eqref{f'neq0}.
 \end{lem}

\begin{proof}
To prove the statement of the lemma, we will first extend the function $f\in C^\infty(\R)$ to $\bar{f}\in C^\infty(\R^5)$ by taking $\bar{f}(q_0,q_1,\eta):=f(\eta)$ where $f$ depends on $q_0$ and $q_1$ in the way described in \eqref{Deff}. 

Recall, that $f(0)=\frac{1}{2}|q_1-q_0|^2$, thus 
$$
\bar{f}^{-1}(0)\cap \left(\R^4\times\{0\}\right)=\Delta\times\{0\}.
$$
We will show now that $\bar{f}^{-1}(0)\setminus (\Delta\times\{0\})$ is a smooth manifold. Using \eqref{Deff} we can calculate $D\bar{f}$:
\begin{align*}
\partial_\eta \bar{f} & = -\eta \left(2 c  +q_1^TR(\eta)q_0\right),\\
\partial_{q_0}\bar{f} & = \eta q_1^T R\left(\eta+\frac{\pi}{2}\right)- q_1^T R(\eta)+q_0,\\
\partial_{q_1}\bar{f} & = \eta q_0^T R\left(-\eta-\frac{\pi}{2}\right)- q_0^T R(-\eta)+q_1.
\end{align*}
A straightforward computation shows that $\bar{f}$ satisfies the following relation:
$$
\bar{f}(q_0,q_1,\eta)=-c\eta^2+\frac{1}{2}\left( q_0\partial_{q_0}\bar{f}+ q_1\partial_{q_1}\bar{f}\right).
$$
Consequently, for all $(q_0,q_1,\eta)\in \bar{f}^{-1}(0)$ we have $|\eta||\partial_{q_0}\bar{f}||\partial_{q_1}\bar{f}|\neq 0$. In particular, for all $(q_0,q_1,\eta)\in \bar{f}^{-1}(0)\setminus (\Delta\times\{0\})$ the derivative $D\bar{f}\neq 0$ and thus, by the inverse function theorem, $\bar{f}^{-1}(0)\setminus (\Delta\times\{0\})$ is a smooth manifold.

Note, that if we consider the function $f\in C^\infty(\R)$ as in \eqref{Deff} corresponding to a fixed pair $(q_0,q_1)\in \R^4$, then
$$
\left( \bar{f}^{-1}(0)\setminus \Delta\right) \cap \left( \{(q_0,q_1)\}\times\R\right)=Z(f)\setminus \{0\}.
$$
Additionally, $f'(\eta)=\partial_\eta \bar{f}(q_0,q_1,\eta)$ for all $\eta\in \R$. Therefore, if we denote
\begin{align*}
\mathcal{Q}^c := & \left\lbrace  (q_0,q_1)\in \R^4\ |\ \forall\ \eta\in Z(f)\setminus\{0\}\quad f'(\eta)\neq 0\right\rbrace,\\
\textrm{then}\qquad \mathcal{Q}^c =& \left\lbrace (q_0,q_1)\in \R^4\ |\ \forall\ (q_0,q_1,\eta)\in \bar{f}^{-1}(0)\setminus \Delta\quad \partial_\eta \bar{f}\neq 0 \right\rbrace.
\end{align*}
Consequently, the set $\mathcal{Q}^c$ satisfies the assumptions of the theorem. 
On the other hand, since $\bar{f}^{-1}(0)\setminus \Delta$ is a smooth manifold, the set of regular values of the projection $P: \bar{f}^{-1}(0)\setminus \Delta \to \R^4$ is in fact equal to $\mathcal{Q}^c$. By the Morse-Sard Theorem $\mathcal{Q}^c$ is residual in $\R^4$.
\end{proof}

\begin{lem}\label{lem:Res_c}
For a fixed pair $(q_0,q_1)\in \R^4, q_0\neq q_1$ there exists a residual set $\mathcal{I}_{q_0,q_1}\subseteq \R_+$, such that for all $c\in \mathcal{I}_{q_0,q_1}$ the corresponding function $f$ as defined in \eqref{Deff} satisfies property \eqref{f'neq0}.
\end{lem}
\begin{proof}
In this proof we will follow the same arguments as in the proof of Lemma \ref{lem:Res_q0q1}. First, we extend the function $f\in C^\infty(\R)$ as in \eqref{Deff} to a function $\tilde{f}:\R_+\times\R \to \R$ by setting $\tilde{f}(c,\eta):=f(\eta)$. Second, we observe that $\partial_c \tilde{f}(c,\eta)= -\eta^2 < 0$ for all $\eta\neq 0$. By assumption we take $q_0\neq q_1$, hence $0\notin\tilde{f}^{-1}(0)$ and thus $\partial_c \tilde{f}(c, \eta)<0$ for all $(c,\eta) \in \tilde{f}^{-1}(0)$. This allows us to use the inverse function theorem to conclude that $\tilde{f}^{-1}(0)$ is a smooth manifold. Furthermore, we observe that the set of regular values of the projection $P:\tilde{f}^{-1}(0) \to \R^+$ is equal to
\begin{align*}
\mathcal{I}_{q_0,q_1} := & \left\lbrace c\in \R_+\ |\ \forall\ (c,\eta)\in \tilde{f}^{-1}(0)\quad \partial_\eta \tilde{f}\neq 0 \right\rbrace\\
 = &\left\lbrace c\in \R_+\ |\ \forall\ \eta\in Z(f)\quad f'\neq 0 \right\rbrace.
\end{align*}
Finally, by the Morse-Sard Theorem, we conclude that $\mathcal{I}_{q_0,q_1}$ is residual in $\R_+$.
\end{proof}

In this final lemma, we show how the condition \eqref{f'neq0} implies that the function $f$ has an odd number of zeros on each the positive and the negative half line.
\begin{lem}\label{lem:Odd}
If $c>0$ and $(q_0, q_1)\in \R^4, q_0\neq q_1$ are such that the corresponding function $f$ satisfies \eqref{f'neq0}, then $f$ has an odd number of roots on each positive and negative half line.
\end{lem}
\begin{proof}
First we will show that if $t_0,t_1\in \R, t_0<t_1$ are two roots of a smooth function $f\in C^\infty(\R)$, such that $f'(t_0), f'(t_1)\neq 0$ and $f(t)\neq 0$ for all\linebreak $t\in (t_0,t_1)$, then $f'(t_0)f'(t_1)<0$. Suppose the opposite is true and $f'(t_0)f'(t_1)>0$. Without loss of generality, we can assume that $f'(t_0),f'(t_1)>0$. Then there would exist $\delta>0$, such that for all $t\in (t_0, t_0+\delta),\ f(t)>0$ and for all $t\in (t_1-\delta,t_1),\ f(t)<0$. Consequently, by the intermediate value theorem there would have to exist $t\in (t_0+\delta,t_1-\delta)$, such that $f(t)=0$. But that brings us a contradiction.

Let now $f$ be the function defined in \eqref{Deff}. We will show that the set of roots of $f$ is bounded. Observe that $f$ can be bounded from above by
$$
f(\eta) \leq -c \eta^2 + |\eta||q_0||q_1|+\frac{1}{2}(|q_0|+|q_1|)^2.
$$
Consequently, we have
$f(\eta) < 0$ for $\eta \in (-\infty, -\delta_0)\cup (\delta_0, +\infty)$, where
\begin{equation}\label{delta}
\delta_0:= \frac{1}{2c}\left(\sqrt{|q_0|^2|q_1|^2+2c(|q_0|+|q_1|)^2}+|q_0||q_1|\right).
\end{equation}
 In particular, the set of roots of $f$ denoted by $Z(f)$ is a subset of $[-\delta_0,\delta_0]$. Since $Z(f)$ is bounded and discrete the number $k:=\#\{ \eta \in Z(f)\ |\ \eta>0\}\in \mathbb{N}$ is well defined. Moreover, 
we can enumerate the elements of the set $\{ \eta \in Z(f)\ |\ \eta>0\}=\{\eta_i\}_{i=1}^{k}$, such that $\eta_i< \eta_{i+1}$ for $i=1,\dots k-1$ and
\begin{equation}\label{AssEta}
\begin{aligned}
(\eta_i, \eta_{i+1})\cap Z(f) & = \emptyset && \textrm{for} && i=1, \dots k-1,\\
(0, \eta_1)\cap Z(f) & =\emptyset&& \textrm{and} && (\eta_{k}, +\infty)\cap Z(f)=\emptyset.
\end{aligned}
\end{equation}
Our aim is to prove that $k$ is an odd number.

We will show now that if $f$ satisfies \eqref{f'neq0} then $f'(\eta_1)<0$. Recall that $f(0)=\frac{1}{2}|q_1-q_0|^2$, which in the case $q_0\neq q_1$ means $f(0)>0$. Suppose now that $f'(\eta_1)>0$. That would imply that there exists $\delta>0$, such that\linebreak $f(\eta)<f(\eta_1)=0$ for $\eta \in (\eta_1-\delta, \eta_1)$. Consequently, by the intermediate value theorem, there would have to exist $\eta \in (0, \eta_1-\delta)$, such that $f(\eta)=0$. But that would contradict our assumption \eqref{AssEta}. Thus $f'(\eta_1)<0$.

The last step would be to show that if $f$ satisfies \eqref{f'neq0} then $f'(\eta_{k})<0$.  Recall that $\lim_{\eta \to +\infty}f(\eta)=-\infty$.
Suppose now that $f'(\eta_{k})>0$. That would imply that there exists $\delta>0$, such that for all $\eta \in (\eta_{k}, \eta_{k}+\delta)$, $f(\eta)>0$. Consequently, by the intermediate value theorem, there would have to exist $\eta \in (\eta_{k}+\delta,+\infty)$, such that $f(\eta)=0$. But that contradicts assumption \eqref{AssEta}.

Let us recall what we have learned: if $f$ satisfies \eqref{f'neq0} then $f'(\eta_1)<0$, $f'(\eta_i)$ has an opposite sign than $f'(\eta_{i+1})$ for all $i=1, \dots k-1$ and finally $f'(\eta_{k})<0$. This implies that $k$ is an odd number.

Using the same arguments we prove that $\#\{ \eta \in Z(f)\ |\ \eta<0\}$ is an odd number as well.
\end{proof}

Despite the fact that the positive Lagrangian Rabinowitz Floer homology has only one generator, the number of positive critical points of the Rabinowitz action functional corresponding to a fixed pair of endpoints $q_0,q_1 \in \R^2$ does not necessarily have to be $1$. In fact, its cardinality depends on the value of the energy $c>0$.
As shown in Lemma 4.4, the positive critical set of the Rabinowitz action functional is in bijection with the set of positive zeroes of the corresponding function $f$ defined in (4.8). In Figure \ref{fig:fFunction} we depict the functions $f$ corresponding to the chord endpoints $q_0=(1,0)$, $q_1=(0,1)$ and three different energies $c= \frac{1}{5}$, $c=\frac{1}{10}$, and $c=\frac{1}{20}$, respectively. We see that the number of positive zeroes of the functions increases as the energy $c$ decreases.

\begin{figure}
\centering
\includegraphics[width=.75\linewidth]{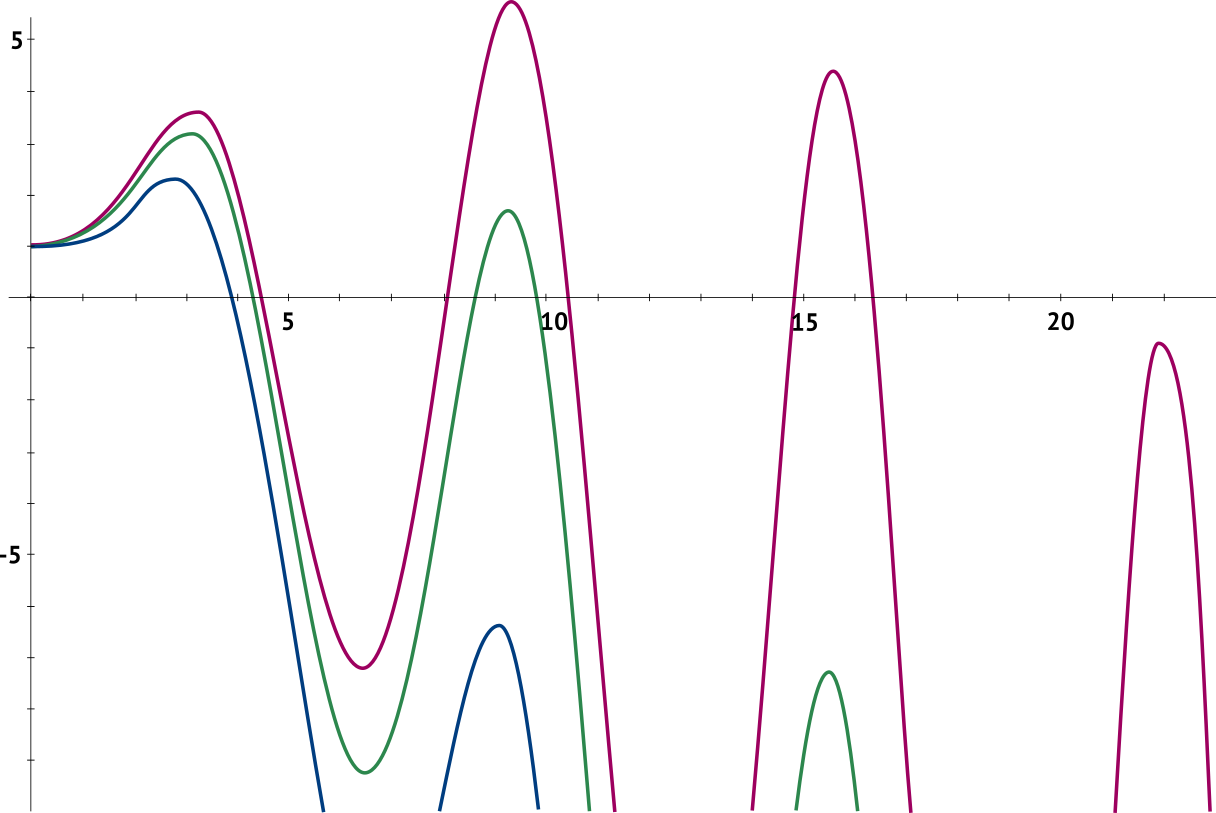}
\caption{The three functions $f$ corresponding to energy $c=\frac{1}{5}$ (blue), $c=\frac{1}{10}$ (green), and $c=\frac{1}{20}$ (magenta) crossing zero in exactly $1$, $3$, and $5$ points, respectively.}
\label{fig:fFunction}
\end{figure}

The following proposition shows that the number of positive critical points of the Rabinowitz action functional tends to $\infty$ as the energy $c$ approaches $0$.

\begin{prop}
For fixed $q_0,q_1 \in \R^2$ we have
$$
\lim_{c\searrow 0}\# \left(\Crit^+ \A_{q_0,q_1}^{H_0-c}\right)=+\infty.
$$
\end{prop}

\begin{proof}
By Lemma 4.4 we know that for any fixed $q_0,q_1 \in \R^2$ and $c>0$ we have $ \#\left(\Crit^+ \A_{q_0,q_1}^{H_0-c}\right)= \# \{ \eta>0\ |\ f(\eta)=0\}$, where $f$ is the function defined in (4.8) which depends on $q_0$, $q_1$, and $c$.
We express $q_0$ and $q_1$ in polar coordinates as
\begin{align*}
q_0 & = r_0  \left(\begin{array}{c}\cos (\alpha) \\ \sin (\alpha)\end{array}\right) = r_0 R^T(\alpha)\left(\begin{array}{c} 1 \\ 0\end{array}\right),\\
q_1 & = r_1 \left(\begin{array}{c}\cos (\alpha+\theta) \\ \sin (\alpha+\theta)\end{array}\right) = r_1 R^T(\alpha+\theta)\left(\begin{array}{c} 1 \\ 0\end{array}\right).
\end{align*}
Plugging this into the function $f$ from (4.8), we obtain
\begin{align*}
f(\eta) 
& = -c \eta^2+\frac{1}{2}\left(r_0^2+r_1^2\right)+ r_0 r_1 (1,0)\left(\eta R\left(\eta+\theta+\frac{\pi}{2}\right)-R(\eta+\theta)\right)\left(\begin{array}{c} 1 \\ 0\end{array}\right)\\
  & = -c \eta^2-\eta\ r_0 r_1 \sin\left(\eta +\theta
  \right) +\frac{1}{2}\left(r_0^2+r_1^2\right)-r_0 r_1 \cos (\eta +\theta).
\end{align*}
Note that the function depends only on the relative angle $\theta$ between the two endpoints. This is not very surprising as the Hamiltonian $H_0$ is invariant under rotations around the origin.
Observe that
\begin{align*}
f(\eta) & = - c\eta^2 +\frac{1}{2}(r_0-r_1)^2 && \textrm{for}\qquad \eta = 2\pi n - \theta, && n \in \mathbb{N},\\
f(\eta) & = - c\eta^2 +\frac{1}{2}(r_0+r_1)^2 && \textrm{for}\qquad \eta = \pi (2n+1) - \theta, && n \in \mathbb{N}.
\end{align*}
Consequently, we have
\begin{align*}
f(\eta) & <0 && \textrm{for}\qquad \eta = 2\pi n - \theta > \frac{1}{\sqrt{2c}}|r_0-r_1|, && n\in\mathbb{N},\\
f(\eta) & >0 && \textrm{for}\qquad \eta = \pi (2n+1) - \theta < \frac{1}{\sqrt{2c}}(r_0+r_1), && n\in\mathbb{N}.
\end{align*}
Note that if we assume
$$
\frac{\theta}{\pi} +\frac{1}{\pi\sqrt{2c}}|r_0-r_1| < n < \frac{\theta}{\pi} +\frac{1}{\pi\sqrt{2c}}(r_0+r_1)-1,$$
then
$$
\frac{\theta}{\pi} +\frac{1}{\pi\sqrt{2c}}|r_0-r_1| < n, n+1 <  \frac{\theta}{\pi} +\frac{1}{\pi\sqrt{2c}}(r_0+r_1).
$$
Now exactly one of $n,n+1$ is even and the other is odd, so
$$
f(\pi n - \theta)f(\pi (n+1) - \theta)<0,
$$
and thus by the intermediate value theorem there exists $\eta\in (\pi n-\theta, \pi(n+1)-\theta)$ such that $f(\eta)=0$.

Therefore, we can estimate
\begin{align*}
\#\left(\Crit^+ \A_{q_0,q_1}^{H_0-c}\right) & =\# \left\lbrace \eta>0\ |\ f(\eta)=0\right\rbrace \\
&\geq \# \left( \frac{\theta}{\pi} + \frac{|r_0-r_1|}{\pi\sqrt{2c}}, \frac{\theta}{\pi} +\frac{(r_0+r_1)}{\pi\sqrt{2c}}-1\right)\cap \mathbb{N}.
\end{align*}
The length of the interval on the right hand side of the inequality is equal to $\frac{\sqrt{2}}{\pi\sqrt{c}}\min\{r_0,r_1\}-1$. Since $\lim_{c\to 0}\left(\frac{\sqrt{2}}{\pi\sqrt{c}}\min\{r_0,r_1\}-1\right)=+\infty$, we conclude
$$
 \lim_{c\searrow 0} \#\left(\Crit^+ \A_{q_0,q_1}^{H_0-c}\right) \geq \lim_{c\searrow 0}\# 
\left( \frac{\theta}{\pi} + \frac{|r_0-r_1|}{\pi\sqrt{2c}}, \frac{\theta}{\pi} +\frac{(r_0+r_1)}{\pi\sqrt{2c}}-1\right)\cap\mathbb{N} = +\infty.
$$
\end{proof}

In Figures \ref{fig:1Orbit3Energies}, \ref{fig:3Orbits} and \ref{fig:5Orbits} we present the plots of the Reeb chords from $q_0=(0,1)$ to $q_1=(1,0)$ for various energies. More precisely, the graphs depict the projections of the Reeb chords onto the plane of positions. The plots have been obtained using the formula from Lemma \ref{lem:bijection}.

Let us analyse how the number of Reeb chords depends on the energy $c\in \left\lbrace 1, \frac{1}{2}, \frac{1}{5}, \frac{1}{10}, \frac{1}{20}\right\rbrace$.
By Lemma 4.5 we have $\Crit^+\A_{q_0,q_1}^{H_0-c}=1$ for $c \geq \frac{1}{2}$.
On the other hand, from Lemma 4.4 we know that $\Crit^+\A_{q_0,q_1}^{H_0-c}$ is in bijection with the zeroes of the corresponding function $f$ in (4.8).
By \eqref{delta} we know that
$$
\left\lbrace \eta > 0\ |\ f(\eta)=0\right\rbrace \subseteq [0, \delta_0],\qquad
\textrm{where}\qquad
\delta_0=\frac{1}{2c}\left( \sqrt{1+8c}+1\right).
$$
Calculating the interval for various energies we obtain
\begin{align*}
\textrm{For}\quad c &=\frac{1}{5},\quad & \delta_0 &=\frac{5}{2}\left( \sqrt{1+\frac{8}{5}}+1\right)=\frac{5}{2}\left( \sqrt{\frac{13}{5}}+1\right)<\frac{15}{2}=7\frac{1}{2},\\
\textrm{For}\quad c&=\frac{1}{10},& \delta_0 &=5\left( \sqrt{1+\frac{4}{5}}+1\right)=5\left( \frac{3}{\sqrt{5}}+1\right)< 5\left( \frac{3}{2}+1\right)=12\frac{1}{2},\\
\textrm{For}\quad c&=\frac{1}{20},& \delta_0 &=10\left( \sqrt{1+\frac{2}{5}}+1\right)=10\left( \sqrt{\frac{7}{5}}+1\right)<10\left(\frac{6}{5}+1\right)=22.
\end{align*}

Consequently, Figure \ref{fig:fFunction}. depicts the three functions $f$ corresponding to energies $c\in \left\lbrace \frac{1}{5}, \frac{1}{10}, \frac{1}{20}\right\rbrace$ on the whole interval $[0,\delta_0]$. Therefore we can deduce that the plot in Figure \ref{fig:fFunction}. depicts all the zeros of the three functions. Hence
$$
\#\Crit^+\A_{q_0,q_1}^{H_0-\frac{1}{5}}=1, \qquad \#\Crit^+\A_{q_0,q_1}^{H_0-\frac{1}{10}}=3\quad \textrm{and}\quad \#\Crit^+\A_{q_0,q_1}^{H_0-\frac{1}{20}}=5.
$$
From the discussion above we know that 
$$
\#\Crit^+\A_{q_0,q_1}^{H_0-1}=\#\Crit^+\A_{q_0,q_1}^{H_0-\frac{1}{2}}=\#\Crit^+\A_{q_0,q_1}^{H_0-\frac{1}{5}}=1.
$$
In other words, on each of the level sets corresponding to $c=1$, $c=\frac{1}{2}$ and $c=\frac{1}{5}$ there is exactly one Reeb chord from $q_0=(0,1)$ to $q_1=(1,0)$. These Reeb chords are presented in Figure \ref{fig:1Orbit3Energies}.

\begin{figure}
\begin{minipage}[c]{0.43\linewidth}
\includegraphics[width=\linewidth]{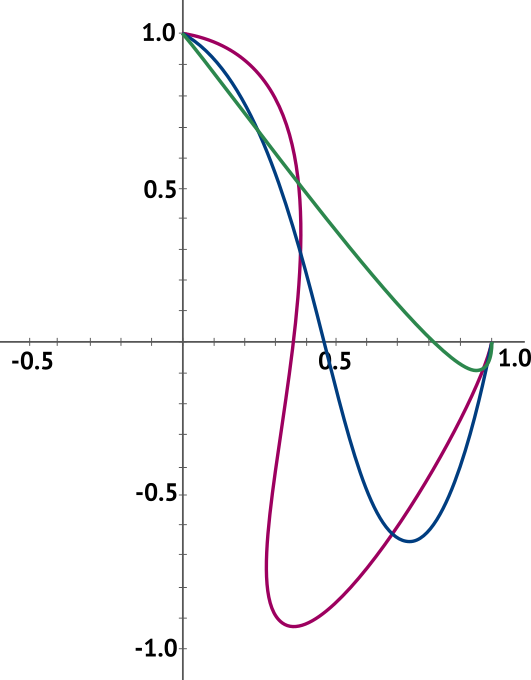}

\caption{The unique Reeb chords of energy $c=1$ (green), $c=\frac{1}{2}$ (blue), and $c=\frac{1}{5}$ (magenta).}
\label{fig:1Orbit3Energies}
\end{minipage}
\hfill
\begin{minipage}[c]{0.54\linewidth}
\vspace*{-.45cm}
\includegraphics[width=\linewidth]{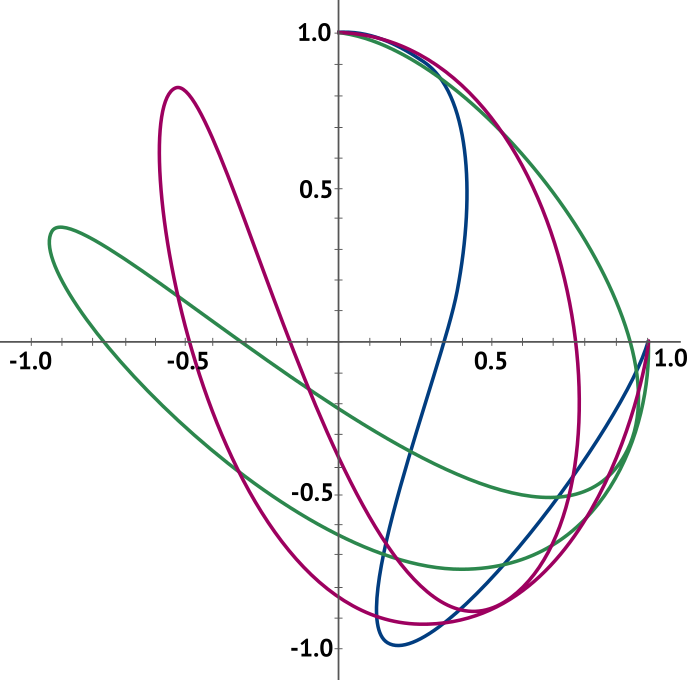}
\caption{The three Reeb chords of energy $c=\frac{1}{10}$.}
\label{fig:3Orbits}
\end{minipage}%
\end{figure}

Figure \ref{fig:3Orbits} depicts the three Reeb chords from $q_0=(0,1)$ to $q_1=(1,0)$ of energy $c=\frac{1}{10}$, and Figure \ref{fig:5Orbits} shows the five Reeb chords of energy $c=\frac{1}{20}$.

\begin{figure}
\centering
\includegraphics[width=.6\linewidth]{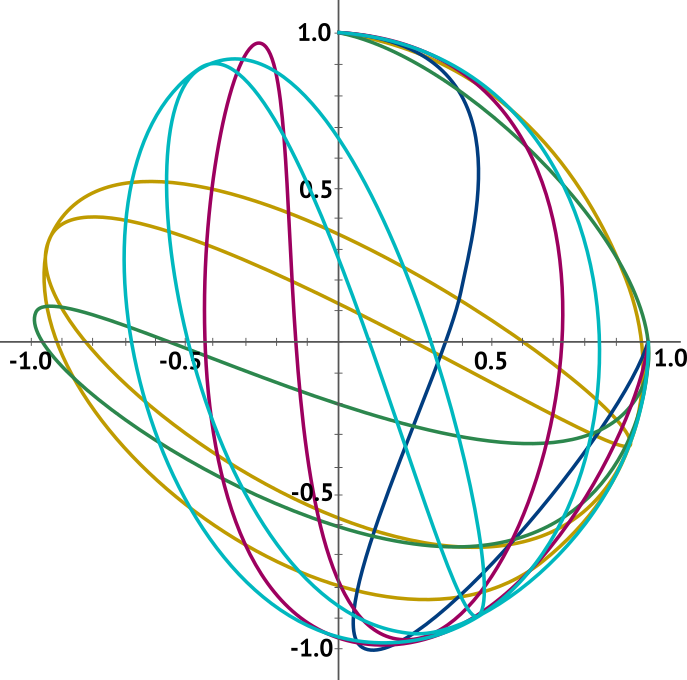}
\caption{The five Reeb chords of energy $c=\frac{1}{20}$.}
\label{fig:5Orbits}
\end{figure}
\subsection{Calculating the Maslov index}

In this subsection we will prove Theorem \ref{thm:posLRFH}. Let $H_0$ be the Copernican Hamiltonian as in \eqref{DefH0} and let $\mathcal{H}$ be the set of perturbations as defined in \eqref{DefHset}.
By Proposition \ref{prop:posLRFH} we know that for all $h\in \mathcal{H}$ and all pairs $q_0,q_1\in \R^2$, $q_0\neq q_1$ the positive Lagrangian Rabinowitz Floer homology of the triple $\LRFH_*^+(\A^{H_0-h}_{q_0,q_1})$ is well defined and isomorphic to $\LRFH_*^+(\A^{H_0-c}_{q_0,q_1})$ for any $c>0$. On the other hand, by Proposition \ref{prop:CritPM} we know that if $c\geq \frac{1}{2}|q_0||q_1|$ then $\Crit^+\A^{H_0-c}_{q_0,q_1}$ has only one element. Consequently, $\LRFH_*^+(\A^{H_0-c}_{q_0,q_1})$ has only one generator and its boundary operator is $0$. Therefore, in order to calculate the positive Lagrangian Rabinowitz Floer homology explicitly what is left to do is to calculate the Maslov index of $(v,\eta)\in \Crit^+\A^{H_0-c}_{q_0,q_1}$.

\begin{lem}\label{lem:ChordsBijection}
For $c>0$ define a Hamiltonian $H_c:T^*\R^2\to \R$ as
\begin{equation}\label{DefHDelta}
H_\delta(q,p):=\frac{|p|^2}{2}+\delta(p_1q_2-p_2q_1).
\end{equation}
Let $H_0$ be the Copernican Hamiltonian as in \eqref{DefH0} and let $\varphi_\delta$ be a diffeomorphism of $T^*\R^2$ defined $\varphi_\delta(q,p):=\left(\sqrt{\delta}q,\frac{1}{\sqrt{\delta}}p\right)$. Then
$$
(v,\eta)\in \Crit \A_{q_0,q_1}^{H_\delta-1}\qquad \iff \qquad \left(\varphi_\delta\circ v, \delta\eta\right)\in \Crit \A_{\varphi_\delta(q_0),\varphi_\delta(q_1)}^{H_0-\frac{1}{\delta}}.
$$
\end{lem}
\begin{proof}
First observe that $\varphi$ preserves the standard symplectic form, so it is in fact a symplectomorphism. Moreover, the Hamiltonians satisfy the following relation $H_\delta=\delta H_0\circ \varphi_\delta$, thus 
$$
H_0^{-1}\left(\frac{1}{\delta}\right)=\varphi_\delta (H_\delta^{-1}(1) )\qquad\textrm{and}\qquad X_{H_\delta}=\delta D\varphi_\delta^{-1}X_{H_0}.
$$
Now $(v,\eta)\in \Crit \A_{q_0,q_1}^{H_\delta-1}$ if and only if $v(i)=q_i$ for $i=1,2$, $v([0,1])\subseteq H_\delta^{-1}(1)$ and $\partial_t v=\eta X_{H_\delta}(v)$. Naturally, the first two conditions are trivially equivalent. It suffices to verify that
$$
\frac{d}{dt}(\varphi_\delta\circ v) = D\varphi_\delta (\partial_t v) = \eta D\varphi_\delta (X_{H_\delta}(v))=\delta X_{H_0}(\varphi_\delta \circ v).
$$
\end{proof}
\begin{lem}\label{lem:Convergence}
Let $\{H_\delta\}_{\delta>0}$ be the family of Hamiltonians defined in \eqref{DefHDelta}. Let $H_{\bullet}:T^*\R^2\to\R$ be the kinetic Hamiltonian $H_{\bullet}(q,p):=\frac{1}{2}|p|^2$.
Fix $q_0,q_1\in \R^2$, $q_0\neq q_1$. Then for $\delta<\sqrt{\frac{2}{|q_0||q_1|}}$ (in case $q_0=0$ or $q_1=0$ for any $\delta>0$) we have $\#\Crit^+\A^{H_\delta-1}_{q_0,q_1}=1$ and the family of Reeb chords $(v_\delta,\eta_\delta)\in \Crit^+\A^{H_\delta-1}_{q_0,q_1}$ satisfies
$$
\lim_{\delta\to 0}(v_\delta,\eta_\delta)=(v_0,\eta_0)\in \Crit^+\A^{H_{\bullet}-1}_{q_0,q_1}.
$$ 
\end{lem}

\begin{proof}
Using Lemma \ref{lem:ChordsBijection} we know that for every $\delta>0$ there is a bijection between $\Crit^+\A_{q_0,q_1}^{H_\delta-1}$ and $\Crit^+ \A_{\varphi_\delta(q_0),\varphi_\delta(q_1)}^{H_0-\frac{1}{\delta}}$. On the other hand, by Proposition \ref{prop:CritPM} we know that for $|\varphi_\delta(q_0)||\varphi_\delta(q_1)|=\delta|q_0||q_1|\leq \frac{2}{\delta}$ we have $1=\# \Crit^+ \A_{\varphi_\delta(q_0),\varphi_\delta(q_1)}^{H_0-\frac{1}{\delta}}=\# \Crit^+ \A_{q_0,q_1}^{H_\delta-1}$. Consequently, for every $\delta<\sqrt{\frac{2}{|q_0||q_1|}}$ (or in the case either $q_0$ or $q_1$ are zero for any $\delta>0$) there is a unique solution $(v_\delta,\eta_\delta)\in  \Crit^+ \A_{q_0,q_1}^{H_\delta-1}$.

Having established that $(v_\delta,\eta_\delta)\in \Crit^+\A^{H_\delta-1}_{q_0,q_1}$ is uniquely defined for small enough $\delta$, we can use the function $f$ from Lemma \ref{lem:bijection} to estimate $\eta_\delta$. We simply need to map $(c,q_0,q_1,\eta)\mapsto (\frac{1}{\delta}, \sqrt{\delta}q_0, \sqrt{\delta}q_1, \delta \eta)$ to obtain that $\eta_\delta$ is the unique positive root of the function:
\begin{align}
g_\delta(\eta)&:=  -\eta^2+\delta \eta q_1^TR\left(\delta\eta+\frac{\pi}{2}\right)q_0 +\frac{1}{2}\left(|q_0|^2+|q_1|^2\right)-q_1^TR(\delta\eta)q_0\nonumber\\
&= \frac{1}{2}|q_1-q_0|^2-\eta^2+q_1^T\left( \Id-R(\delta\eta)-\delta \eta R'(\delta \eta) \right)q_0\label{g_delta}
\end{align}
This way we know that $(v_\delta,\eta_\delta)\in \Crit^+\A^{H_\delta-1}_{q_0,q_1}$ if and only if $g_\delta(\eta_\delta)=0$.

By \eqref{g_delta} we can use the Taylor expansion to obtain the following estimate
\begin{gather*}
\frac{1}{2}|q_1-q_0|^2-\eta^2(1+\delta^2|q_0||q_1|)\leq g_\delta(\eta) \leq \frac{1}{2}|q_1-q_0|^2-\eta^2(1-\delta^2|q_0||q_1|),\\
g_\delta(\eta_\delta)=0\quad \Rightarrow \quad \frac{|q_1-q_0|}{\sqrt{2(1+\delta^2|q_0||q_1|)}}\leq \eta_\delta \leq\frac{|q_1-q_0|}{\sqrt{2(1-\delta^2|q_0||q_1|)}},
\end{gather*}
which directly gives us $\lim_{\delta\to 0}\eta_\delta=\frac{1}{\sqrt{2}}|q_1-q_0|$.

To show the convergence of $v_\delta$ we again use Lemma \ref{lem:bijection} to present the equation for $v_\delta$ explicitly. By Lemma \ref{lem:ChordsBijection} is suffices to just use the mapping $(q_0,q_1,\eta)\mapsto (\sqrt{\delta}q_0, \sqrt{\delta}q_1, \delta \eta)$ in the formula for $v_\delta$ from Lemma \ref{lem:bijection} to obtain:
\begin{align*}
\varphi_\delta\circ v_\delta(t) & =\left(\begin{array}{c c}
(1-t) R(t\delta\eta_\delta) & t R(\delta\eta_\delta(t-1))\\
 -\frac{1}{\delta\eta_\delta}R(t\delta\eta_\delta) & \frac{1}{\delta\eta_\delta} R(\delta\eta_\delta(t-1))
\end{array}\right)
\left(\begin{array}{c}
 \sqrt{\delta} q_0 \\  \sqrt{\delta} q_1
\end{array}\right),\\
v_\delta(t) &= \left(\begin{array}{c c}
(1-t) R(t\delta\eta_\delta) & t R(\delta\eta_\delta(t-1))\\
 -\frac{1}{\eta_\delta}R(t\delta\eta_\delta) & \frac{1}{\eta_\delta} R(\delta\eta_\delta(t-1))
\end{array}\right)
\left(\begin{array}{c}
q_0 \\ q_1
\end{array}\right).
\end{align*}
Consider the functions $t \mapsto R(t\delta\eta_\delta)$ and $t\mapsto R((1-t)\delta\eta_\delta)$ on the interval $[0,1]$. Since $\lim_{\delta\to 0}\delta\eta_\delta =0$, we have the uniform convergence $\lim_{\delta\to 0}R(t\delta\eta_\delta)=\lim_{\delta\to 0}R((1-t)\delta\eta_\delta)=\Id$. Consequently,
\begin{equation}\label{straight}
\lim_{\delta\to 0}v_\delta=v_0(t) \quad \textrm{with}\quad v_0(t):=  \left(\begin{array}{c c}
(1-t)  & t \\
 -\frac{\sqrt{2}}{|q_1-q_0|} & \frac{\sqrt{2}}{|q_1-q_0|} 
\end{array}\right)
\left(\begin{array}{c}
q_0 \\ q_1 \end{array}\right).
\end{equation}
A straightforward calculation shows that $(v_0,\eta_0)$ with $\eta_0:=\frac{1}{\sqrt{2}}|q_1-q_0|$ is the unique element of $\Crit^+\A^{H_{\bullet}-1}_{q_0,q_1}$.
\end{proof}

\vspace*{.25cm}
\noindent \textit{Proof of Theorem \ref{thm:posLRFH}:}
By Proposition \ref{prop:posLRFH} we know that for all $h\in \mathcal{H}$ and all $q_0,q_1\in \R^2$ with $q_0\neq q_1$ the positive Lagrangian Rabinowitz Floer homology $\LRFH_*^+(\A^{H_0-h}_{q_0,q_1})$ is well defined and isomorphic to $\LRFH_*^+(\A^{H_0-c}_{q_0,q_1})$ for any $c>0$. On the other hand, by Proposition \ref{prop:CritPM} we know that if $c\geq |q_0||q_1|/2$ then $\Crit^+\A^{H_0-c}_{q_0,q_1}$ has only one element. Consequently, $\LRFH_*^+(\A^{H_0-c}_{q_0,q_1})$ has only one generator and its boundary operator is $0$. It remains to calculate the Maslov index of this generator. 

Let $\{H_\delta\}_{\delta>0}$ be the family of Hamiltonians defined in \eqref{DefHDelta} and let $\varphi_\delta$ be the family of symplectomorphisms defined by $\varphi_\delta(q,p):=(\frac{1}{\sqrt{\delta}}q,\sqrt{\delta}p)$.
By Lemma \ref{lem:ChordsBijection}, for every $c>0$ we have
$$
(v,\eta)\in \Crit^+\A^{H_0-c}_{q_0,q_1}\qquad\iff\qquad(\varphi_\frac{1}{c}^{-1} \circ v, c\eta)\in \Crit^+\A^{H_{1/c}-1}_{\frac{1}{\sqrt{c}}q_0, \frac{1}{\sqrt{c}}q_1}.
$$
Since the Maslov index is invariant under symplectomorphisms, we get 
$$
\mu^{\rm tr}\left( v\left(\frac{\cdot}{\eta}\right)\right)=\mu^{\rm tr}\left(\varphi_\frac{1}{c}^{-1} \circ v\left(\frac{\cdot}{c\eta}\right)\right).
$$
Therefore, it suffices to calculate $\mu^{\rm tr}\left( v_\delta\left(\frac{\cdot}{\eta_\delta}\right)\right)$ of the unique $(v_\delta, \eta_\delta)\in \Crit^+\A^{H_\delta-1}_{q_0,q_1}$ for $\delta>0$ small enough.

Denote by $H_{\bullet}$ the kinetic Hamiltonian $H_{\bullet}(q,p):=\frac{1}{2}|p|^2$ and let $(v_0,\eta_0)$ be the unique element of $\Crit^+\A^{H_{\bullet}-1}_{q_0,q_1}$ as in \eqref{straight}.
By Lemma \ref{lem:Convergence} we know that $\lim_{\delta\to 0}(v_\delta, \eta_\delta)=(v_0,\eta_0)$. Continuity of the Maslov index and Example~\ref{ex:Maslov} implies  
$$
\mu^{\rm tr}\left( v_\delta\left(\frac{\cdot}{\eta_\delta}\right)\right)=\mu^{\rm tr}\left( v_0\left(\frac{\cdot}{\eta_0}\right)\right) = \frac12
$$
for small enough $\delta>0$. This concludes the proof of Theorem \ref{thm:posLRFH}. 
\hfill $\square$

\section{Non-compactly supported potential}

The aim of this section will be to prove Proposition \ref{prop:CritH=CritH1}. Let $H_0: T^*\R^2 \to \R$ be the quadratic Hamiltonian defined in \eqref{DefH0} and let $V:\R^2 \to \R$ be the potential function as in Proposition \ref{prop:CritH=CritH1}.  
\begin{equation}\label{DefH}
\textrm{Define}\qquad H:= H_0 -V.
\end{equation}
Then
\begin{equation}\label{X_H}
\begin{aligned}
X_H & = (p_1+q_2)\partial_{q_1}+(p_2-q_1)\partial_{q_2}+\left(p_2 +\frac{\partial V}{\partial_{q_1}}\right)\partial_{p_1}-\left(p_1-\frac{\partial V}{\partial_{q_2}}\right)\partial_{p_2},\\
X_H & = p_r\partial_r+\left(1+\frac{p_\theta}{r^2}\right)\partial_\theta+\left( \frac{p_\theta^2}{r^3}+ \frac{\partial V}{\partial r}\right)\partial_{p_r}+\frac{\partial V}{\partial \theta}\partial_{p_\theta}.
\end{aligned}
\end{equation}
This first step to prove Proposition \ref{prop:CritH=CritH1} is to show that for positive energy $c>0$ all periodic orbits of $X_H$ on $H^{-1}(c)$ are contained in a compact subset of $T^*\R^2$:

\begin{prop}\label{prop:BoundChord}
Let $H: T^*\R^2 \to \R$ be the Hamiltonian defined in \eqref{DefH}. Fix $c>0$ and $q_0,q_1 \in \R^2$ and let $\A^{H-c}_{q_0,q_1}:\mathscr{H_1}_{q_0,q_1}\times \R \to \R$ be the corresponding action functional. Then the critical set $\Crit \A^{H-c}_{q_0,q_1}$ is bounded in $L^\infty$.
\end{prop}
\begin{proof}
By an argument similar to the one in the proof of Lemma \ref{lem:H0Chord} it suffices to show that the set
\begin{equation}\label{H1_r}
H^{-1}(c)\cap \{\{H, r\}=0\}\cap \{\{H,\{H, r\}\}\leq 0\}
\end{equation}
is compact. By \eqref{X_H} we have that
\begin{align*}
\{H,r\} & = p_r,\\
\{H,\{H,r\}\} & = \frac{p_\theta^2}{r^3}+ \frac{\partial V}{\partial r}.
\end{align*}
For $(r, \theta, p_r, p_\theta) \in H^{-1}(c)\cap  \{\{H, r\}=0\}$ we have
\begin{gather}
 \frac{p_\theta^2}{2r^2}+p_\theta =V(r,\theta)+c\geq c,\nonumber\\
 \left(\frac{p_\theta}{r}+r\right)^2 \geq r^2+2c,\nonumber\\
 \frac{|p_\theta|}{r}\geq \sqrt{r^2+2c}-r.\label{p_theta_r}
\end{gather}
Consequently, by our assumption on the potential function $V$, if $\alpha>2$ and $a>0$ then for $(r, \theta, p_r, p_\theta) \in H^{-1}(c)\cap  \{\{H, r\}=0\}$, such that 
\begin{gather*}
r \geq r_1:=\max\left\lbrace r_0,\frac{1}{2}\sqrt{c}, \left( \frac{a\alpha}{c^2}\right)^{\frac{1}{\alpha-2}}\right\rbrace\\ \textrm{we have} \quad
\frac{|p_\theta|}{r}  \geq \sqrt{r^2+2c}-r > \frac{c}{2 r},\\
\{H,\{H, r\}\} \geq \frac{c^2}{4r^3}-\frac{a\alpha}{r^{\alpha+1}}>0.
\end{gather*}

On the other hand, if $\alpha = 2$ and $a\in \left(0 ,\frac{c^2}{4}\right)$ then for $(r, \theta, p_r, p_\theta) \in H^{-1}(c)\cap  \{\{H, r\}=0\}$, such that
\begin{gather*}
r \geq r_1:=\max\left\lbrace r_0,\sqrt{\frac{a}{c-\sqrt{2a}}}\right\rbrace\\ \textrm{we have} \quad
\frac{|p_\theta|}{r}  \geq \sqrt{r^2+2c}-r > \frac{\sqrt{2a}}{ r},\\
\{H,\{H, r\}\} \geq \frac{2a}{r^3}-\frac{2a}{r^{3}}>0.
\end{gather*}
Consequently, in both cases the set
$$
 H^{-1}(c)\cap  \{\{H, r\}=0\}\cap \left\lbrace \{H,\{H, r\}\} \leq 0 \right\rbrace
$$
is bounded in $r$ and for every $(v, \eta) \in \Crit \A^{H-c}_{q_0,q_1}$ we have
$$
\sup r \circ v \leq R_0 := \max\{|q_0|, |q_1|, r_1\}.
$$
By an argument similar to the one presented in the Lemma \ref{lem:H0Chord} we obtain that $p_r\circ v$ and $p_\theta\circ v$ are also uniformly bounded in the following way:
\begin{align*}
|p_r| & \leq \sqrt{R_0^2+ 2(\sup_{r\leq R_0}V+c)},\\
|p_\theta| & \leq R_0 \left( R_0+\sqrt{R_0^2+ 2(\sup_{r\leq R_0}V+c)}\right).
\end{align*}
\end{proof}
\begin{rem}
Note that the assertions of Proposition \ref{prop:BoundChord} hold true also for the potential $V$ satisfying $V(r,\theta) \leq \frac{a}{r^2}$ and $\frac{\partial V}{\partial r} \geq -\frac{2 a}{r^3}$ for $r >r_0$ and $a \in \left[ \frac{c^2}{4}, \frac{c^2}{2} \right)$. However, we were unable to prove Proposition \ref{prop:CritH=CritH1} for this class of potentials and that is why we restrict ourselves to $a < \frac{c^2}{4}$.
\end{rem}

Now we would like change the Hamiltonian $H$, by multiplying the potential function $V$ with a compactly supported function $\varphi: T^*\R^2 \to [0,1]$, in the following way
\begin{equation}\label{DefH1}
H_1(q,p):=H_0(q,p)+\varphi(q,p)V(q).
\end{equation}
The associated Hamiltonian vector field is
\begin{equation}\label{X_H1}
\begin{aligned}
X_{H_1} & = \left(p_r-V\frac{\partial \varphi}{\partial p_r}\right)\partial_r+\left(1+\frac{p_\theta}{r^2}-V\frac{\partial \varphi}{\partial p_\theta}\right)\partial_\theta\\
& + \left( \frac{p_\theta^2}{r^3}+ \varphi\frac{\partial V}{\partial r}+V\frac{\partial \varphi}{\partial r}\right)\partial_{p_r}+\left(\varphi\frac{\partial V}{\partial \theta}+V\frac{\partial \varphi}{\partial \theta}\right)\partial_{p_\theta}.
\end{aligned}
\end{equation}

If we choose $\varphi$, such that for all the Reeb chords $(v,\eta) \in \Crit \A^{H-c}_{q_0,q_1}$ we would have $v([0,1]) \subseteq \varphi^{-1}(1)$, then $\Crit(\A^{H-c}_{q_0,q_1}) \subseteq \Crit(\A^{H_1-c}_{q_0,q_1})$. In fact, we would like to choose $\varphi$, such that $\Crit(\A^{H_1-c}_{q_0,q_1})= \Crit(\A^{H-c}_{q_0,q_1})$.

Let $\chi:\R\to [0,1]$ be a smooth function, such that $-2<\chi'< 0$ on $[0,1]$ and
\begin{align}
&&\chi(x) & =\begin{cases}
1 & \textrm{for} \quad x\leq 0,\\
0 & \textrm{for} \quad x\geq 1.
\end{cases}\nonumber\\
&\textrm{Define}
&\chi_0(r) &:=\chi\left(\beta(r-R_1) \right),\nonumber\\
&&\chi_1(r,\theta, p_r,p_\theta)& :=\chi( H_0(r,\theta,p_r,p_\theta)-\sup V-c),\nonumber\\
&&\varphi(r,\theta,p_r,p_\theta) & := \chi_0(r)\chi_1(r,\theta,p_r,p_\theta).\label{defPhi}
\end{align}
\begin{align}
&\textrm{where} & R_1 &:= \begin{cases}
\max\left\lbrace |q_0|, |q_1|, r_0, \frac{1}{2}\sqrt{c},\left( \frac{8a\alpha}{c^2}\right)^\frac{1}{\alpha-2}\right\rbrace & \textrm{for}\quad \alpha>2,\\
\max\left\lbrace |q_0|, |q_1|, r_0, \frac{c+2\sqrt{a}}{2\sqrt{c-2\sqrt{a}}}\right\rbrace & \textrm{for}\quad \alpha=2.
\end{cases}\label{R1}\\
&\textrm{and} & \beta &:=\begin{cases}
\frac{2-\alpha}{2 R_1}  & \textrm{for}\quad \alpha>2,\\
\frac{(c+2\sqrt{a})^2-16a}{8a R_1} & \textrm{for}\quad \alpha=2.
\end{cases}\label{beta}
\end{align}
Note that in both for $\alpha>2$ and for $\alpha=2$ we have $\beta>0$. In the later case, it follows from the assumption that $a< \frac{c^2}{4}$.

This way we have
\begin{equation}
\chi_0(r) = \begin{cases}
1 & \textrm{for}\ r\leq R_1,\\
0 & \textrm{for}\ r\geq R_1+\frac{1}{\beta}.
\end{cases}\label{chi0}
\end{equation}
We will show that for a function $\varphi$ defined as in \eqref{defPhi} we have 
$$
\Crit (\A^{H_0-\varphi V-c}_{q_0,q_1}) = \Crit(\A^{H_0-V-c}_{q_0,q_1}).
$$
The first step would be to show that $\varphi$ has compact support:

\begin{lem}
The function $\varphi: T^*\R^2\to \R$ defined as in \eqref{defPhi} has compact support.
\end{lem}
\begin{proof}
By \eqref{defPhi} and \eqref{chi0} we have
\begin{align*}
\operatorname{supp}\varphi & \subseteq \{H_0 \leq \sup V+c + 1\}\cap  \left\lbrace r \leq R_1+\frac{1}{\beta}\right\rbrace\\
& \subseteq  \left\lbrace \frac{1}{2}p_r^2+\frac{1}{2}\left( \frac{p_\theta}{r}+r\right)^2 \leq \frac{1}{2}r^2+\sup V+c + 1\right\rbrace\cap  \left\lbrace r \leq R_1+\frac{1}{\beta}\right\rbrace\\
& \subseteq  \left\lbrace \frac{1}{2}p_r^2+\frac{1}{2}\left( \frac{p_\theta}{r}+r\right)^2 \leq \frac{1}{2} \left( R_1+\frac{1}{\beta}\right)^2+\sup V + c+ 1\right\rbrace\cap \left\lbrace r \leq R_1+\frac{1}{\beta}\right\rbrace.
\end{align*}
Since both of the sets on the right-hand side are compact, their intersection is compact, so $\operatorname{supp}\varphi$ is compact as a closed subset of a compact set.
\end{proof}

\begin{lem}
Let $V:\R^2\to \R$ be a potential function as in Proposition \ref{prop:CritH=CritH1} and let $\varphi$ be the corresponding function defined in \eqref{defPhi}. Then $\varphi V+c \in \mathcal{H}$.
\end{lem}

\begin{proof}
First observe that
$$
\supp d(\varphi V)=\supp ( V d\varphi+ \varphi dV)\subseteq \supp \varphi,
$$
hence $d(\varphi V)\in C_c^\infty(T^*\R^2)$. On the other hand,
 by definition:
\begin{align*}
\varphi(p,q) V(q) & +c - d(\varphi V)(p\partial_p) = V(q)\chi_0(q)\left(\chi_1(q,p)-d\chi_1(p\partial_p)\right)+c\\
& = V(q) \chi_0(q) \left( \chi_1(q,p)- \chi'(H_0(q,p) - \sup V-c) dH_0(p\partial_p)\right)+c.
\end{align*}
Since by assumption we the functions $V$, $\chi_0$, $\chi_1$ and $-\chi'$ are non-negative it suffices to prove that $dH_0(p\partial_p)\geq 0$ on $\chi_1^{-1}((0,1))$. Note that on $\chi_1^{-1}((0,1))$ we have $H_0 > \sup V+c$.  Combining that with \eqref{H0pdp} we can calculate that on $\chi_1^{-1}((0,1))$ we have:
$$ 
dH_0(p\partial_p)= \frac{1}{2}\|p\|^2+H_0\geq \frac{1}{2}\|p\|^2 + \sup V+c > 0.
$$
\end{proof}

Now we have proven that $\varphi$ as in \eqref{defPhi} is an eligible candidate, we will continue with the proof of Proposition \ref{prop:CritH=CritH1}. However, we will first start with the proof of a series of lemmas:

\begin{lem}\label{lem:vinPhi(1)}
Let $H$ be the Hamiltonian defined in \eqref{DefH} and let $\varphi$ be the function defined in \eqref{defPhi}.
Then
 $$
v([0,1])\subseteq \varphi^{-1}(1)\quad \textrm{for all}\quad (v,\eta)\in \Crit\A^{H-c}_{q_0,q_1}.
$$
\end{lem}

\begin{proof}
By definition of $H$ we have
$$
H_ 0 = H+ V \leq H + \sup V,
$$
so $H^{-1}(c) \subseteq H_0^{-1}((-\infty, \sup V+c])$. On the other hand,
by Proposition \ref{prop:BoundChord} for every $(v,\eta)\in \Crit\A^{H-c}_{q_0,q_1}$ we have
\begin{align*}
v([0,1]) & \subseteq  \{ r \leq R_0\}\cap  H^{-1}(c) \subseteq \{r \leq R_1\}\cap  H_0^{-1}((-\infty, \sup V+c])\\
& = \chi_0^{-1}(1)\cap \chi_1^{-1}(1)=\varphi^{-1}(1).
\end{align*}
\end{proof}

\begin{lem}\label{lem:subsetvarphi(1)}
Let $H_1$ be the Hamiltonian defined in \eqref{DefH1} and let $\varphi$ be the function defined in \eqref{defPhi}. Then
$$
H_1^{-1}(c)\cap \{\{H_1,r\}=0\}\cap \{\{H_1,\{H_1,r\}\}\leq 0\} \subseteq \varphi^{-1}(1).
$$
\end{lem}

\begin{proof}
First observe that $H_1\big|_{\varphi^{-1}(0)}=H_0$, so by \eqref{supp_h} we have
$$
H_1^{-1}(c)\cap \{\{H_1,r\}=0\}\cap \{\{H_1,\{H_1,r\}\}\leq 0\}\cap \varphi^{-1}(0)=\emptyset.
$$
On the other hand, if $x \in \chi_1^{-1}([0,1))$, then $H_0(x)>\sup V+c$ and
$$
H_1(x)=H_0(x)-\varphi (x) V(x) \geq H_0(x) - \sup V>c.
$$
Hence $\chi_1^{-1}((0,1])\cap H_1^{-1}(c)=\emptyset$. Therefore we can restrict ourselves to the analysis of the set $\chi_0^{-1}((0,1))$, i.e. when $ R_1 < r  < R_1+\frac{1}{\beta}$.

Let us calculate
\begin{align*}
\frac{\partial\varphi}{\partial r} & = \beta\chi_0'\chi_1-\frac{p_\theta^2}{r^3}\chi_0\chi_1', && \frac{\partial \varphi}{\partial \theta}=0,\\
\frac{\partial \varphi}{\partial p_r} & = p_r\chi_0\chi_1', && \frac{\partial\varphi}{\partial p_\theta} = \left(\frac{p_\theta}{r^2}+1\right)\chi_0 \chi_1'.
\end{align*}
In particular, by \eqref{X_H} we have
$$
\{H_1,r\} = dr(X_{H_1})= p_r- V\frac{\partial \varphi}{\partial p_r}=p_r (1- V \chi_0 \chi_1').
$$
By assumption $V, \chi \geq 0$ and $\chi'\leq 0$, thus $\{H_1,r\}=0$ implies $p_r=0$.

Consider now $(r, \theta, p_r, p_\theta) \in H_1^{-1}(0)\cap \left(\{H_1, r\}\right)^{-1}(0)$. Then
\begin{align*}
\{H_1,\{H_1,r\}\} & = \{H_1, p_r\}- \frac{\partial \varphi}{\partial p_r}\{H_1, V\}-V\left\lbrace H_1, \frac{\partial \varphi}{\partial p_r}\right\rbrace\\
& = \{H_1, p_r\}-V \frac{\partial^2\varphi}{\partial p_r^2} \{H_1, p_r\}\\
& =  \left(\frac{p_\theta^2}{r^3}+ \varphi\frac{\partial V}{\partial r}+V\frac{\partial \varphi}{\partial r}\right)\left(1-V\frac{\partial^2\varphi}{\partial p_r^2}\right)\\
& =  \left(\frac{p_\theta^2}{r^3}+ \chi_0 \chi_1\frac{\partial V}{\partial r}+V\left( \beta\chi_0'\chi_1-\frac{p_\theta^2}{r^3}\chi_0\chi_1'\right)\right)\left(1-\chi_0\chi_1'V\right)
\end{align*}
Again, since by assumption $V, \chi \geq 0$ and $\chi'\leq 0$, thus $1-\chi_0\chi_1'V\geq 1$ and $\{H_1,\{H_1,r\}\}$ has the same sign as 
\begin{equation}\label{eq5}
\frac{p_\theta^2}{r^3}+ \chi_0 \chi_1\frac{\partial V}{\partial r}+V\left( \beta\chi_0'\chi_1-\frac{p_\theta^2}{r^3}\chi_0\chi_1'\right)
\end{equation}
Furthermore, since by assumption $V, \chi_0, p_\theta \geq 0$ and $\chi_1' \leq 0$, consequently $-\frac{p_\theta^2}{r^3}\chi_0\chi_1'V \geq 0$ and we can estimate \eqref{eq5} from below by
\begin{equation}\label{eq6}
\frac{p_\theta^2}{r^3}+ \chi_0 \chi_1\frac{\partial V}{\partial r}+ V\beta \chi_0'\chi_1.
\end{equation}
We will show that \eqref{eq6} is positive for both cases $\alpha>2$ and $\alpha=2$.

\noindent\textbf{Case $\alpha>2$:}\\
By \eqref{R1} we have that $r > R_1 \geq \left(\frac{8a\alpha}{c^2}\right)^{\frac{1}{\alpha-2}}$, so $\frac{c^2}{4r^3}>\frac{2a\alpha}{r^{\alpha+1}}$. On the other hand, $r < R_1+\frac{1}{\beta}=R_1\frac{\alpha}{\alpha-2}$. Combining that with the assumptions on the potential $V$ and the fact that $-2\leq \chi'\leq 0$ and $0\leq \chi \leq 1$ we obtain 
\begin{align*}
\frac{p_\theta^2}{r^3}+ \chi_0 \chi_1\frac{\partial V}{\partial r}+\frac{\alpha-2}{2R_1}V \chi_0'\chi_1  & \geq \frac{c^2}{4r^3}-\frac{a\alpha}{r^{\alpha+1}}-\frac{\alpha-2}{R_1}\frac{a}{ r^\alpha}\\
 &> \frac{a\alpha}{r^{\alpha+1}}-\frac{\alpha -2}{R_1}\frac{a}{ r^\alpha}\\
 & = \frac{a(\alpha -2)}{R_1 r^{\alpha+1}}\left( R_1 \frac{\alpha}{\alpha -2} -r\right)>0.
\end{align*}

\noindent\textbf{Case $\alpha=2$:}\\
By \eqref{p_theta_r}, \eqref{R1} and \eqref{chi0} we have
\begin{gather*}
r  \geq \frac{c+2\sqrt{a}}{2\sqrt{c-2\sqrt{a}}}\qquad\textrm{and}\qquad c-2\sqrt{a} \geq \frac{(c+2\sqrt{a})^2}{4r^2},\\
r^2 + 2c \geq r^2+c+2\sqrt{a}+\frac{(c+2\sqrt{a})^2}{4r^2}=\left(r+ \frac{c+2\sqrt{a}}{2 r}\right)^2,\\
\frac{|p_\theta|}{r} \geq \sqrt{r^2+2c}-r\geq \frac{c+2\sqrt{a}}{2r}.
\end{gather*}
Therefore, by \eqref{beta} and \eqref{chi0} we obtain
\begin{align*}
\frac{p_\theta^2}{r^3} & + \chi_0 \chi_1\frac{\partial V}{\partial r}+V\beta \chi_0'\chi_1  \geq \frac{(c+2\sqrt{a})^2}{4r^3}-\frac{2a}{r^3}-\frac{2\beta a}{r^2}\\
& = \frac{2a}{r^3}\left( \frac{\left(c+2\sqrt{a}\right)^2}{8a}-1-\beta r\right)=\frac{2a}{r^3}\left( R_1 \beta + 1 -\beta r\right)>0.
\end{align*}

Consequently, in both cases we have
$$
H_1^{-1}(c)\cap \{\{H_1,r\}=0\}\cap \{\{H_1,\{H_1,r\}\}\leq 0\}\cap \chi_0^{-1}((0,1))=\emptyset,
$$
which proves the claim.
\end{proof}

\noindent\textit{Proof of Proposition \ref{prop:CritH=CritH1}:}
By Lemma \ref{lem:vinPhi(1)} we know that
$$
v([0,1])\subseteq \varphi^{-1}(1)\quad \textrm{for all}\quad (v,\eta)\in \Crit\A^{H-c}_{q_0,q_1}.
$$
But by definition $H_1\big|_{\varphi^{-1}(1)}=H\big|_{\varphi^{-1}(1)}$. Thus $\Crit\A^{H-c}_{q_0,q_1} \subseteq \Crit\A^{H_1-c}_{q_0,q_1}$.

On the other hand, by an argument as in the proof of Proposition \ref{prop:BoundChord} we know that for all $(v,\eta)\in \Crit\A^{H_1-c}_{q_0,q_1}$ if $\max_{t\in [0,1]}r \circ v = r\circ v(t_0)$, then
$$
v(t_0) \in \{q_0, q_1\}\cup \left( H_1^{-1}(c)\cap \{\{H_1,r\}=0\}\cap \{\{H_1,\{H_1,r\}\}\leq 0\}\right).
$$
Moreover, by Lemma \ref{lem:subsetvarphi(1)} we know that
\begin{align*}
v(t_0) \in \{q_0, q_1\} & \cup \left( H_1^{-1}(c)\cap \{\{H_1,r\}=0\}\cap \{\{H_1,\{H_1,r\}\}\leq 0\}\right)\\
& \subseteq H_1^{-1}(c)\cap \varphi^{-1}(1) \subseteq H^{-1}(c)\cap \{ r \leq R_1\}.
\end{align*}
Consequently, for all $(v,\eta)\in \Crit\A^{H_1-c}_{q_0,q_1}$ we have $\max r\circ v \leq R_1$. In other words, for all $(v,\eta)\in \Crit\A^{H_1-c}_{q_0,q_1}$ we have
$$
v([0,1]) \subseteq H_1^{-1}(c)\cap \{ r \leq R_1\}.
$$
On the other hand, the Hamiltonians satisfy 
$$
H_1^{-1}(c)=\{H_0 = \varphi V+c\}\subseteq H_0^{-1}((-\infty, \sup V+c]).$$
Consequently, for all $(v,\eta)\in \Crit\A^{H_1-c}_{q_0,q_1}$ we have
$$
v([0,1]) \subseteq H_1^{-1}(c)\cap \{ r \leq R_1\} \subseteq H_0^{-1}((-\infty, \sup V+c])\cap \{ r \leq R_1\} = \varphi^{-1}(1).
$$
But on $\varphi^{-1}(1)$ the two Hamiltonians $H_1$ and $H$ coincide. Therefore, $\Crit \A^{H-c}_{q_0,q_1} = \Crit \A^{H_1-c}_{q_0,q_1}$.

\hfill $\square$

\printbibliography

\end{document}